\documentclass[a4paper,11pt]{article}
\usepackage{amsmath,amsthm,amssymb}
\usepackage[usenames, dvipsnames]{color}
\usepackage{enumerate}

\numberwithin{equation}{section}

\newcommand{\ot}{\otimes}
\newcommand{\id}{\mathord{\operatorname{id}}}
\newcommand{\Z}{\mathbb{Z}}
\newcommand{\R}{\mathbb{R}}
\newcommand{\N}{\mathbb{N}}
\newcommand{\C}{\mathbb{C}}
\newcommand{\F}{\mathbb{F}}

\newcommand{\PolG}{{\rm Pol}(G)}
\newcommand{\Gq}{\mathbb{G}}

\theoremstyle{plain}
\newtheorem{theorem}{Theorem}[section]

\newtheorem{lemma}[theorem]{Lemma}
\newtheorem{proposition}[theorem]{Proposition}
\newtheorem{corollary}[theorem]{Corollary}
\theoremstyle{definition}
\newtheorem{definition}[theorem]{Definition}
\newtheorem{example}[theorem]{Example}

\newtheorem{remark}[theorem]{Remark}

\allowdisplaybreaks[4]

\begin{document}
%%%%%%%%%%%%%%%%%%%%%%%%%%%%%%%%%%%%%%%%%%%%%%%%%%%%%%%%%%%%%%%%%%%%%%%%%%%%%%%%%%%%%%%%%%

\begin{center}
{\LARGE\bf  On Compact Bicrossed Products}
\bigskip

{\sc Pierre Fima\footnote{Supported by the ANR grants NEUMANN and OSQPI}, Kunal Mukherjee and Issan Patri}
\end{center}

\begin{abstract}
\noindent We make a comprehensive and self-contained study of compact bicrossed products arising from matched pairs of discrete groups and  compact groups. We exhibit an automatic regularity property of such a matched pair and produce an easy construction of the associated bicrossed product $\Gq$. We investigate the relative co-property $(T)$ and the relative co-Haagerup property of the pair comprising of the compact group and the bicrossed product, discuss property $(T)$ and  Haagerup property of the discrete dual $\widehat{\Gq}$, and review co-amenability of $\Gq$ as well. We distinguish two such non-trivial compact bicrossed products with relative co-property $(T)$ and also provide an infinite family of pairwise non isomorphic non-trivial discrete quantum groups with property $(T)$, the existence of even one of the latter was unknown. Finally, we examine all the properties mentioned above for the crossed product quantum group given by an action by quantum automorphisms of a discrete group on a compact quantum group, and also establish the permanence of rapid decay and weak amenability and provide several explicit examples.
\end{abstract}

\section{Introduction}
In the eighties, Woronowicz \cite{Wo87,Wo88,Wo95} introduced the notion of compact quantum groups and generalized the classical Peter-Weyl representation theory. However, the theory of quantum groups goes back to Kac \cite{Ka63, Ka65} in the early sixties, and his notion of ring groups in modern terms are known as finite dimensional Kac algebras. In the fundamental work \cite{Ka68} on extensions of finite groups, Kac introduced the notion of \textit{matched pair} of finite groups and developed the bicrossed product construction giving the first examples of semisimple Hopf algebras that are neither commutative nor cocommutative. It was later generalized by Baaj and Skandalis \cite{BS93} in the context of Kac algebras and then by Vaes and Vainerman \cite{VV03} in the framework of locally compact \footnote{All l.c. spaces are assumed to be Hausdorff.} (l.c. in the sequel) quantum groups; the latter was introduced by Kustermans and Vaes in \cite{KV00}. In the classical case, i.e., in the ambience of groups,  Baaj and Skandalis concentrated only on the case of \textit{regular} matched pairs of l.c. groups.  In \cite{VV03}, the authors extended the study to \textit{semi-regular} matched pairs of l.c. groups. The case of a general matched pair of locally compact groups was settled by Baaj, Skandalis and Vaes in \cite{BSV03}. 

As a standing assumption, all throughout the paper, all Hilbert spaces and all C*-algebras are separable, all von Neumann algebras have separable preduals, all discrete groups are countable and all compact groups are Hausdorff and second countable.

The theory of quantum groups is fathomless. In order to have deeper insights, it is necessary to generate and study many explicit examples. The bicrossed product construction is a way to get abundant non-trivial examples of quantum groups which are very far from groups \cite{Fi07}. A \textit{compact bicrossed product} is one, in which the resulting quantum group is compact. Without being bogged technically, the bicrossed product construction in the classical case associates a l.c. quantum group to a matched pair of l.c. groups $(G_1, G_2)$. The associated l.c. quantum group $($in the bicrossed product construction$)$ has a Haar state, i.e., is a compact quantum group, if and only if $G_1$ is discrete and $G_2$ is compact \cite{VV03}. In this paper, such a pair will be called as a \textit{compact} matched pair. Moving to the quantum case, 
one can introduce the notion of matched pair of l.c. quantum groups, and perform an analogous bicrossed product construction to manufacture a l.c. quantum group that generalizes the classical bicrossed product construction. This construction is quite technical and we refer the interested reader to \cite{VV03} for details. It is to be noted that, in the same vein, the crossed product of a compact quantum group $G$ by a countable discrete group $\Gamma$ acting on $G$ by quantum automorphisms (see Section 2.2 for precise definition), as considered by Wang \cite{Wa95b}, is subsumed in the quantum bicrossed product construction and hence is a  simple case of compact bicrossed product. Needless to say, that the aforesaid class of bicrossed products in the quantum setup, does not exhaust the entire class of compact bicrossed products. We point out though, that despite the intricacy in the bicrossed product construction, the `compactness' of the matched pair in the classical case (for groups) alleviates technical obstacles, which is the primary sagacity of this paper. 

%we highlight in this paper. 

This paper investigates compact bicrossed products in both classical and quantum setting and studies their approximation properties, namely,  amenability, $K$-amenability, weak amenability, (relative) Haagerup property, (relative) property (T) and rapid decay, which enables one to manufacture explicit examples.  The paper has two major parts: one dealing with the classical case and one dealing with the quantum case. In the quantum case, we only concentrate on compact crossed products.

We provide a totally self-contained and direct approach dedicated towards the construction of a compact bicrossed product arising from matched pair of groups $(\Gamma,G)$, where $\Gamma$ is discrete and $G$ is compact. An advantage with our construction is that  it avoids technical intricacies that are obligatory when dealing with l.c. groups.  In the process, we observe that the compactness of $G$ constrains the matched pair to automatically satisfy a regularity property, notably, $\Gamma,G$ are subgroups of a l.c. group $H$ such that $\Gamma G =H$ and the canonical action of either group on its complementary pair is continuous. Moreover, the action of $\Gamma$ on $G$ happens via measure preserving homeomorphisms.  This regularity is not automatic in the l.c. setting and one has to compensate with `almost everywhere statements'. The aforesaid regularity galvanizes one to directly perform the bicrossed construction; the bicrossed product is of course known to be a Kac algebra. The continuous action of the group $G$ on the countable \textit{set} $\Gamma$ yields magic unitaries, which along with the irreducible unitary  representations of $G$ and the action of $\Gamma$ on $G$ by measure preserving homeomorphisms assist us in constructing the bicrossed product in an elegant fashion (Theorems \ref{ThmBicrossed}). Some easy consequences on amenability (which is known from \cite{DQV02}), $K$-amenability, Haagerup property are also presented in Corollary \ref{CorMatched}. We also compute the intrinsic group and the spectrum of the full C*-algebra of the associated bicrossed product in terms of the fixed points of the canonical actions and the spectrum of the groups in Proposition \ref{PropInt}. Needless to say, these are isomorphism invariants for compact quantum groups. 
 
With the help of the construction above, we explore the approximation properties of the dual of a compact bicrossed product arising from a compact matched pair of groups. We characterize the relative co-property $(T)$ for the pair $(G,\Gq)$, where $\Gq$ is the bicrossed product of the compact matched pair $(\Gamma, G)$, in terms of the action of $\Gamma$ on $G$. More precisely,  the negation of relative co-property $(T)$ for the pair $(G,\Gq)$ amounts to the existence of an asymptotically $\Gamma$-invariant sequence of Borel probability measures on $G$ each of which assign zero weight to the identity $e$ of the group but yet converge to the Dirac measure $\delta_e$ in the weak* topology (Theorem \ref{ThmRelT}).  In the event of existence of such a sequence of measures on $G$, by a standard result in measure theory (due to Parthasarathy and Steernman \cite{PS85}), the measures in the sequence versus their push forwards with respect to the group action implemented by $\Gamma$ have large common support. Thus, along the way, we show that such a sequence of measures can be replaced by one for which the $\Gamma$-action on $G$ is nonsingular. This result generalizes the classical characterization of the relative property $(T)$ for the pair $(H_0,\Gamma_0\ltimes H_0)$ (originally defined in \cite{Ma82}), where $\Gamma_0$ is a countable discrete group acting on a countable discrete abelian group $H_0$ \cite{CT11}.  We show that if the dual $\widehat{\Gq}$ of the bicrossed product has property $(T)$, then $\Gamma$ necessarily has property $(T)$ and the set of fixed points in $G$ of the action of $\Gamma$ on $G$ is finite. We also show that if $\Gamma$ has $(T)$ and $G$ is finite then $\widehat{\Gq}$ has $(T)$ and the converse holds when the action of $\Gamma$ on $G$ is compact (Theorem \ref{ThmPropT}). 

Proceeding further, we characterize the relative co-Haagerup property for the pair $(G,\Gq)$ again in terms of the action of $\Gamma$ on $G$. Like before, we prove that this property is equivalent to the existence of an approximately $\Gamma$-invaraint sequence of probability measures on $G$ which converge in weak* topology to $\delta_e$ and whose Fourier transform (regarded as an element of the multiplier algebra of $C^*_r(G))$ fall in $C^*_r(G)$ (Theorem \ref{ThmRelH}), and, like before, we show that the measures can be chosen such that the action of $\Gamma$ on $G$ is nonsingular. Again, this result generalizes the classical characterization of the relative Haagerup property for the pair $(H_0,\Gamma_0\ltimes H_0)$, where $\Gamma_0$ is a countable discrete group acting on the countable discrete abelian group $H_0$ \cite{CT11}. 

In the quantum setting, an example of a matched pair of a classical countable discrete  group with a compact quantum group is the pair arising in a crossed product in which the discrete group acts on the compact quantum group by quantum automorphisms \cite{Wa95b}. Since one of the involved actions is trivial, the representation theory is easier to study. But as the compact quantum group need not be commutative, Kac or co-amenable,  approximation properties become harder to exhibit. We provide a self-contained and very short approach to this construction and study all the properties mentioned above for the associated  crossed product quantum group. Let $\alpha\,:\,\Gamma\curvearrowright G$ be an action of the discrete group $\Gamma$ on the compact quantum group $G$ by quantum automorphisms and $\Gq$ be the crossed product quantum group. In this context, we first provide a short account of the quantum group structure of $\Gq$ and its representation theory which was initially studied by Wang in \cite{Wa95b} (but, in contrast with the work of Wang, we do not invoke free products).  We deduce obvious consequences on amenability and $K$-amenability of  $\widehat\Gq$ in Corollary \ref{Amenablestuff} and describe the intrinsic group and the spectrum of the full C*-algebra of $\Gq$ in Proposition \ref{Prop-CrossedInt}.

In the quantum setting, we study weak amenablity of $\widehat{\Gq}$. In \cite{KR99}, it was proved that when $G$ is Kac, then $\Lambda_{cb}(\widehat{G})=\Lambda_{cb}(C(G))=\Lambda_{cb}(\rm{L}^{\infty}(G))$. In our setup we estimate (in Theorem \ref{ThmWA}) the Cowling-Haagerup constants under compactness. When the action $\alpha$ of $\Gamma$ is compact we show that $\Lambda_{cb}(C(\Gq))\leq\Lambda_{cb}(\Gamma)\Lambda_{cb}(\widehat{G})$. 

Rapidly decreasing functions on group C*-algebras were first studied by Jolissaint in \cite{Jo90}. Generalizing this notion, 
rapid decay $($(RD) in the sequel$)$ for quantum groups was studied in \cite{Ve07} and subsequently this notion was calibrated in \cite{BVZ14}. Following \cite{BVZ14}, we show the permanence of (RD) in the setup of crossed products. To be precise, we show that if $\Gamma$ acts on $G$ via quantum automorphisms, there is a length function $l$ on $\rm{Irr}(G)$ which is invariant with respect to the canonical action of $\Gamma$ on $\rm{Irr}(G)$ such that $(\widehat{G},l)$ has (RD), and $\Gamma$ has (RD), then $\widehat{\Gq}$ has (RD) with respect to a  pertinent length function on $\rm{Irr}(\Gq)$ (Theorem \ref{ThmRD}).

Our characterization of the relative co-property $(T)$ for the pair $(G,\Gq)$ is analogous to the classical bicrossed product case: the approximating measures and $\delta_e$ in the characterization of the classical case are replaced in the quantum setting respectively by states on $C_m(G)$ and the counit of $G$. This proof is technically more involved than the classical case (Theorem \ref{ThmRelT2}). We also obtain a statement about property $(T)$ for $\widehat{\Gq}$ analogous to the property $(T)$ statement we mentioned above for classical bicrossed products (Theorem \ref{Thm-Crossed-T}).

Analogous statements hold for the relative co-Haagerup property of the pair $(G,\Gq)$ as well (Theorem \ref{Q-RelHaag}).  Moreover, we generalize a result of Jolissaint regarding Haagerup property to the setup of non tracial von Neumann algebras \cite{Jo07}: for a compact, state $($faithful normal$)$ preserving action of a countable discrete group with Haagerup property on a von Neumann algebra with the same property, the crossed product has the Haagerup property (Poposition \ref{PropHaag}). Hence, if $\Gamma$ and ${\rm L}^\infty(G)$ have the Haagerup property and the action $\alpha$ is compact  then ${\rm L}^\infty(\Gq)$ also has the Haagerup property. It is known that, for any compact quantum group $G$, if $\widehat{G}$ has the Haagerup property then ${\rm L}^\infty(G)$ also has the Haagerup property and the converse holds when $G$ is Kac \cite{DFSW13}. In general, one needs to assume that $\widehat{G}$ is strongly inner amenable \cite{OOT15}. Nevertheless, we show that if $\widehat{G}$ and $\Gamma$ both have the Haagerup property and the action of $\Gamma$ on $G$ is compact, then $\widehat\Gq$ has the Haagerup property (Theorem \ref{Thm-Crossed-H}).

It is now appropriate to highlight our examples. We point out that it is quite hard to generate examples of compact matched
pairs of groups for which both the actions are non-trivial. Thus, starting with a bicrossed product arising from a compact matched pair for which any one of the actions is trivial, we leverage a \textit{crossed homomorphism} $($see Section $7$ for definition$)$ to deform the original matched pair into one for which both the canonical actions become possibly 
non-trivial $($see discussions in the beginning of Sections $7.1.1$ and $7.1.2$). The added advantages with this deformation 
process are two fold. Firstly, the computations of the spectrum of the maximal C*-algebra and the intrinsic group of the deformed bicrossed product are  still convenient. Secondly, all the approximation properties and notably the relative co-property $(T)$ and the relative co-Haagerup property are inherited by the deformed bicrossed product. This allows us to provide a concrete infinite family of pairwise non-isomorphic, non-commutative and non-cocommutative infinite dimensional compact quantum groups whose duals have the property $(T)$ (Theorem \ref{InfiniteExample}). We mention that, as far as we are aware, these are the first explicit non-trivial examples of such compact quantum groups, since the twisting example of \cite{Fi10} is based on \cite[Theorem 3]{Fi10} and the proof of this theorem is erroneous. Using the same methodology, we also distinguish two compact bicrossed products arising from compact matched pairs both of which have relative co-property $(T)$ and for both of which the canonical actions are non-trivial. We are able to distinguish these quantum groups in the most obvious way, by computing the spectrum of the maximal $C^*$-algebra.

We also provide examples of non-trivial crossed product compact quantum groups by considering the canonical conjugation action induced by a countable subgroup of the spectrum of the full C*-algebra of a non-trivial compact quantum group. For these specific crossed products, we compute the intrinsic groups and the spectrum of the full C*-algebras, estimate the Cowling-Haagerup constants and characterized the property $(RD)$, the Haagerup property and the property $(T)$ in terms of the discrete group $\Gamma$ and the compact quantum group $G$ in Corollary \ref{CorExCrossed} and we apply this results to the free orthogonal and free unitary quantum groups in Example \ref{ExCrossedFree}. Finally, we provide some explicit non-trivial examples of crossed product without the Haagerup property but with the relative Haagerup property in Example \ref{ExCrossedRelH}.

The lay out of the paper is as follows. In Section 2, we jot down all the notations, recall preliminary facts and basics of compact quantum groups that is used all throughout this paper.  In the same section, we also prove that co-Haagerup property and co-weak amenability of a finite index quantum subgroup extend to the compact quantum group. Section 3 concentrates on the bicrossed product construction from compact matched pairs of a discrete group and a compact group. In Section 4 and Section 5, we respectively study (relative) Kazhdan property and (relative) Haagerup property for the dual of a compact bicrossed products. Section 6 is divided into many subsections, and in this section, we study the properties of crossed products of compact quantum groups by discrete groups.  Section 7 is dedicated to examples.

\textbf{Acknowledgement}: The authors are very grateful to Prof. Karl H. Hofmann for his illuminating remarks that led the authors to the explicit examples in Section 7. The authors would also like to thank Makoto Yamashita, Stefaan Vaes and Christian Voigt for comments and discussions.

%%%%%%%%%%%%%%%%%%%%%%%%%%%%%%%%%%%%%%%%%%%%%%%%%%%%%%%%%%%%%%%%%%%%%%%%%%%%
\section{Preliminaries}\label{section-preliminaries}
%%%%%%%%%%%%%%%%%%%%%%%%%%%%%%%%%%%%%%%%%%%%%%%%%%%%%%%%%%%%%%%%%%%%%%%%%%%%

\textbf{Notations.} Throughout the paper, $\mathcal{B}(H)$ denotes the von Neumann algebra of all bounded operators on the Hilbert space $H$. The inner products of Hilbert spaces are assumed to be linear in the first variable. The same symbol $\ot$ will denote the tensor product of Hilbert spaces, the minimal tensor product of C*-algebras and as well as the tensor product of von Neumann algebras.

%%%%%%%%%%%%%%%%%%%%%%%%%%%%%%%%%%%%%%%%%%%%%%%%%%%%%%%%%%%%%%%%%%%%%%%%%%%%
\subsection{Compact group action on countable sets}
%%%%%%%%%%%%%%%%%%%%%%%%%%%%%%%%%%%%%%%%%%%%%%%%%%%%%%%%%%%%%%%%%%%%%%%%%%%%

We first record some facts regarding actions of compact groups on countable sets. This will be necessary in studying the bicrossed product construction for compact matched pairs of groups.   

Let $X$ be a countable infinite set and let $S(X)$ be the group of bijections of $X$. It is a Polish group equipped with the topology of pointwise convergence which is the topology generated by the sets $S_{x,y}=\{\alpha\in S(X)\,:\,\alpha(x)=y\}$ for $x,y\in X$. Since $S_{x,y}^c=\cup_{z\in X\setminus\{y\}}S_{x,z}$, these sets are clopen in $S(X)$. Moreover, for any compact subset $K\subset S(X)$ and for any $x\in X$, the orbit $K\cdot x\subset X$ is finite. Indeed, from the open cover $K\subset\cup_{y\in X} S_{x,y}$, we find $y_1,\cdots, y_n\in X$ such that $K\subset\cup_{i=1}^n S_{x,y_i}$, which implies that $K\cdot x\subset\{y_1,\cdots,y_n\}$.

Let $\beta\,:\,G\rightarrow S(X)$ be a continuous right action of $G$ on $X$. To simplify the notations, we write $x\cdot g=\beta_g(x)$ for $g\in G$ and $x\in X$.

Observe that, since $\beta$ is continuous and $G$ is compact, every $\beta$-orbit in $X$ is finite. Fix $r,s\in X$ and denote by $A_{r,s}$ the set
$$A_{r,s}=\{g\in G\,:\,r\cdot g=s\}=\beta^{-1}(S_{r,s}).$$
Note that, since $\beta$ is continuous, $A_{r,s}$ is open and closed in $G$ for all $r,s\in X$. Hence, $1_{A_{r,s}}\in C(G)$. Moreover, $1_{A_{r,s}}\neq 0$ if and only if $r$ and $s$ are in the same orbit and we have the following relations:

\begin{enumerate}
\item $1_{A_{s,r}}1_{A_{t,r}}=\delta_{t,s}1_{A_{s,r}}$ for all $r,s,t\in X$.
\item $1_{A_{s,r}}1_{A_{s,t}}=\delta_{r,t}1_{A_{s,r}}$ for all $r,s,t\in\ X$.
\item $\sum_{s\in X}1_{A_{r,s}}=\sum_{s\in r\cdot G}1_{A_{r,s}}=1$ for all $r\in X$.
\item $\sum_{r\in X}1_{A_{r,s}}=\sum_{r\in s\cdot G}1_{A_{r,s}}=1$ for all $r\in X$.
\item If $r\cdot G=s\cdot G$, then $\Delta_G(1_{A_{s,r}})=\sum_{t\in s.G}1_{A_{s,t}}\ot 1_{A_{t,r}}$,
\end{enumerate}
where $\Delta_G$ is the usual comultiplication on $C(G)$. In other words, for every orbit $x\cdot G$, the matrix $(1_{A_{r,s}})_{r,s\in x\cdot G}\in M_{\vert x\cdot G\vert}(\C)\ot C(G)$ is a magic unitary and a unitary representation of $G$.

%%%%%%%%%%%%%%%%%%%%%%%%%%%%%%%%%%%%%%%%%%%%%%%%%%%%%%%%%%%%%%%%%%%%%%%%%%%%%%%
\subsection{Compact and discrete quantum groups}\label{SectionCQG}
%%%%%%%%%%%%%%%%%%%%%%%%%%%%%%%%%%%%%%%%

In this section, we recall well known and basic facts about compact quantum groups that will be indispensable. For the general theory of compact quantum groups, we refer the reader to \cite{Wo87, Wo95}.

For a compact quantum group $G$ with comultiplication $\Delta$ (or $\Delta_G$ when there can be ambiguity), we denote by $h$ (or $h_G)$ the Haar state on $G$ and by $C(G)$ (resp. ${\rm L}^\infty(G)$) the C*-algebra (resp. the von Neumann algebra) generated by the GNS construction of $h$. Hence we view $C(G)\subset\mathcal{B}({\rm L}^2(G))$, where ${\rm L}^2(G)$ is the GNS space of $h$. The reader should be cautious that the symbol $\Delta$ $($or $\Delta_{G})$ will be used to denote the comultiplications of all three compact quantum groups $C(G)$, the universal quantum group of $C(G)$ and ${\rm L}^\infty(G)$.  For two finite dimensional representations of $G$, we denote by Mor$(u,v)$ the space of intertwiners from $u$ to $v$ and by $u\ot v$ their tensor product. The trivial representation is denoted by $1$. We also denote by ${\rm Irr}(G)$ the set of equivalence classes of irreducible unitary representations of $G$. For $x\in{\rm Irr}(G)$, we choose a representative $u^x\in\mathcal{B}(H_x)\ot C(G)$, where $u^{x}$ is a irreducible representation on the Hilbert space $H_x$.

Recall that there is a natural involution $x\mapsto\overline{x}$ such that $u^{\overline{x}}$ is the unique (up to equivalence) irreducible representation of $G$ such that Mor$(1,x\ot \overline{x})\neq\{0\}\neq$Mor$(\overline{x}\ot x,1)$. For any $x\in{\rm Irr}(G)$, take a non-zero element $E_x\in{\rm Mor}(1,x\ot\overline{x})$ and define an anti-linear map $J_x\,:\, H_x\rightarrow H_{\overline{x}}$ by letting $\xi\mapsto (\xi^*\ot 1)E_x$. Define $Q_x=J_xJ_x^*\in\mathcal{B}(H_x)$. We normalize $E_x$ in such a way that ${\rm Tr}_x(Q_x)={\rm Tr}_x(Q_x^{-1})$, where ${\rm Tr}_x$ is the unique trace on $\mathcal{B}(H_x)$ such that ${\rm Tr}_x(1)={\rm dim}(x)$. This uniquely determines $Q_x$ and fixes $E_x$ up to a complex number of modulus $1$. The number ${\rm dim}_q(x):={\rm Tr}_x(Q_x)={\rm Tr}_x(Q_x^{-1})$ is called the \textit{quantum dimension of} $x$. Let $u^x_{cc}=(\id\otimes S_{G}^2)(u^x)$, where $S_G$ denotes the antipode of $G$. It is well known (see e.g. section 5 of \cite{Wo87}) that $Q_x$ is also uniquely determined by the fact that $Q_x\in$ Mor$(u^{x}, u^x_{cc})$ and  that $Q_x$ is invertible and ${\rm Tr}_x(Q_x)={\rm Tr}_x(Q_x^{-1})>0$.

We denote by ${\rm Pol}(G)$ the linear span of the coefficients of $\{u^x: x\in {\rm Irr}(G)\}$, which is a unital dense $\ast$-subalgebra of $C(G)$. We also denote by $C_m(G)$ the enveloping $C^*$-algebra of ${\rm Pol}(G)$, by $\lambda$ (or $\lambda_G$) the canonical surjection from $C_m(G)$ to $C(G)$ and by $\varepsilon$ (or $\varepsilon_G$) the counit on $C_m(G)$.

For a unital C*-algebra $A$, we use the standard notation ${\rm Sp}(A)$ to denote the spectrum of $A$.

Let $G$ be a compact quantum group and write $\chi(G):={\rm Sp}(C_m(G))$. It is a group with the product defined by $gh=(g\ot h)\circ\Delta$, for $g,h\in \chi(G)$. The unit of $\chi(G)$ is the counit $\varepsilon_G\in C_m(G)^*$ and the inverse of $g\in \chi(G)$ is given by $g\circ S_G$, where $S_G$ is the antipode on $C_m(G)$. Viewing $\chi(G)$ as a closed subset of the unit ball of $C_m(G)^*$, one can consider the weak* topology on $\chi(G)$ which turns  $\chi(G)$ to a compact group.

Finally, let ${\rm Int}(G)=\{u\in\mathcal{U}(C_m(G))\,:\,\Delta_G(u)=u\ot u\}$ denote the intrinsic group of $G$. It is the set of all $1$-dimensional irreducible unitary representations of $G$ and it is countable (since $C_m(G)$ is supposed to be separable).

For a classical l.c. group $H$, we denote by ${\rm Sp}(H)$ the spectrum of $C^*(H)$. It is a l.c. abelian group (with pointwise multiplication and weak* topology arising from the inclusion ${\rm Sp}(H)\subset C^*(H)^*$) and is compact if $H$ is discrete and mutatis mutandis; we will view it as the group of $1$-dimensional unitary representations of $H$. It is the Pontryagin dual of $H$ (when $H$ is abelian). We do not use the standard notation $\widehat{H}$ since we reserve this notation for the dual quantum group and, in the non-abelian case, it does not correspond to the dual quantum group.

We also denote by ${\rm Aut}(G)$ the set of quantum automorphisms of a compact quantum group $G$. More precisely, ${\rm Aut}(G)=\{\alpha\in{\rm Aut}(C_m(G))\,:\,\Delta\circ\alpha=(\alpha\ot\alpha)\circ\Delta\}$.

Hence, ${\rm Aut}(G)$ as a closed subgroup of the Polish \footnote{with respect to the topology of pointwise norm convergence} group ${\rm Aut}(C_m(G))$, is itself a Polish group.

Observe that each $\alpha\in{\rm Aut}(G)$ induces a bijection $\alpha\in S({\rm Irr}(G))$. Indeed, for $x\in{\rm Irr}(G)$, $\alpha(x)$ is the equivalence class of the irreducible unitary representation $(\id\ot\alpha)(u^x)$. By construction, the map ${\rm Aut}(G)\rightarrow S({\rm Irr}(G))$ is a group homomorphism.

We will need the following auxiliary result which is certainly well known to specialists. We include a proof since we could not locate any reference in the literature.

\begin{proposition}\label{PropContinuity}
The map ${\rm Aut}(G)\rightarrow S({\rm Irr}(G))$ is continuous.
\end{proposition}

\begin{proof}
We shall need the following well known lemma which is of independent interest. We include a proof for the convenience of the reader.

\begin{lemma}\label{LemCloseRep}
Let $u,v\in\mathcal{B}(H)\ot C_m(G)$ be two unitary representations of $G$ on the same finite dimensional Hilbert space $H$. If $\Vert u-v\Vert<1$, then $u$ and $v$ are equivalent.
\end{lemma}

\begin{proof}
Define $x=(\id\ot h)(v^*u)\in\mathcal{B}(H)$. Since $u$ and $v$ are unitary representations, $h$ being the Haar state forces $(x\ot 1)u=v(x\ot 1)$. We have $u^*(x^*x\ot 1)u=x^*x\ot 1$. Hence, $u^*\vert x\vert\ot 1 u=\vert x\vert\ot 1$. Now observe that $\Vert 1-x\Vert=\Vert(\id\ot h)(1-v^*u)\Vert\leq\Vert 1-v^*u\Vert=\Vert v-u\Vert<1$. Hence $x$ is invertible, and in the polar decomposition $x=w\vert x\vert$, the polar part $w$ is a unitary. Consequently, $v^*(w\vert x\vert\ot 1)u=v^*(w\ot 1)u(\vert x\vert\ot 1)=(w\ot 1)(\vert x\vert \ot 1)$. By uniqueness of the polar decomposition of $x\ot 1$, we deduce that $v^*(w\ot 1)u=w\ot 1$. Hence, $u$ and $v$ are equivalent.
\end{proof}

We can now prove the proposition. Let $(\alpha_n)_n$ be a sequence in ${\rm Aut}(G)$ which converges to $\alpha\in{\rm Aut}(G)$. Let $F\subset{\rm Irr}(G)$ be a finite subset and let $N\in\N$ be such that for all $n\geq N$ 
$$\Vert (\id\ot\alpha_n)(u^x)-(\id\ot\alpha)(u^x)\Vert< \frac{1}{2}\quad\text{for all }x\in F.$$
It follows from Lemma \ref{LemCloseRep} that $(\id\ot\alpha_n)(u^x)$ and $(\id\ot\alpha)(u^x)$ are equivalent for all $n\geq N$ and for all $x\in F$. This means that $\alpha_n(x)=\alpha(x)$ for all $x\in F$ whenever $n\geq N$. This establishes the continuity. 
\end{proof}

\begin{remark}\label{RmkAut}
We can also define ${\rm Aut}_r(G)=\{\alpha\in{\rm Aut}(C(G))\,:\,\Delta\circ\alpha=(\alpha\ot\alpha)\circ\Delta\}$ which is again a Polish group as it is a closed subgroup of the Polish group ${\rm Aut}(C(G))$. Since any $\alpha\in{\rm Aut}(G)$ preserves the Haar state, it defines a unique element in ${\rm Aut}_r(G)$. Hence, we have a canonical map ${\rm Aut}(G)\rightarrow {\rm Aut}_r(G)$ which is obviously a group homomorphism. Moreover, it is actually bijective. The inverse bijection is constructed in the following way. Since any $\alpha\in{\rm Aut}_r(G)$ restrict to an automorphism of ${\rm Pol}(G)$, it extends uniquely by the universal property to an automorphism in ${\rm Aut}(G)$. It is also easy to check that the map ${\rm Aut}(G)\rightarrow {\rm Aut}_r(G)$ is continuous.

Also, since any automorphism of $C(G)$ intertwining $\Delta$ has a unique normal extension to ${\rm L}^\infty(G)$, it induces a map ${\rm Aut}_r(G)\rightarrow {\rm Aut}_\infty(G)$, where ${\rm Aut}_\infty(G)=\{\alpha\in{\rm Aut}({\rm L}^\infty(G))\,:\,\Delta\circ\alpha=(\alpha\ot\alpha)\circ\Delta\}$. As before, this map is a bijective group homomorphism and is continuous (the topology on ${\rm Aut}({\rm L}^\infty(G))$ being governed by the pointwise $\Vert\cdot\Vert_{2,h}$ convergence).
\end{remark}
For a discrete group $\Gamma$ and a compact quantum group $G$, a group homomorphism $\alpha\,:\,\Gamma\rightarrow{\rm Aut}(G)$ is called \textit{an action by quantum automorphisms} and is denoted by $\alpha\,:\,\Gamma\curvearrowright G$, see \cite[Section 6]{Pa13}. We call such an action \textit{compact} if the closure of the image of $\Gamma$ in ${\rm Aut}(G)$ is compact. By remark \ref{RmkAut}, it follows that for compact actions, the associated actions of $\Gamma$ on the C*-algebra $C(G)$ (and $C_m(G)$) and the von Neumann algebra ${\rm L}^\infty(G)$ are compact. By Proposition \ref{PropContinuity}, it follows that for compact actions the induced action of $\Gamma$ on ${\rm Irr}(G)$ has all orbits finite. It is shown in \cite{MP15} that the converse is actually true: $\Gamma\curvearrowright G$ is compact if and only if the induced action of $\Gamma$ on $\text{Irr}(G)$ has all orbits finite.

The associated operator algebras of the discrete dual $\widehat{G}$ of $G$ are denoted by
$$\ell^\infty(\widehat{G})=\bigoplus^{\ell^\infty}_{x\in{\rm Irr}(G)}\mathcal{B}(H_x)\quad\text{and}\quad c_0(\widehat{G})=\bigoplus^{c_0}_{x\in{\rm Irr}(G)}\mathcal{B}(H_x).$$

We denote by $V_G=\bigoplus_{x\in{\rm Irr}(G)}u^x\in M(c_0(\widehat{G})\ot C_m(G))$ to be the \emph{maximal multiplicative unitary}. Let $p_x$ be the minimal central projection of $\ell^\infty(\widehat{G})$ corresponding to the block $\mathcal{B}(H_x)$. We say that $a\in{\ell}^\infty(\widehat{G})$ has \textit{finite support} if $ap_x=0$ for all but finitely many $x\in{\rm Irr}(G)$. The set of finitely supported elements of $\ell^\infty(\widehat{G})$ is dense in $c_c(\widehat{G})$ and the latter is equal to the algebraic direct sum $c_c(\widehat{G})=\bigoplus^{alg}_{x\in{\rm Irr}(G)}\mathcal{B}(H_x)$.

The (left-invariant) Haar weight on $\widehat{G}$ is the n.s.f. weight on $\ell^\infty({\widehat{G}})$ defined by 
$$h_{\widehat{G}}(a)=\sum_{x\in{\rm Irr}(G)}{\rm Tr}_x(Q_x){\rm Tr}_x(Q_xap_x),$$ 
whenever the formula makes sense. It is known that the GNS representation of $h_{\widehat{G}}$ is of the form $(\widehat{\lambda}_G,{\rm L}^2(G),\Lambda_{\widehat{G}})$, where $\Lambda_{\widehat{G}}\,:\, c_c(\widehat{G})\rightarrow {\rm L}^2(G)$ is linear with dense range and $\widehat{\lambda}_G\,:\, \ell^{\infty}(\widehat{G})\rightarrow\mathcal{B}({\rm L}^2(G))$ is a unital normal $*$-homomorphism such that $\Delta_G(x)=W_G(x\ot 1)W_G^*$ for all $x\in C(G)$, where $W_G=(\widehat{\lambda}_G\ot\lambda_G)(V_G)$. We call $W_G$ the \textit{reduced multiplicative unitary}.

%%%%%%%%%%%%%%%%%%%%%%%%%%%%%%%%%%%%%%%%%%%%%%%%%%%%%%%%%%%%%%%%%%%%%%%%%%%
\subsection{Approximation properties}\label{section-AP}
%%%%%%%%%%%%%%%%%%%%%%%%%%%%%%%%%%%%%%%%%%%%%%%%%%%%%%%%%%%%%%%%%%%%%%%%%%%

In this section we recall the definition of the Haagerup property, weak amenability and Cowling-Haagerup constants for discrete quantum groups.  We also show some basic facts we could not find in the literature: permanence of the (co)-Haagerup property and (co)-weak amenability from a quantum subgroup of finite index to the ambiance compact quantum group.

Let $G$ be a compact quantum group. For $\omega\in C_m(G)^*$, define its Fourier transform $\widehat{\omega}=(\id\ot\omega)(V)\in M(c_0(\widehat{G}))$, where $V=\bigoplus_{x\in{\rm Irr}(G)}u^x\in M(c_0(\widehat{G})\ot C_m(G))$ is the maximal multiplicative unitary. Observe that $\omega\mapsto \widehat{\omega}$ is linear and $\Vert\widehat{\omega}\Vert_{\mathcal{B}({\rm L}^2(G))}\leq\Vert\omega\Vert_{C_m(G)^*}$ for all $\omega\in C_m(G)^*$.

When $G$ is a classical compact group with Haar measure $\mu$ and $\nu$ is a complex Borel measure on $G$, then the Fourier transform $\widehat{\nu}\in M(C_r^*(G))$ is the operator $\widehat{\nu}=\int_G\lambda_gd\nu(g)\in M(C^*_r(G))\subset\mathcal{B}({\rm L}^2(G))$.

Following \cite{DFSW13}, we say that $\widehat{G}$ \textit{has the Haagerup property} if there exists a sequence of states $\omega_n\in C_m(G)^*$ such that $\omega_n\rightarrow\varepsilon_G$ in the weak* topology and $\widehat{\omega}_n\in c_0(\widehat{G})$ for all $n\in\N$.

For $a\in \ell^{\infty}(\widehat{G})$ with finite support, we define a finite rank map $m_a\,:\, C(G)\rightarrow C(G)$ by $(\id\ot m_a)(u^x)=u^x(ap_x\ot 1)$. We say that a sequence $a_i\in \ell^\infty(\widehat{G})$ \textit{converges pointwise to} $1$, if $\Vert a_ip_x-p_x\Vert_{\mathcal{B}(H_x)}\rightarrow 0$ for all $x\in{\rm Irr}(G)$.

Recall that $\widehat{G}$ is said to be \textit{weakly amenable} if there exists a sequence of finitely supported $a_i\in \ell^\infty(\widehat{G})$ converging pointwise to $1$ and such that $C=\sup_i\Vert m_{a_i}\Vert_{cb}<\infty$. The infimum of those $C$ is denoted by $\Lambda_{cb}(\widehat{G})$ (and is, by definition, infinite if $\widehat{G}$ is not weakly amenable). It was proved in \cite{KR99} that, when $G$ is Kac, we have $\Lambda_{cb}(\widehat{G})=\Lambda_{cb}(C(G))=\Lambda_{cb}({\rm L}^\infty(G))$.

\begin{definition}\label{Defn:FinInd}
We say that a compact quantum group $H$ is a (quantum) subgroup of $G$ is there exists a surjection $\rho:C_m(G)\rightarrow C_m(H)$ such that $(\rho\otimes \rho)\circ \Delta_{G}=\Delta_{H}\circ \rho$. We define the (left) coset space by $C_m(G/H) := \{a \in C_m(G) \mid (\id \otimes \rho) \Delta_{G}(a) = a \otimes 1\}$. We say that $H$ \textit{is a finite index subgroup of} $G$ if $C_m(G/H)$ is finite dimensional.
\end{definition}

We refer to \cite{DKSS12} for a systematic treatment of the notion of (closed) subgroups of locally compact quantum groups. 

\begin{theorem}\label{wab}
Let $H$ be a finite index quantum subgroup of $G$. Then the following holds.
\begin{enumerate}
\item If $\widehat{H}$ has the Haagerup property, then $\widehat{G}$ has the Haagerup property.
\item $\Lambda_{cb}(\widehat{G})\leq\Lambda_{cb}(\widehat{H})$. 
\end{enumerate}
\end{theorem}

\begin{proof} We will need the following Claim.

\textbf{Claim.}\textit{ If $H$ is a finite index quantum subgroup of $G$ with surjective morphism $\rho\,:\,C_m(G)\rightarrow C_m(H)$ then the set $N_y^\rho=\{x\in{\rm Irr}(G)\,:\,{\rm Mor}(v^y,(\id\ot\rho)(u^x))\neq\{0\}\}$ is finite for all $y\in{\rm Irr}(H)$, where $\{v^y\,:\,y\in{\rm Irr}(H)\}$ is a complete set of representatives.}

\textit{Proof of the Claim.} We first show that $N^\rho_1$ is finite. Let $x\in N_1^\rho$ and $\xi\in H_x$ be such that $\Vert\xi\Vert=1$ and $(\id\ot\rho)(u^x)\xi\ot 1=\xi\ot 1$. Choose an orthonormal basis $(e^x_k)_k$ of $H_x$ such that $e^x_1=\xi$. Observe that the coefficients of $u^x$ with respect with this orthonormal basis satisfy $\rho(u^x_{11})=1$ and $\rho(u^x_{k1})=0$ for all $k\neq 1$. It follows that $u^x_{11}\in C_m(G/H)$. Since the coefficients of non-equivalent representations are linearly independent and since $ C_m(G/H)$ is finite dimensional, it follows that the set $N^\rho_1$ is finite.

Suppose that there exists $y\in{\rm Irr}(H)\setminus\{1\}$ such that $N^\rho_y$ is infinite and let $(x_n)_{n\in\N\cup \{0\}}$ be an infinite sequence of elements in $N^\rho_y$. Since $(\id\otimes \rho)(u^{\overline{x}_0}\otimes u^{x_i})$ has a sub-representation isomorphic to $v^{\overline{y}}\ot v^y$, it contains the trivial representation. It follows that, for all $i\geq 1$, there exists $z_i\in N^\rho_1$ such that $z_i\subset \overline{x}_0\ot x_i$. Hence, $x_i\subset x_0\ot z_i$ and the set $\{z_i\,:\,i\geq 1\}$ is infinite, a contradiction. 

$(1).$ Let $(\mu_n)_{n\in\N}$ be a sequence of states on $C_m(H)$ such that $\widehat{\mu_n}\in c_0(\widehat{H})$ for all $n\in\N$ and $\mu_n \rightarrow \varepsilon_{H}$ in the weak* topology. Define $\mu_n\circ \rho\in C_m(G)^*$, where $\rho:C_m(G)\rightarrow C_m(H)$ is the subgroup surjection. Since $\varepsilon_{G}=\varepsilon_{H}\circ \rho$, we have $\omega_n\circ \rho\rightarrow \varepsilon_{G}$ in the weak* topology. Let $n\in\N$ and $\epsilon>0$. We need to show that the set $G_{n,\epsilon}=\{x\in \text{Irr}(G):|\vert (\id\otimes \omega_n)(u^x)\vert|\geq \epsilon\}$ is finite. Since $\widehat{\mu_n}\in c_0(\widehat{H})$, the set $H_{n,\epsilon}=\{y\in \text{Irr}(H): |\vert (id\otimes \mu_n)(v^y) \vert| \geq \epsilon\}$ is finite, and since $G_{n,\epsilon}=\cup_{y\in H_{n,\epsilon}}N_y^\rho$, by the previous claim we are done.

$(2).$ We may and will suppose that $\widehat{H}$ is weakly amenable.  Let $\epsilon>0$ and $a_i\in \ell^\infty(\widehat{H})$ be a sequence of finitely supported elements that converges to $1$ pointwise and such that ${\rm sup}_i\|m_{a_i}\|_{cb}\leq\Lambda_{cb}(\widehat{H})+\epsilon$.

We consider the dual morphism $\widehat{\rho}\,:\,c_0(\widehat{H})\rightarrow M(c_0(\widehat{G}))$, which is the unique non-degenerate $*$-homomorphism satisfying $(\id\otimes \rho)(V_{G})=(\widehat{\rho}\otimes \id)(V_{H})$. 

We first show that $\widehat{\rho}({a_i})\in \ell^\infty(\widehat{G})$ is finitely supported for all $i$ and the sequence $(\widehat{\rho}(a_i))_i$ converges to $1$ pointwise. Consider the functional $\omega_{a_i}\in C_m(H)^*$ defined by $(\id\otimes\omega_{a_i})(v^y)=a_ip_y$ for all $y\in{\rm Irr}(H)$ so $(\id\otimes\omega_{a_i})(V_{H})=a_i$ and, by definition of the dual morphism $\widehat{\rho}(a_i)=(\id\otimes \omega_{a_i}\circ \rho)(V_{G})$, we have $\widehat{\rho}(a_i)p_x=(\id\ot\omega_{a_i}\circ\rho)(u^x)$ and $\{x\in{\rm Irr}(G)\,:\,\widehat{\rho}(a_i)p_x\neq 0\}=\cup_{y\in{\rm Irr}(H), a_ip_y\neq 0}N_y^\rho$. Hence, $\widehat{\rho}(a_i)$ is finitely supported for all $i$. Moreover, for all $x\in{\rm Irr}(G)$,
\begin{eqnarray*}
\Vert\widehat{\rho}(a_i)p_x-p_x\Vert&=&\Vert(\id\ot\omega_{a_i}\circ\rho)(u^x)-p_x\Vert=\underset{y\in{\rm Irr}(H)\text{ and }x\in N_y^\rho}{\sup}\Vert(\id\ot\omega_{a_i})(v^y)-p_y\Vert\\
&=&\underset{y\in{\rm Irr}(H)\text{ and }x\in N_y^\rho}{\sup}\Vert a_ip_y-p_y\Vert\rightarrow _i0.
\end{eqnarray*}

We now show that $\sup_i\|m_{\widehat{\rho}(a_i)}\|_{cb}<\Lambda_{cb}(\widehat{H})+\epsilon$. First let us note that, by Fell's Absorption Principle, we have $(W_G)_{12}(V_G)_{13}=(V_G)_{23}(W_G)_{12}(V_G)_{23}^\ast$. Thus,  
there exists a $\ast$-homomorphism $\tilde{\Delta}_G:C(G)\rightarrow C(G)\otimes C_m(G)$ which extends the comultiplication $\Delta_G$ on Pol$(G)$. We now define a unital $\ast$-homomorphism $\pi: C(G)\rightarrow C(G)\otimes C(H)$ such that
$$\pi(x)=({\rm id}\otimes \lambda_H\circ \rho)\circ \tilde{\Delta}_G,$$
where $\lambda_H:C_m(H)\rightarrow C(H)$ denotes the canonical surjection given by the GNS-representation with respect to the haar state of $H$. Clearly, $\pi$ extends the map $({\rm id}\otimes \rho)\circ \Delta_G$ on Pol$(G)$. Now it not hard to see that the map $\pi$ is a \textit{right quantum homomorphism} (see section 1 of \cite{MRW12}); in other words $\pi$ satisfies the equations -
$$(\Delta_G\otimes {\rm id})\circ \pi= ({\rm id}\otimes \pi)\circ\Delta_G,$$
$$({\rm id}\otimes \Delta_H)\circ\pi= (\pi\otimes {\rm id})\circ \pi.$$
Both of the above equations follow easily from the coassociativity condition of the co-multiplication of $G$ and $H$ and from the fact that $\pi=({\rm id}\otimes \rho)\circ \Delta_G$ and $(\rho\otimes\rho)\circ \Delta_G=\Delta_H\circ \rho$ on Pol$(G)$. This together with Theorem 5.3 of \cite{MRW12} implies that there exists a unitary operator $V_{\rho}\in \mathcal{B}({\rm L}^2(G))\otimes C(H)$ such that 
$$\pi(x)=V_{\rho}(x\otimes 1)V_{\rho}^\ast.$$
Hence, it follows that, $\pi$ is isometric. 

It is now not hard to see that $(\id\otimes m_{a_i})\pi=\pi\circ m_{\widehat{\rho}(a_i)}$ for all $i$. Indeed, since $m_{a_i}(x)=(\id\otimes \omega_{a_i})\Delta_{H}(x)$ and $m_{\widehat{\rho}(a_i)}(x)=(\id\otimes \omega_{a_i}\circ\rho)\Delta_{G}(x)=(\id\otimes \omega_{a_i}\circ\rho)\Delta_{G}(x)$ for all $x\in{\rm Pol}(G)$, we find that for $x\in{\rm Pol}(G)$,
\begin{eqnarray*}
(\id\otimes m_{a_i})\pi(x)&=&(\id\ot\id\otimes \omega_{a_i})(\id\ot\Delta_{H})\circ\pi(x)=(\id\ot\id\otimes \omega_{a_i})(\pi\ot\rho)\circ\Delta_G(x)\\
&=&\pi\left((\id\otimes \omega_{a_i}\circ\rho)(\Delta_G(x)\right)=\pi\circ m_{\widehat{\rho}(a_i)}(x).
\end{eqnarray*}

Since $\pi$ is isometric, we have $\|m_{\widehat{\rho}(a_i)}\|_{cb}\leq \|m_{a_i}\|\leq\Lambda_{cb}(\widehat{H})+\epsilon$ for all $i$. Hence, $\Lambda_{cb}(\widehat{G})\leq \Lambda_{cb}(\widehat{H})+\epsilon$. Since $\epsilon$ is arbitrary the proof is complete. 
\end{proof}

%%%%%%%%%%%%%%%%%%%%%%%%%%%%%%%%%%%%%%%%%%%%%%%%%%%%%%%%%%%%%%%%%%%%%%%%%%%%%%
\section{Bicrossed products}\label{section-bicrossed}
%%%%%%%%%%%%%%%%%%%%%%%%%%%%%%%%%%%%%%%%**************************************

This section has two parts. In the first part, we discuss on matched pair of groups
of which one is compact and show an automatic regularity property of such matched pairs (Proposition \ref{PropMatched}).  In the second part, we study bicrossed products of  
compact matched pair of groups and study their representation theory and related concepts. 

%%%%%%%%%%%%%%%%%%%%%%%%%%%%%%%%%%%%%%%%%%%%%%%%%%%%%%%%%%%%%%%%%%%%%%%%%%%%%%%%%%%%%
\subsection{Matched pairs}
%%%%%%%%%%%%%%%%%%%%%%%%%%%%%%%%%%%%%%%%%%%%%%%%%%%%%%%%%%%%%%%%%%%%%%%%%%%%%%%%%%%%%%%%%

\begin{definition}[\cite{BSV03}] We say that a pair of l.c. groups $(G_1,G_2)$ is \textit{matched} if both $G_1,G_2$ are closed subgroups of a l.c. group $H$ satisfying $G_1\cap G_2=\{e\}$ and such that the complement of $G_1 G_2$ in $H$ has Haar measure zero.
\end{definition}

From a matched pair $(G_1,G_2)$ as above, one can construct a l.c. quantum group called the bicrossed product and it follows from \cite{VV03} that the bicrossed product is compact if and only if $G_1$ is discrete and $G_2$ is compact. In the next proposition, we show some regularity properties of matched pairs $(G_1,G_2)$ with $G_2$ being compact.

\begin{proposition}\label{PropMatched}
Let $(G_1,G_2)$ be a matched pair and suppose that $G_2$ is compact. Then $G_1 G_2=H$, and, for all $(\gamma,g)\in G_1\times G_2$ there exists unique $(\alpha_\gamma(g),\beta_\gamma(g))\in G_2\times G_1$ such that $\gamma g=\alpha_\gamma(g)\beta_g(\gamma)$. Moreover,
\begin{enumerate}
\item For $g,h\in G_2$ and $r,s\in G_1$, we have
\begin{equation}\label{EqMatched}
\alpha_r(gh)=\alpha_r(g)\alpha_{\beta_g(r)}(h),\,\,\beta_g(rs)=\beta_{\alpha_s(g)}(r)\beta_g(s)\quad\text{and}\quad\alpha_r(e)=e,\,\,\beta_g(e)=e.
\end{equation}
\item $\alpha$ is a continuous left action of $G_1$ on the topological space $G_2$. Moreover, the Haar measure on $G_2$ is $\alpha$-invariant whenever $G_1$ is discrete.
\item $\beta$ is a continuous right action of $G_2$ on the topological space $G_1$.
\item $\alpha$ is trivial $\Leftrightarrow$ $G_1$ is normal in $H$. Also, $\beta$ is trivial $\Leftrightarrow$ $G_2$ is normal in $H$.
\end{enumerate}
\end{proposition}

\begin{proof}
First observe that, since $G_2$ is compact, $H$ is Hausdorff and $G_1$ is closed, the set $G_1 G_2$ is closed. Hence, the complement of $G_1 G_2$ is open and has Haar measure zero. It follows that $G_1 G_2=H=H^{-1}=G_2^{-1}G_1^{-1}=G_2G_1$. Since $G_1\cap G_2=\{e\}$, the existence and uniqueness of $\alpha_\gamma(g)$ and $\beta_g(\gamma)$ for all $\gamma\in G_1$ and $g\in G_2$ are obvious. Then, the relations in $(1)$ and the facts that $\alpha$ (resp. $\beta$) is a left (resp. right) action as in the statement easily follow from the aforementioned uniqueness.

Now let us check the continuity of these actions. Since the subgroup $G_1$ is closed in the l.c. group $H$, so $H/G_1$ is a l.c. Hausdorff space equipped with the quotient topology and the projection map $\pi\,:\,H\rightarrow H/G_1$ is continuous. Hence, $\pi_{\vert_{G_2}}\,:\,G_2\rightarrow H/G_1$ is continuous and bijective since $G_1\cap G_2=\{e\}$ and $G_1G_2=H$. Since $G_2$ is compact, $\pi_{|G_2}$ is an homeomorphism. Let $\rho\,:\, H/G_1\rightarrow G_2$ be the inverse of $\pi_{|G_{2}}$ and observe that the map $\alpha\,:\,G_1\times G_2\rightarrow G_2$, $(\gamma,g)\mapsto\alpha_\gamma(g)$ satisfies $\alpha=\rho\circ\pi\circ\psi$, where $\psi\,:\,G_1\times G_2\rightarrow H$ is the continuous map given by $\psi(\gamma,g)=\gamma g$, for $\gamma\in G_1, g\in G_2$. Consequently, the action $\alpha$ is continuous. Since for all $\gamma\in G_1$ and $g\in G_2$, we have $\beta_g(\gamma)=\alpha_\gamma(g)^{-1}\gamma g$, we deduce the continuity of $\beta\,:\, G_1\times G_2\rightarrow G_1$, $(\gamma,g)\mapsto\beta_g(\gamma)$ from the continuity of $\alpha$ and the continuity of the product and inverse operations in $H$.

Moreover,  suppose that $G_1$ is discrete. Then $G_1$ is a co-compact lattice in $H$ and it follows from the general theory that $H$ is unimodular and hence there exists a unique $H$-invariant Borel probability measure $\nu$ on $H/G_1$. Consider the homeomorphism $\pi_{\vert_{G_2}}\,:\,G_2\rightarrow H/G_1$ and the Borel probability measure $\mu=(\pi_{\vert_{G_2}})_{*}(\nu)$ on $G_2$. Since, for all $\gamma\in G_1$, the map $\pi_{\vert_{G_2}}$ intertwines the homeomorphism $\alpha_\gamma$ of $G_2$ with the left translation by $\gamma$ on $H/G_1$ and since $\nu$ is invariant under the left translation by $\gamma$, it follows that $\mu$ is invariant under $\alpha_\gamma$. Also, $\pi_{\vert_{G_2}}$ intertwines the left translation by $h$ on $G_2$ with the left translation by $h$ on $H/G_1$  for all $h\in G_2$. Hence, $\mu$ is invariant under the left translation by $h$ for all $h\in G_2$. It follows that $\mu$ is the Haar measure.

Suppose that $G_1$ is normal is $H$. Then for all $\gamma\in G_1$, $g\in G_2$, we have $g^{-1}\gamma g=g^{-1}\alpha_\gamma(g)\beta_g(\gamma)\in G_1$. Since $g^{-1}\alpha_\gamma(g)\in G_2$ and $G_1\cap G_2=\{1\}$, we deduce that $g^{-1}\alpha_\gamma(g)=1$ for all $\gamma\in G_1$, $g\in G_2$. For the reverse implication in (4), suppose that $\alpha$ is trivial. Then for all $\gamma\in G_1$, $g\in G_2$, we have $\gamma g=g\beta_g(\gamma)\in G_1$. Hence, $g^{-1}G_1 g\subset G_1$ for all $g\in G_2$ and since we trivially have $\gamma^{-1} G_1\gamma\subset G_1$ for all $\gamma\in G_1$ and $H=G_1 G_2$, we deduce that $G_1$ is normal in $H$. The proof of the last assertion of the Proposition is analogous.
\end{proof}

The next Proposition is well known, it is called the Zappa-Sz\'ep product (also known as the Zappa-R\'edei-Sz\'ep product, general product or knit product) and it is a converse of Proposition \ref{PropMatched}. We include a proof for the convenience of the reader.

\begin{proposition}\label{PropRecMatched}
Suppose that $G_1$ and $G_2$ are two l.c. groups with a continuous left action $\alpha$ of $G_1$ on the topological space $G_2$ and a continuous right action $\beta$ of $G_2$ on the topological space $G_1$ satisfying the relations $(\ref{EqMatched})$. Then there exists a l.c. group $H$ for which $G_1,G_2$ are closed subgroups satisfying  $G_1\cap G_2=\{e\}$, $H=G_1G_2$, and for all $\gamma\in G_1,g\in G_2$, $\gamma g=\alpha_\gamma(g)\beta_g(\gamma)$.
\end{proposition}

\begin{proof}
Consider the l.c. space $H=G_1\times G_2$ and define a product on $H$ by the formula:
$$(r,g)(s,h)=(\beta_h( r )s,g\alpha_r(h))\quad\text{for all }r,s\in G_1,\,\,g,h\in G_2.$$
It is routine to check that this multiplication turns $H$ into a l.c. group. Moreover, we may identify $G_1$ with a closed subgroup of $H$ by the map $G_1\ni r\mapsto (r,1)\in G_1\times G_2$ and $G_2$ with a closed subgroup of $H$ by the map $G_2\ni g\mapsto (1,g)\in G_1\times G_2$. After these identifications, we have $H=G_1G_2$, $G_1\cap G_2=\{e\}$, and for all $\gamma\in G_1,g\in G_2$, $\gamma g=\alpha_\gamma(g)\beta_g(\gamma)$.
\end{proof}

%%%%%%%%%%%%%%%%%%%%%%%%%%%%%%%%%%%%%%%%%%%%%%%%%%%%%%%%%%%%%%%%%%%%%%%%%
\subsection{The bicrossed product construction}
%%%%%%%%%%%%%%%%%%%%%%%%%%%%%%%%%%%%%%%%%%%%%%%%%%%%%%%%%%%%%%%%%%%%%%%%%
We first construct the bicrossed product from a compact matched pair and then study its representation theory. Along the way we prove some straight forward consequences e.g., amenability, $K$-amenability and Haagerup property of the dual of the bicrossed product. We also compute the intrinsic group and the spectrum of the maximal C*-algebra of the 
bicrossed product.

Let $(\Gamma,G)$ be a matched pair of a countable discrete group $\Gamma$ and a compact group $G$. Associated to the continuous action $\beta$ of the compact group $G$ on the countable infinite set $\Gamma$, we have a magic unitary $v^{\gamma\cdot G}=(v_{rs})_{r,s\in\gamma\cdot G}\in M_{\vert\gamma\cdot G\vert}(\C)\ot C(G)$  for every $\gamma\cdot G\in\Gamma/G$, where $v_{rs}=1_{A_{r,s}}$ and $A_{r,s}=\{g\in G\,:\,\beta_g(r)=s\}$. 

We define the C*-algebra $A_m=\Gamma\,_{\alpha,f}\ltimes C(G)$ to be the full crossed product and the C*-algebra  $A=\Gamma _\alpha\ltimes C(G)$ to be the reduced crossed product. With abuse of notation, we denote by $\alpha$ the canonical injective maps from $C(G)$ to $A_m$ and from $C(G)$ to $A$. We also denote by $u_\gamma$, $\gamma\in\Gamma$, the canonical unitaries viewed in either $A_m$ or $A$. Observe that $A_m$ is the enveloping C*-algebra of the unital *-algebra
$$\mathcal{A}={\rm Span}\{u_\gamma\alpha(u^x_{ij})\,:\,\gamma\in\Gamma, x\in{\rm Irr}(G),1\leq i,j\leq{\rm dim}(x)\}.$$

Let $\lambda\,:\, A_m\rightarrow A$ be the canonical surjection. Since the action $\alpha$ on $(G,\mu)$ is  $\mu$-preserving and $\mu$ is a probability measure, so there exists a unique faithful trace $\tau$ on $A$ such that
\begin{align*}
\tau(u_\gamma\alpha(F))=\delta_{\gamma,e}\int Fd\mu, \text{  }\gamma\in \Gamma, F\in C(G).
\end{align*}

\begin{theorem}\label{ThmBicrossed}
There exists a unique unital $*$-homomorphism $\Delta_m\,:\,A_m\rightarrow A_m\ot A_m$ such that
$$\Delta_m\circ\alpha=(\alpha\ot\alpha)\circ\Delta_G\quad\text{and}\quad\Delta_m(u_\gamma)=\sum_{r\in \gamma\cdot G}u_\gamma\alpha(v_{\gamma,r})\ot u_r,\,\,\forall\gamma\in\Gamma.$$
Moreover, $\Gq=(A_m,\Delta_m)$ is a compact quantum group and we have:
\begin{enumerate}
\item The Haar state of $\Gq$ is $h=\tau\circ\lambda$, hence $\Gq$ is Kac.
\item The elements $V^{\gamma\cdot G}=\sum_{r,s\in\gamma\cdot G}e_{r,s}\ot u_r\alpha(v_{r,s})\in M_{\vert\gamma \cdot G\vert}(\C)\ot A_m$, for $\gamma\in\Gamma$, are pairwise non-equivalent irreducible unitary representations of $\Gq$ such that $\overline{V^{\gamma\cdot G}}\simeq V^{\gamma^{-1}\cdot G}$ and any irreducible unitary representation of $\Gq$ is a equivalent to a subrepresentation of $V^{\gamma\cdot G}\ot v^x$ for some $\gamma\cdot G\in\Gamma/G$ and $x\in{\rm Irr}(G)$, where $v^x=(\id\ot\alpha)(u^x)$. 
\item We have $C_m(\Gq)=A_m$, $C(\Gq)=A$, $\PolG=\mathcal{A}$, $\lambda$ is the canonical surjection and ${\rm L}^\infty(\Gq)$ is the von Neumann algebraic crossed product.
\item The counit $\varepsilon_\Gq\,:\, C_m(\Gq)\rightarrow\C$ is the unique unital $*$-homomorphism such that $\varepsilon_\Gq(\alpha(F))=F(e)$ for all $F\in C(G)$ and $\varepsilon_\Gq(u_\gamma)=1$ for all $\gamma\in\Gamma$.
\end{enumerate}
\end{theorem}

The compact quantum group $\Gq$ associated to the compact matched pair $(\Gamma,G)$ in Theorem \ref{ThmBicrossed} is called the \textit{bicrossed product}.
\begin{proof}
The uniqueness of $\Delta_m$ is obvious. To show the existence, it suffices to check that $\Delta_m$ satisfies the universal property of $A_m$.

Let us check that $\gamma\mapsto\Delta_m(u_\gamma)$ is a unitary representation of $\Gamma$. Let $\gamma\in\Gamma$. We first check that $\Delta_m(u_\gamma)$ is unitary. Observe that, for all $g\in G$ and $\gamma\in\Gamma$, we have
$$1=\beta_g(\gamma^{-1}\gamma)=\beta_{\alpha_\gamma(g)}(\gamma^{-1})\beta_g(\gamma).$$
Hence, $(\beta_g(\gamma))^{-1}=\beta_{\alpha_\gamma(g)}(\gamma^{-1})$. From this relation it is easy to check that $\Gamma^{-1}\cdot G=\{r^{-1}\,:\, r\in\gamma\cdot G\}$ and $\alpha_{\gamma}(v_{\gamma,r^{-1}})=v_{\gamma^{-1},r}$ for all $r\in\Gamma$. It follows that
\begin{align*}
\Delta_m(u_\gamma)^*&=\sum_{r\in\gamma\cdot G}\alpha(v_{\gamma,r})u_{\gamma^{-1}}\ot u_{r^{-1}}=\sum_{r\in\gamma\cdot G}u_{\gamma^{-1}}\alpha(\alpha_\gamma(v_{\gamma,r}))\ot u_{r^{-1}}\\
&=\sum_{r\in\gamma^{-1}\cdot G}u_{\gamma^{-1}}\alpha(v_{\gamma^{-1},r})\ot u_{r}=\Delta_m(u_{\gamma^{-1}}).
\end{align*} 

Let $\gamma_1,\gamma_2\in\Gamma$. We have
$$
\Delta_m(u_{\gamma_1})\Delta_m(u_{\gamma_2})
=\sum_{r\in\gamma_1\cdot G,s\in\gamma_2\cdot G}u_{\gamma_1}\alpha(v_{\gamma_1,r})u_{\gamma_2}\alpha(v_{\gamma_2,s})\ot u_{rs}=\sum_{r,s}u_{\gamma_1\gamma_2}\alpha\left(\alpha_{\gamma_2^{-1}}(v_{\gamma_1,r})v_{\gamma_2,s}\right)\ot u_{rs}.
$$
Observe that $\alpha_{\gamma_2^{-1}}(v_{\gamma_1,r})v_{\gamma_2,s}=1_{B_{\gamma_1,\gamma_2,r,s}}$, where
$$B_{\gamma_1,\gamma_2,r,s}=\{g\in G\,:\,\beta_{\alpha_{\gamma_2}(g)}(\gamma_1)=r\text{ and }\beta_g(\gamma_2)=s\}\subset A_{\gamma_1\gamma_2,rs}=\{g\in G\,:\,\beta_g(\gamma_1\gamma_2)=rs\},$$
since $\beta_{\alpha_{\gamma_2}(g)}(\gamma_1)\beta_g(\gamma_2)=\beta_g(\gamma_1\gamma_2)$. In particular, $B_{\gamma_1,\gamma_2,r,s}=\emptyset$ whenever $rs\notin\gamma_1\gamma_2\cdot G$; hence
$$
\Delta_m(u_{\gamma_1})\Delta_m(u_{\gamma_2})=\sum_{t\in\gamma_1\gamma_2\cdot G,r\in\gamma_1\cdot G}u_{\gamma_1\gamma_2}\alpha\left(1_{B_{\gamma_1,\gamma_2,r,r^{-1}t}}\right)\ot u_t=\sum_{t\in\gamma_1\gamma_2\cdot G}u_{\gamma_1\gamma_2}\alpha(F_t)\ot u_t,$$
where $F_t=\sum_r1_{B_{\gamma_1,\gamma_2,r,r^{-1}t}}=1_{\sqcup_rB_{\gamma_1,\gamma_2,r,r^{-1}t}}=1_{A_{\gamma_1\gamma_2,t}}$, and  $A_{\gamma_1\gamma_2,t}=\{g\in G\,:\,\gamma_1\gamma_2\cdot g=t\}$. Consequently, $1_{A_{\gamma_1\gamma_2,t}}=v_{\gamma_1\gamma_2,t}$ and $\Delta_m(u_{\gamma_1})\Delta_m(u_{\gamma_2})=\Delta_m(u_{\gamma_1\gamma_2})$. Since $\Delta_m(u_e)=1$, it follows that $\gamma\mapsto\Delta_m(u_\gamma)$ is a unitary representation of $\Gamma$.

Let us now check that the relations of the crossed product are satisfied. For $\gamma\in\Gamma$ and $F\in{\rm Pol}(G)$ we have:
\begin{eqnarray*}
\Delta_m(u_\gamma)\Delta_m(\alpha(F))\Delta_m(u_\gamma^*)&=&
\sum_{r,s}(u_\gamma\ot u_r)(\alpha\ot\alpha)\left((v_{\gamma,r}\ot 1)\Delta_G(F)\right)(u_{\gamma^{-1}}\alpha(v_{\gamma^{-1},s})\ot u_s)\\
&=&\sum_{r,s}(u_\gamma\ot u_r)(\alpha\ot\alpha)\left((v_{\gamma,r}\alpha_{\gamma^{-1}}(v_{\gamma^{-1},s})\ot 1)\Delta_G(F)\right)(u_{\gamma^{-1}}\ot u_s)\\
&=&\sum_{r,s}(\alpha\ot\alpha)\left((\alpha_\gamma(v_{\gamma,r})v_{\gamma^{-1},s}\ot 1)(\alpha_\gamma\ot\alpha_r)(\Delta_G(F))\right)(1\ot u_{rs})\\
&=&\sum_{r,t}(\alpha\ot\alpha)\left((\alpha_\gamma(v_{\gamma,r})v_{\gamma^{-1},r^{-1}t}\ot 1)(\alpha_\gamma\ot\alpha_r)(\Delta_G(F))\right)(1\ot u_{t})\\
&=&\sum_{t}(\alpha\ot\alpha)(H_t)(1\ot u_{t}),
\end{eqnarray*}
where $H_t=\sum_r(\alpha_\gamma(v_{\gamma,r})v_{\gamma^{-1},r^{-1}t}\ot 1)(\alpha_\gamma\ot\alpha_r)(\Delta_G(F))$. 

Observe that $\alpha_\gamma(v_{\gamma,r})v_{\gamma^{-1},r^{-1}t}=1_{B_{\gamma,r,t}}$, where
$$B_{\gamma,r,t}=\{g\in G\,:\,\beta_{\alpha_{\gamma^{-1}}(g)}(\gamma)=r\text{ and }\beta_g(\gamma^{-1})=r^{-1}t\}.$$
Since $\beta_{\alpha_{\gamma^{-1}}(g)}(\gamma)\beta_g(\gamma^{-1})=\beta_g(\gamma\gamma^{-1})=\beta_g(e)=e$, we deduce that $B_{\gamma,r,t}=\emptyset$ whenever $t\neq e$, and it is easy to see that $\sqcup_{r\in\gamma\cdot G}B_{\gamma,r,e}=G$. Hence, $H_t=0$ for $t\neq e$.  Again for $g\in B_{\gamma,r,e}$ and $h\in G$, one has $H_e(g,h)=F(\alpha_{\gamma^{-1}}(g)\alpha_{r^{-1}}(h))=F(\alpha_{\gamma^{-1}}(g)\alpha_{\beta_g(\gamma^{-1})}(h))=F(\alpha_{\gamma^{-1}}(gh))$. It follows that $H_e=\Delta_G(\alpha_\gamma(F))$. Consequently, $\Delta_m(u_\gamma)\Delta_m(\alpha(F))\Delta_m(u_\gamma^*)=(\alpha\ot\alpha)(H_e)$.  This completes the proof of the existence of $\Delta_m$.

It is clear that $v^x$ $($as defined in the statement$)$ is unitary and since $(\alpha\ot\alpha)\Delta_G=\Delta_m\circ\alpha$, we have $\Delta_m(v^x_{ij})=\sum_k v^x_{ik}\ot v^x_{kj}$. Observe that $V^{\gamma\cdot G}=D_\gamma(\id\ot\alpha)(v^{\gamma\cdot G})\in M_{\vert\gamma\cdot G\vert}(\C)\ot\mathcal{A}$, where $D_\gamma$ is the diagonal matrix with entries $u_r$, $r\in\gamma\cdot G$. Hence, $V^{\gamma\cdot G}$ is unitary. Moreover,
\begin{eqnarray*}
\Delta_m(V^{\gamma\cdot G}_{rs})&=&\Delta_m(u_r\alpha(v_{rs}))
=\sum_{t\in r\cdot G=\gamma\cdot G}(u_r\alpha(v_{rt})\ot u_t)(\alpha\ot\alpha)(\Delta_G(v_{rs}))\\
&=&\sum_{t,z\in\gamma\cdot G}u_r\alpha(v_{rt}v_{rz})\ot u_t\alpha(v_{zs})
=\sum_{t\in\gamma\cdot G}u_r\alpha(v_{rt})\ot u_t\alpha(v_{ts})=\sum_{t\in\gamma\cdot G}V^{\gamma\cdot G}_{rt}\ot V^{\gamma\cdot G}_{ts}.
\end{eqnarray*}
It follows from \cite[Definition 2.1']{Wa95a} that $\Gq$ is a compact quantum group and $V^{\gamma\cdot G}$, $v^x$ are unitary representations of $\Gq$ for all $\gamma\cdot G\in\Gamma/G$ and $x\in{\rm Irr}(G)$.

$(1)$. Since $\sum_s V^{\gamma\cdot G}_{rs}=u_r$, the linear span of the coefficients of the representations $V^{\gamma\cdot G}\ot v^x$ for $\gamma\in \Gamma/G$ and $x\in \text{Irr}(G)$ is equal to $\mathcal{A}$. Hence, it suffices to check the invariance of $h$ on the coefficients of $V^{\gamma\cdot G}\ot v^x$. We have
$$h(V^{\gamma\cdot G}_{rs} v^x_{ij}))=h(u_r\alpha(v_{rs}v^x_{ij}))=\delta_{r,e}\int_Gv_{es}v^x_{ij}d\mu=\delta_{r,e}\delta_{s,e}\int_Gv^x_{ij}d\mu=\delta_{r,e}\delta_{s,e}\delta_{x,1},$$
since $v_{es}=\delta_{s,e}1$ and $v^x$ is irreducible. Hence, if $x\neq 1$, we have
$$(\id\ot h)\Delta_m(V^{\gamma\cdot G}_{rs} v^x_{ij})=\sum_{t,k} V^{\gamma\cdot G}_{rt}v^x_{ik}h(V^{\gamma\cdot G}_{ts}v^x_{kj})=0=\sum_{t,k} h(V^{\gamma\cdot G}_{rt}v^x_{ik})V^{\gamma\cdot G}_{ts}v^x_{kj}=(h\ot\id)\Delta_m(V^{\gamma\cdot G}_{rs} v^x_{ij}).$$
And, if $x=1$, we have $(\id\ot h)\Delta_m(V^{\gamma\cdot G}_{rs})=\sum_{t} V^{\gamma\cdot G}_{rt}h(V^\gamma_{ts})=\delta_{\gamma,e}1=(h\ot\id)\Delta_m(V^{\gamma\cdot G}_{rs})$. It follows that $h$ is the Haar state. 

$(2)$. For the unitary representation $V^{\gamma\cdot G}$ (of $\Gq$ or $G$), we denote by $\chi(\gamma\cdot G)=\sum_{r\in\gamma\cdot G}u_r\alpha(v_{rr})$ its character. One has, for $\gamma,\gamma'\in\Gamma$,
\begin{eqnarray*}
h(\chi(\gamma\cdot G)^*\chi(\gamma'\cdot G))&=&\sum_{r\in\gamma\cdot G,s\in\gamma'\cdot G}h(\alpha(v_{rr})u_{r^{-1}s}\alpha(v_{ss}))
=\sum_{r\in\gamma\cdot G,s\in\gamma'\cdot G}h(u_{r^{-1}s}\alpha(\alpha_{s^{-1}r}(v_{rr})v_{ss}))\\
&=&\delta_{\gamma\cdot G,\gamma'\cdot G}\sum_{r\in\gamma\cdot G}\int_Gv_{rr}d\mu=\delta_{\gamma\cdot G,\gamma'\cdot G}\int_G(\sum_rv_{rr})d\mu=\delta_{\gamma\cdot G,\gamma'\cdot G},
\end{eqnarray*}
since $\int_G(\sum_rv_{rr})d\mu$ is equal to the dimension of the invariant vectors in the representation $(v_{rs})_{r,s\in\gamma\cdot G}$ which is $1$ since the action of $G$ on $\gamma\cdot G$ is transitive. This shows that the representations $V^{\gamma\cdot G}$ are irreducible and pairwise non-isomorphic for $\gamma\cdot G\in \Gamma/G$. Since the linear span of the coefficients of $V^{\gamma\cdot G}\ot v^x$ is equal to $\mathcal{A}$ and hence dense in $A_m$, it follows that any irreducible representation of $\Gq$ is equivalent to some subrepresentation of $V^{\gamma\cdot G}\ot v^x$. Finally, the bicrossed product relations imply that $(\beta_g(\gamma))^{-1}=\beta_{\alpha_\gamma(g)}(\gamma^{-1})$ for all $g\in G$, $\gamma\in\Gamma$ hence, $\gamma^{-1}\cdot G=(\gamma\cdot G)^{-1}$ and, the linear map $s_\gamma\,:\,\C\rightarrow \C^{\vert\gamma\cdot G\vert}\ot\C^{\vert\gamma^{-1}\cdot G\vert}$, $s_\gamma(1)=\sum_{r\in\gamma\cdot G}e_r\ot e_{r^{-1}}$ is well defined. Using the relations $u_r\alpha(v_{rs})u_r^*=\alpha(v_{rs}\circ\alpha_{r^{-1}})=\alpha(v_{r^{-1}s^{-1}})$ for all $r,s\in\gamma\cdot G$ it is easy to check that $s_\gamma\in{\rm Mor}(1,V^{\gamma\cdot G}\ot V^{\gamma^{-1}\cdot G})$ which implies that $\overline{V^{\gamma\cdot G}}\simeq V^{\gamma^{-1}\cdot G}$.

$(3)$. We have already shown that ${\rm Pol}(\Gq)=\mathcal{A}$. It follows that $C_m(\Gq)=A_m$. Since $\lambda$ is surjective and $\tau$ is faithful on $A$, it follows that $C(\Gq)=A$ and ${\rm L}^\infty(\Gq)$ is the bicommutant of $A$ in $\mathcal{B}(\ell^2(\Gamma)\ot{\rm L}^2(G))$ i.e., it is the von Neumann algebraic crossed product. Finally, since $\lambda$ is the identity on $\mathcal{A}={\rm Pol}(\Gq)$, it follows that $\lambda$ is the canonical surjection.

$(4)$. The fact that $\varepsilon_\Gq(\alpha(F))=F(e)$ for all $F\in C(G)$ is obvious since $\alpha$ intertwines the colmultiplications. Fix $\gamma\in\Gamma$. Since $V^{\gamma\cdot G}$ is irreducible, we have that $(\id\ot\varepsilon_\Gq)(V^{\gamma\cdot G})=1$. Hence,
$$1=\sum_{r,s\in\gamma.G}e_{r,s}\varepsilon_\Gq(u_r)v_{r,s}(e)=\sum_{r\in\gamma.G}e_{r,r}\varepsilon_\Gq(u_r).$$
It follows that $\varepsilon_\Gq(u_\gamma)=1$.
\end{proof}

\begin{remark}\label{RmkNonTrivialMatched}
Let $\Gq$ be the bicrossed product coming from a compact matched pair $(\Gamma,G)$ as above. From the definition, it is easy to check that $C_m(\Gq)$ is commutative if and only if the action $\alpha$ is trivial and $\Gamma$ is abelian. Moreover, $\Gq$ is cocommutative if and only if the action $\beta$ is trivial and $G$ is abelian. %Hence, the procedures described in Examples \ref{ExMatchedPair} and \ref{ExRightCrossed} lead to many non-trivial examples of compact quantum groups.
\end{remark}

\begin{remark}\label{RmkFull}
The following observation is well known. Let $\alpha\,:\,\Gamma\curvearrowright A$ be an action of the countable group $\Gamma$ on the unital C*-algebra $A$ and let $C$ be the full crossed product which is generated by the unitaries $u_\gamma$, $\gamma\in\Gamma$, and by the copy $\alpha(A)$ of the C*-algebra $A$. If $A$ has a character $\varepsilon\in A^*$ such that $\varepsilon(\alpha_\gamma(a))=\varepsilon(a)$ for all $a\in A$ and $\gamma\in\Gamma$, then the C*-subalgebra $B\subset C$ generated by $\{u_\gamma\,:\,\gamma\in\Gamma\}$ is canonically isomorphic to $C^*(\Gamma)$. Indeed, it suffices to check that $B$ satisfies the universal property of $C^*(\Gamma)$. Let $v\,:\,\Gamma\rightarrow\mathcal{U}(H)$ be a unitary representation of $\Gamma$ on $H$. Consider the unital $*$-homomorphism $\pi\,:\, A\rightarrow\mathcal{B}(H)$ given by $\pi(a)=\varepsilon(a)\id_{H}$, $a\in A$. We have
$v_\gamma\pi(a)v_{\gamma^{-1}}=\varepsilon(a)\id_H=\varepsilon(\alpha_\gamma(a))\id_H=\pi(\alpha_\gamma(a))$. Hence, we obtain a representation of $C$ that we can restrict to $B$ to get the universal property.

Let $(\Gamma,G)$ be a matched pair.  Since the map $\varepsilon\,:\,C(G)\rightarrow\C$ by $F\mapsto F(e)$ is a $\alpha$-invariant character, it follows from the preceding observation that the C*-subalgebra of $A_m$ generated by $u_\gamma$, $\gamma\in\Gamma$, is canonically isomorphic to $C^*(\Gamma)$.
\end{remark}

\vspace{0.2cm} 

We now give some obvious consequences of the preceding result concerning amenability, $K$-amenability and the Hagerup property. The first assertion of the following corollary is already known \cite{DQV02} but we include an easy proof for the convenience of the reader. We refer to \cite{Ve04} for the definition of K-amenability of discrete quantum groups.

\begin{corollary}\label{CorMatched}
The following holds:
\begin{enumerate}
\item $\Gq$ is co-amenable if and only if $\Gamma$ is amenable.
\item If $\Gamma$ is $K$-amenable, then $\widehat{\Gq}$ is $K$-amenable.
\item If $\widehat{\Gq}$ has the Haagerup property, then $\Gamma$ has the Haagerup property.
\item If the action of $\Gamma$ on ${\rm L}^\infty(G)$ is compact and $\Gamma$ has the Haagerup property, then $\widehat{\Gq}$ has the Haagerup property.
\end{enumerate}
\end{corollary}

\begin{proof}
$(1)$. If $\Gamma$ is amenable, then we trivially have that $\lambda$ is an isomorphism; hence, $\Gq$ is co-amenable. Conversely, if $\Gq$ is co-amenable, then the Haar state $h=\tau\circ\lambda$ is faithful on $A_m$. Since $h(u_\gamma)=\delta_{\gamma,e}$, $\gamma\in \Gamma$, we conclude from Remark \ref{RmkFull}, that the canonical trace on $C^*(\Gamma)$ has to be faithful. Hence, $\Gamma$ is amenable.

$(2)$. It is an immediate consequence of \cite[Theorem 2.1 (c)]{Cu83}.

$(3)$. It follows from \cite[Theorem 6.7]{DFSW13}, since ${\rm L}(\Gamma)$ is a von Neumann subalgebra of ${\rm L}^\infty(\Gq)$.

$(4)$. This is a direct consequence of \cite[Corollary 3.4]{Jo07} and \cite[Theorem 6.7]{DFSW13}.
\end{proof}

We end this section with a description of the ${\rm Int}(\Gq)$ and $\chi(\Gq)$ (see Section \ref{SectionCQG}) in terms of the matched pair $(G,\Gamma)$. It will be used to distinguish various explicit examples in Section \ref{section-examples}.

Observe that the relations in Equation $(\ref{EqMatched})$ imply that $\Gamma^\beta=\{\gamma\in\Gamma\,:\,\beta_g(\gamma)=\gamma\,\,\forall g\in G\}$ and $G^\alpha=\{g\in G\,:\,\alpha_\gamma(g)=g\,\,\forall\gamma\in\Gamma\}$ are respectively subgroups of $\Gamma$ and $G$. Moreover, since $\beta$ is continuous, $G^\beta$ is closed, hence compact. Thus, when $(\Gamma,G)$ is a compact matched pair, the relations in Equation $(\ref{EqMatched})$ imply that the associations 
$$\gamma\cdot\omega=\omega\circ\alpha_{\gamma}\quad\text{and}\quad g\cdot\mu=\mu\circ\beta_g,\quad\text{for all }\gamma\in\Gamma,g\in G,\omega\in{\rm Sp}(G),\mu\in{\rm Sp}(\Gamma),$$
define two actions by group homomorphisms, namely: $(i)$ right action of $\Gamma^\beta$ on ${\rm Sp}(G)$ that we still denote by $\alpha$, and $(ii)$ left action of $G^\alpha$ on ${\rm Sp}(\Gamma)$ that we still denote by $\beta$. Also, $\beta$ is a continuous action by homeomorphisms.

\begin{proposition}\label{PropInt}
There are canonical group isomorphisms:
$${\rm Int}(\Gq)\simeq{\rm Sp}(G)\rtimes _\alpha\Gamma^\beta\quad\text{and}\quad\chi(\Gq)\simeq G^\alpha \,_\beta\ltimes{\rm Sp}(\Gamma).$$
The second isomorphism is moreover a homeomorphism.
\end{proposition}

\begin{proof}
The irreducible representation $V^{\gamma.G}$ of $\Gq$ is of dimension $1$ $\Leftrightarrow$ $\vert\gamma\cdot G\vert=1$ $\Leftrightarrow$ $\gamma\in\Gamma^\beta$. By assertion $(2)$ of Theorem \ref{ThmBicrossed}, there is a bijective map 
$$\pi\,:\,{\rm Sp}(G)\rtimes _\alpha\Gamma^\beta\rightarrow{\rm Int}(\Gq)\,:\,(\omega,\gamma)\mapsto u_\gamma\alpha(\omega)\in C_m(\Gq),\quad\omega\in{\rm Sp}(G),\gamma\in\Gamma^\beta.$$
The relations of the crossed product and the group law in the right semi-direct product imply that $\pi$ is a group homomorphism.

Let $(g,\mu)\in G^\alpha\times{\rm Sp}(\Gamma)$. Since $g\in G^\alpha$, the unital $*$-homomorphism $C(G)\rightarrow\C$ given by $F\mapsto F(g)$ and the unitary representation $\mu\,:\,\Gamma\rightarrow S^1$ give a covariant representation. Hence, we get a unique $\rho(g,\mu)\in\chi(\Gq)$ such that $\rho(g,\mu)(u_\gamma\alpha(F))=\mu(\gamma)F(g)$ for all $\gamma\in\Gamma$, $F\in C(G)$. It defines a map $\rho\,:\,G^\alpha \,_\beta\ltimes{\rm Sp}(\Gamma)\rightarrow\chi(\Gq)$ which is obviously injective.

For all $g,h\in G^\alpha$,$\gamma\in\Gamma$ and $F\in C(G)$, one has
\begin{eqnarray*}
(\rho(g,\omega)\cdot\rho(h,\mu))(u_\gamma\alpha(F))&=&
(\rho(g,\omega)\ot\rho(h,\mu))(\Delta_m(u_\gamma\alpha(F)))
=\sum_{r\in\gamma\cdot G}\omega(\gamma)v_{\gamma,r}(g)\mu(r)F(gh)\\
&=&\omega(\gamma)\mu(\beta_g(\gamma))F(gh)=(\rho(gh,\omega\cdot\mu\circ\beta_g))(u_\gamma\alpha(F)).
\end{eqnarray*}
Hence, $\rho$ is a group homomorphism. 

Let us check that $\rho$ is surjective. Let $\chi\in\chi(\Gq)$, then $\chi\circ\alpha\in{\rm Sp}(C(G))$. Let $g\in G$ be such that $\chi(\alpha(F))=F(g)$ for all $F\in C(G)$. Actually $g\in G^\alpha$. Indeed, for all $\gamma\in\Gamma$ and all $F\in C(G)$, one has
$$F(\alpha_{\gamma^{-1}}(g))=\chi(\alpha(\alpha_{\gamma}(F)))=\chi(u_{\gamma}\alpha(F)u_\gamma^*)=\chi(\alpha(F))=F(g);$$
now use the fact that $C(G)$ separates points of $G$ to establish $g\in G^{\alpha}$. 
Define $\omega=(\gamma\mapsto\chi(u_\gamma))\in{\rm Sp}(\Gamma)$. Consequently, $\chi=\rho(g,\omega)$ and $\rho$ is surjective.

Finally, the map $\rho^{-1}\,:\,\chi(\Gq)\rightarrow G^\alpha\,_\beta\ltimes{\rm Sp}(\Gamma)$ is continuous, since $p_1\circ\rho^{-1}\,:\,\chi(\Gq)\rightarrow{\rm Sp}(C(G))=G$ by $\chi\mapsto\chi\circ\alpha$ and $p_2\circ\rho^{-1}\,:\,\chi(\Gq)\rightarrow{\rm Sp}(\Gamma)$ by $\chi\mapsto(\gamma\mapsto\chi(u_\gamma)),$ are obviously continuous, where $p_1$ and $p_2$ are the canonical projections. By compactness, $\rho$ is an homeomorphism.
\end{proof}

%%%%%%%%%%%%%%%%%%%%%%%%%%%%%%End

%%%%%%%%%%%%%%%%%%%%%%%%%%%%%%%%%%%%%%%%
\section{Property $(T)$ and bicrossed product}\label{section-relativeT}
%%%%%%%%%%%%%%%%%%%%%%%%%%%%%%%%%%%%%%%%
This section is dedicated to the relative co-property $(T)$ of the pair $(G,\Gq)$ and Kazhdan
property of the dual of the bicrossed product $\Gq$ constructed in Section 3. The results in this section generalize classical
results on relative property $(T)$ for inclusion of groups of the form $(H, \Gamma\ltimes H)$, where $H$ and $\Gamma$ are discrete groups and $H$ is abelian \cite{CT11}.

%%%%%%%%%%%%%%%%%%%%%%%%%%%%%%%%%%%%%%%%
\subsection{Relative property $(T)$ for compact bicrossed product}
%%%%%%%%%%%%%%%%%%%%%%%%%%%%%%%%%%%%%%%%

\begin{definition}
Let $G$ and $\Gq$ be two compact quantum groups with an injective unital $*$-homomorphism $\alpha\,:\,C_m(G)\rightarrow C_m(\Gq)$ such that $\Delta_{\Gq}\circ\alpha=(\alpha\ot\alpha)\circ\Delta_G$. We say that the pair $(G,\Gq)$ has the relative co-property $(T)$, if  for every representation $\pi\,:\,C_m(\Gq)\rightarrow\mathcal{B}(H)$ we have $\varepsilon_\Gq\prec\pi\implies\varepsilon_G\subset\pi\circ\alpha$.
\end{definition}

Observe that, by \cite[Proposition 2.3]{Ky11}, $\widehat{\Gq}$ has the property $(T)$ in the sense of \cite{Fi10} if and only if the pair $(\Gq,\Gq)$ has the relative co-property $(T)$ (with $\alpha=\id$). Also, if $\Lambda,\Gamma$ are countable discrete groups and $\Lambda<\Gamma$, then the pair $(\widehat{\Lambda},\widehat{\Gamma})$ has the relative co-property $(T)$ if and only if the pair $(\Lambda,\Gamma)$ has the relative property $(T)$ in the classical sense.

Let $(\Gamma,G)$ be a matched pair of a countable discrete group $\Gamma$ and a compact group $G$. Let $\Gq$ be the bicrossed product. In the following result, we characterize the relative co-property $(T)$ of the pair $(G,\Gq)$ in terms of the action $\alpha$ of $\Gamma$ on $C(G)$. This is a non-commutative version of \cite[Theorem 1]{CT11} and the proof is similar. We will use freely the notations and results of Section \ref{section-bicrossed}.

\begin{theorem}\label{ThmRelT}
The following are equivalent:
\begin{enumerate}
\item The pair $(G,\Gq)$ does not have the relative co-property $(T)$.
\item There exists a sequence $(\mu_n)_{n\in\N}$ of Borel probability measures on $G$ such that
\begin{enumerate}
\item $\mu_n(\{e\})=0$ for all $n\in\N$;
\item $\mu_n\rightarrow\delta_e$ weak*;
\item $\Vert\alpha_\gamma(\mu_n)-\mu_n\Vert\rightarrow 0$ for all $\gamma\in\Gamma$.
\end{enumerate}
\end{enumerate}
\end{theorem}

\begin{proof}
For a representation $\pi\,:\,C_m(\Gq)\rightarrow\mathcal{B}(H)$, we have $\varepsilon_G\subset\pi\circ\alpha$ if and only if $K_\pi\neq\{0\}$, where $$K_\pi=\{\xi\in H\,:\,\pi\circ\alpha(F)\xi=F(e)\xi\text{ for all }F\in C(G)\}.$$
Define $\rho=\pi\circ\alpha\,:\,C(G)\rightarrow\mathcal{B}(H)$, and for all $\xi,\eta\in H$, let $\mu_{\xi,\eta}$ be the unique complex Borel  measure on $G$ such that $\int_G Fd\mu_{\xi,\eta}=\langle\rho(F)\xi,\eta\rangle$ for all $F\in C(G)$. Let $\mathcal{B}(G)$ be the collection of Borel subsets of $G$ and $E\,:\,\mathcal{B}(G)\rightarrow\mathcal{B}(H)$ be the projection-valued measure associated to $\rho$ i.e., for all $B\in\mathcal{B}(G)$, the projection $E(B)\in\mathcal{B}(H)$ is the unique operator such that $\langle E(B)\xi,\eta\rangle=\mu_{\xi,\eta}(B)$ for all $\xi,\eta\in H$.

Observe that a vector $\xi\in H$ satisfies $\rho(F)\xi=F(e)\xi$ for all $F\in C(G)$, if and only if $\mu_{\xi,\eta}=\langle\xi,\eta\rangle\delta_e$ for all $\eta\in H$, which in turn is true if and only if $\langle E(\{e\})\xi,\eta\rangle=\langle\xi,\eta\rangle$ for all $\eta\in H$. Hence, $E(\{e\})$ is the orthogonal projection onto $K_\pi$.

$(1)\implies (2)$. Suppose that the pair $(G,\Gq)$ \textit{does not} have the relative co-property $(T)$. Let $\pi\,:\,C_m(\Gq)\rightarrow\mathcal{B}(H)$ be a representation such that $\varepsilon_\Gq\prec\pi$ and $K_\pi=\{0\}$. Hence, $\mu_{\xi,\eta}(\{e\})=\langle E(\{e\})\xi,\eta\rangle=0$ for all $\xi,\eta\in H$.

Since $\varepsilon_\Gq\prec\pi$, let $(\xi_n)_{n\in\N}$ be a sequence of unit vectors in $H$ such that $\Vert\pi(x)\xi_n-\varepsilon_\Gq(x)\xi_n\Vert\rightarrow 0$ for all $x\in C_m(\Gq)$. Define $\mu_n=\mu_{\xi_n,\xi_n}$. Then, we have $\mu_n(\{e\})=0$ for all $n\in\N$.  Since $\mu_n$ is a probability measure, so $
\vert\mu_n(F)-\delta_e(F)\vert=\vert\int_G(F-F(e))d\mu_n\vert\leq\Vert F-F(e)\Vert_{L^1(\mu_n)}\leq\Vert F-F(e)\Vert_{L^2(\mu_n)}$,  
for all $F\in C(G)$. Moreover,
$$\Vert F-F(e)\Vert_{L^2(\mu_n)}^2=\Vert\rho(F-F(e)1)\xi_n\Vert^2=\Vert\pi(\alpha(F))\xi_n-\varepsilon_\Gq(\alpha(F))\xi_n\Vert^2\rightarrow 0.
$$
Hence, $\mu_n\rightarrow \delta_e$ weak*. Finally, for all $\gamma\in\Gamma$ and all $F\in C(G)$, we have:
\begin{align*}
\int_GFd\alpha_\gamma(\mu_n)&=\int_G\alpha_{\gamma^{-1}}(F) d\mu_n=\langle\rho(\alpha_{\gamma^{-1}}(F))\xi_n,\xi_n\rangle=\langle \pi(u_\gamma)^*\rho(F)\pi(u_\gamma)\xi_n,\xi_n\rangle\\
&=\langle\rho(F)\pi(u_\gamma)\xi_n,\pi(u_\gamma)\xi_n\rangle.
\end{align*}
It follows that
\begin{eqnarray*}
\left\vert\int_GFd\alpha_\gamma(\mu_n)-\int_GFd\mu_n\right\vert&=&\left\vert\langle\rho(F)\pi(u_\gamma)\xi_n,\pi(u_\gamma)\xi_n\rangle-\langle\rho(F)\xi_n,\xi_n\rangle\right\vert\\
&\leq&\left\vert\langle\rho(F)(\pi(u_\gamma)\xi_n-\xi_n),\pi(u_\gamma)\xi_n\rangle\right\vert+\left\vert\langle\rho(F)\xi_n,\pi(u_\gamma)\xi_n-\xi_n\rangle\right\vert\\
&\leq&2\Vert F\Vert\,\Vert\pi(u_\gamma)\xi_n-\xi_n\Vert, \text{ for all } F\in C(G) \text{ and }\gamma\in\Gamma.
\end{eqnarray*}
Hence, $\Vert\alpha_\gamma(\mu_n)-\mu_n\Vert\leq 2\Vert\pi(u_\gamma)\xi_n-\xi_n\Vert=2\Vert\pi(u_\gamma)\xi_n-\varepsilon_\Gq(u_\gamma)\xi_n\Vert\rightarrow 0$ $($see $(4)$ of Theorem \ref{ThmBicrossed}$)$.

$(2)\implies(1)$. We first prove the following claim. 

\textbf{Claim.} \textit{If $(2)$ holds, then there exists a sequence $(\nu_n)_{n\in\N}$ of Borel probability measures on $G$ satifying $(a)$, $(b)$ and $(c)$ and such that
$\alpha_\gamma(\nu_n)\sim\nu_n$ for all $\gamma\in\Gamma$, $n\in\N$.}

\textit{Proof of the claim.} Denote by $\ell^1(\Gamma)_{1,+}$ the set of positive $\ell^1$ functions on $\Gamma$ with $\Vert f\Vert_1=1$. For $\mu$ a Borel probability measure on $G$  and $f\in \ell^1(\Gamma)_{1,+}$, define the Borel probability measure $f*\mu$ on $G$ by the convex combination
$$f*\mu=\sum_{\gamma\in\Gamma}f(\gamma)\alpha_\gamma(\mu).$$
Observe that for all $\gamma\in\Gamma$, we have $\delta_\gamma*\mu=\alpha_\gamma(\mu)$ and $\alpha_\gamma(f*\mu)=f_\gamma*\mu$, where $f_\gamma\in \ell^1(\Gamma)_{1,+}$ is defined by $f_\gamma(r)=f(\gamma^{-1}r)$, $r\in\Gamma$.

Moreover, if $f\in \ell^1(\Gamma)_{1,+}$ is such that $f(\gamma)>0$ for all $\gamma\in\Gamma$, then since $(f*\mu)(E)=\sum_\gamma f(\gamma)\mu(\alpha_{\gamma^{-1}}(E))$ $(E$ is Borel subset of $G)$, so we have that $(f*\mu)(E)=0$ if and only if $\mu(\alpha_\gamma(E))=0$ for all $\gamma\in\Gamma$. This last condition does not depend on $f$. Hence, if $f\in \ell^1(\Gamma)_{1,+}$ is such that $f>0$, then since $f_\gamma(r)>0$ for all $\gamma,r\in\Gamma$, it follows that $f*\mu\sim \alpha_\gamma(f*\mu)=f_\gamma*\mu$ for all $\gamma\in\Gamma$ as they have the same null sets: the Borel subsets $E$ of $G$ such that $\mu(\alpha_s(E))=0$ for all $s\in\Gamma$.

Therefore, since $\alpha_\gamma(e)=e$ for all $\gamma\in\Gamma$, so 
\begin{align*}
(f*\mu)(\{e\})=\sum_{\gamma}f(\gamma)\mu(\alpha_{\gamma^{-1}}(\{e\}))=\sum_\gamma f(\gamma)\mu(\{e\})=\mu(\{e\}), \text{ for all }f\in \ell^1(\Gamma)_{1,+}.
\end{align*}
Let $(\mu_n)_{n\in\N}$ be a sequence of Borel probability on $G$ satisfying $(a)$, $(b)$ and $(c)$. For all $f\in \ell^1(\Gamma)_{1,+}$ with finite support we have,
\begin{equation}\label{EqProba}
\Vert f*\mu_n-\mu_n\Vert\leq\sum_\gamma f(\gamma)\Vert \delta_\gamma*\mu_n-\mu_n\Vert=\sum_\gamma f(\gamma)\Vert \alpha_\gamma(\mu_n)-\mu_n\Vert\rightarrow 0.
\end{equation}
Since such functions are dense in $\ell^1(\Gamma)_{1,+}$ (in the $\ell^1$-norm), it follows that $\Vert f*\mu_n-\mu_n\Vert\rightarrow 0$ for all $f\in \ell^1(\Gamma)_{1,+}$.

Let $\xi\in \ell^1(\Gamma)_{1,+}$ be any function such that $\xi>0$ and define $\nu_n=\xi*\mu_n$. By the preceding discussion, we know that $\alpha_\gamma(\nu_n)\sim\nu_n$ for all $\gamma\in\Gamma$ and $\nu_n(\{e\})=\mu_n(\{e\})=0$ for all $n\in\N$. Moreover, by Equation $(\ref{EqProba})$, 
\begin{align*}
\Vert \alpha_\gamma(\nu_n)-\nu_n\Vert=\Vert \xi_\gamma*\mu_n-\xi*\mu_n\Vert\leq\Vert\xi_\gamma*\mu_n-\mu_n\Vert+\Vert\mu_n-\xi*\mu_n\Vert\rightarrow 0, \text{ for all }\gamma\in \Gamma.
\end{align*}

Finally, since $\mu_n\rightarrow\delta_e$ weak* and $\alpha_\gamma(e)=e$, one has $\vert\mu_n(F\circ\alpha_\gamma)-F(e)\vert\rightarrow 0$ for all $\gamma\in\Gamma$  and for all $F\in C(G)$. Hence, for all $F\in C(G)$, the dominated convergence theorem implies that
$$\vert \nu_n(F)-\delta_e(F)\vert=\left\vert\sum_\gamma f(\gamma)(\mu_n(F\circ\alpha_\gamma)-F(e))\right\vert\leq\sum_\gamma f(\gamma)\vert\mu_n(F\circ\alpha_\gamma)-F(e)\vert\rightarrow 0.$$
It follows that $\nu_n\rightarrow\delta_e$ weak* and this finishes the proof of the claim.\hfill{$\Box$} 

We now finish the proof of the Theorem. Let $(\mu_n)_{n\in\N}$ be a sequence of Borel probability measures on $G$ as prescribed in the Claim. For $n\in\N$ and $\gamma\in\Gamma$, let $h_n(\gamma)=\frac{d\alpha_\gamma(\mu_n)}{d\mu_n}$; then $0\leq h_n(\gamma)\leq 1$, $\mu_n$ a.e., and by uniqueness of the Radon-Nikodym derivatives and since $\alpha$ is an action, we have for all $n\in\N$, $h_n(\gamma,g)h_n(\gamma^{-1},\alpha_{\gamma^{-1}}(g))=1$, $\mu_n$ a.e. $g\in G$, and for all $\gamma\in\Gamma$. Define $H_n={\rm L}^2(G,\mu_n)$ and let $u_n\,:\,\Gamma\rightarrow\mathcal{U}(H_n)$ be the unitary representations defined by $(u_n(\gamma)\xi)(g)=\xi(\alpha_{\gamma^{-1}}(g))h_n(\gamma,g)^{\frac{1}{2}}$ for $\gamma\in\Gamma,\,g\in G,\,\xi\in H_n$. Also consider the representations $\rho_n\,:\,C(G)\rightarrow\mathcal{B}(H_n)$, defined by $\rho_n(F)\xi(g)=F(g)\xi(g)$, for $\xi\in H_n$, $g\in G$ and $F\in C(G)$. Observe that the projection valued measure associated to $\rho_n$ is given by $(E_n(B)\xi)(g)=1_B(g)\xi(g)$ for all $B\in\mathcal{B}(G)$, $\xi\in H_n$ and $g\in G$. Using the identity $h_n(\gamma,\cdot)h_n(\gamma^{-1},\alpha_{\gamma^{-1}}(\cdot))=1$, we find $u_n(\gamma)\rho_n(F)u_n(\gamma^{-1})=\rho_n(\alpha_\gamma(F))$ for all $\gamma\in\Gamma,\,F\in C(G),\,g\in G$.
Therefore, by the universal property of $A_m$, for each $n\in \mathbb{N}$ there is a unital $*$-homomorphism $\pi_n\,:\,A_m\rightarrow\mathcal{B}(H_n)$ such that $\pi_n(u_\gamma)=u_n(\gamma)$ and $\pi_n\circ\alpha=\rho_n$ for all $n\in\N$. Since $\mu_n(\{e\})=0$, we have $E_n(\{e\})=0$ for all $n\in\N$. Hence, $K_{\pi_n}=\{0\}$ for all $n\in\N$. Consequently, on defining $H=\oplus_{n} H_n$ and $\pi=\oplus_{n}\pi_n\,:\, C_m(\Gq)\rightarrow\mathcal{B}(H)$, it follows that $K_\pi=\{0\}$ as well. Hence, it suffices to show that $\varepsilon_\Gq\prec\pi$.

Define the unit vectors $\xi_n=1\in {\rm L}^2(G,\mu_n)\subset H$, $n\in \N$. Observe that $(\mu_n-\alpha_\gamma(\mu_n))(F)=\int_GF(1-h_n(\gamma))d\mu_n$ for all $F\in C(G)$. Hence, $\Vert \mu_n-\alpha_\gamma(\mu_n)\Vert=\Vert 1-h_n(\gamma)\Vert_{L^1(G,\mu_n)}\rightarrow 0$ for all $\gamma\in\Gamma$. Moreover, as  $0\leq 1-\sqrt{t}\leq\sqrt{1-t}$ for all $0\leq t\leq 1$, it follows that 
$$\Vert\pi(u_\gamma)\xi_n-\xi_n\Vert_H^2=\Vert u_n(\gamma)1-1\Vert_{H_n}^2=\int_G(1-h_n(\gamma)^{\frac{1}{2}})^2d\mu_n
\leq\int_G(1-h_n(\gamma))d\mu_n=\Vert 1-h_n(\gamma)\Vert_{L^1(G,\mu_n)}\rightarrow 0$$
for all $\gamma\in \Gamma$.  Since $\mu_n\rightarrow\delta_e$ weak*, for all $F\in C(G)$, we also have that, 
$$\Vert\pi(\alpha(F))\xi_n-F(e)\xi_n\Vert_H^2=\Vert\rho_n(F)1-F(e)1\Vert_{H_n}^2=\int_G\vert F(g)-F(e)\vert^2d\mu_n\rightarrow 0.$$
Consequently, for all $x=u_\gamma\alpha(F)\in C_m(\Gq)$, we have
\begin{eqnarray*}
\Vert\pi(x)\xi_n-\varepsilon_\Gq(x)\xi_n\Vert&= &\Vert\pi(u_\gamma)\pi(\alpha(F))\xi_n-F(e)\xi_n\Vert\\
&\leq&\Vert\pi(u_\gamma)(\pi(\alpha(F))\xi_n-F(e)\xi_n)\Vert+\vert F(e)\vert\,\Vert\pi(u_\gamma)\xi_n-\xi_n\Vert\\
&\leq&\Vert\pi(\alpha(F))\xi_n-F(e)\xi_n\Vert+\vert F(e)\vert\,\Vert\pi(u_\gamma)\xi_n-\xi_n\Vert\rightarrow 0.\\
\end{eqnarray*}
By linearity and the triangle inequality, we have $\Vert\pi(x)\xi_n-\varepsilon_\Gq(x)\xi_n\Vert\rightarrow 0$ for all $x\in\mathcal{A}$. The proof is complete by density of $\mathcal{A}$ in $C_m(\Gq)$.
\end{proof}

\subsection{Property (T)}
Now we discuss property $(T)$ of $\Gq$. Let $G^\alpha$ be the set of fixed points in $G$ under the action $\alpha$ of $\Gamma$. It is a closed subset of $G$, and, by the relations in Equation $(\ref{EqMatched})$ it is also a subgroup of $G$.
\begin{theorem}\label{ThmPropT}
The following holds:
\begin{enumerate}
\item If $\widehat{\Gq}$ has property $(T)$, then $\Gamma$ has property $(T)$ and $G^\alpha$ is finite. 
\item If $\widehat{\Gq}$ has property $(T)$ and $\alpha$ is compact\footnote{We only need to assume that the closure of the image of $\Gamma$ in the group of homeomorphisms of $G$ is compact for some Hausdorff group topology for which the evaluation map at $e$ is continuous.} then $\Gamma$ has $(T)$ and $G$ is finite.
\item If $\Gamma$ has property $(T)$ and $G$ is finite, then $\widehat{\Gq}$ has property $(T)$.
\end{enumerate}
\end{theorem}

\begin{proof}
$(1)$. Let $\rho\,:\,C(G)\rightarrow C^*(\Gamma)$ be the unital $*$-homomorphism defined by $\rho(F)=F(e)1$ and consider the canonical unitary representation of $\Gamma$ given by $\Gamma\ni\gamma\mapsto\lambda_\gamma\in C^*(\Gamma)$. For all $\gamma\in\Gamma$ and $F\in C(G)$, we have $\rho(\alpha_\gamma(F))=\alpha_\gamma(F)(e)1=F(\alpha_{\gamma^{-1}}(e))1=F(e)1=\lambda_\gamma\rho(F)\lambda_\gamma^*$. Hence, there exists a unique unital $*$-homomorphism $\pi\,:\,C_m(\Gq)\rightarrow C^*(\Gamma)$ such that $\pi\circ\alpha=\rho$ and $\pi(u_\gamma)=\lambda_\gamma$ for all $\gamma\in\Gamma$.
Observe that $\pi$ is surjective and, for all $F\in C(G)$,
$$(\pi\ot\pi)\Delta_{\Gq}(\alpha(F))=(\rho\ot\rho)(\Delta_G(F))=\Delta_G(F)(e,e)1\ot 1=F(e)1\ot 1=\Delta_{\widehat{\Gamma}}(\pi(\alpha(F))).$$
Moreover, since for all $\gamma,r\in\Gamma$ one has $1_{A_{\gamma,r}}(e)=\delta_{\gamma,r}$, we find, for all $\gamma\in\Gamma$,
$$(\pi\ot\pi)\Delta_{\Gq}(u_\gamma)=\sum_{r\in\gamma\cdot G}\pi(u_\gamma\alpha(v^\gamma_{\gamma,r}))\ot\pi(u_r)
=\sum_{r\in\gamma\cdot G}\lambda_\gamma1_{A_{\gamma,r}}(e)\ot\lambda_r=\lambda_\gamma\ot\lambda_\gamma=\Delta_{\widehat{\Gamma}}(\pi(u_\gamma)).$$
So $\pi$ intertwines the comultiplications and property $(T)$ for $\Gamma$ follows from \cite[Proposition 6]{Fi10}.

To show that $G^\alpha$ is finite it suffices, since $G^\alpha$ is closed in $G$ hence compact, to show that $G^\alpha$ is discrete. Let $(g_n)$ be any sequence in $G^\alpha$ such that $g_n\rightarrow e$. Consider the unital $*$-homomorphism $\rho\,:\,C(G)\rightarrow\mathcal{B}(\ell^2(\N))$ defined by $(\rho(F)\xi)(n)=F(g_n)\xi(n)$, for all $\xi\in \ell^2(\mathbb{N})$, and the trivial representation of $\Gamma$ on $\ell^2(\N)$. Since $g_n\in G^\alpha$ for all $n\in\N$ it gives a covariant representation. Hence, there exists a unital $*$-homomorphism $\pi\,:\,C_m(\Gq)\rightarrow \mathcal{B}(\ell^2(\N))$ such that $\pi(u_\gamma \alpha(F))=\rho(F)$ for all $\gamma\in\Gamma$ and $F\in C(G)$. Define $\xi_n=\delta_n\in \ell^2(\N)$. One has $\Vert\pi(u_\gamma\alpha(F))\xi_n-\varepsilon_\Gq(u_\gamma\alpha(F))\xi_n\Vert=\vert F(g_n)-F(e)\vert\rightarrow 0$ for all $F\in C(G)$.
Hence, $\pi$ has almost invariant vectors. By property $(T)$, $\pi$ has a non-zero invariant vector and for such a vector $\xi\in \ell^2(\N)$ we have $F(g_n)\xi(n)=F(e)\xi(n)$ for all $F\in C(G)$ and all $n\in\N$. Let $n_0\in\N$ for which $\xi(n_0)\neq 0$. We have $F(g_{n_0})=F(e)$ for all $F\in C(G)$, which implies that $g_{n_0}=e$ and shows that $G^\alpha$ must be discrete.

$(2)$. It suffices to show that $G$ is finite. The proof is similar to $(1)$. Let $g_n\in G$ be any sequence such that $g_n\rightarrow e$. We view $\alpha$ as a group homomorphism $\alpha\,:\, \Gamma\rightarrow H(G)$, $\gamma\mapsto\alpha_\gamma$, where $H(G)$ is the group of homeomorphisms of $G$ and we write $K=\overline{\alpha(\Gamma)}\subset H(G)$. By assumptions, $K$ is a compact group and we denote by $\nu$ the Haar probability on $K$. Note that, since $\alpha_\gamma(e)=e$ for all $\gamma\in\Gamma$, by continuity of the evaluation at $e$ and density, we also have $x(e)=e$ for all $x\in K$. We define a covariant representation $(\rho,v)$, $\rho\,:\, C(G)\rightarrow\mathcal{B}({\rm L}^2(K\times\N))$ and $v\,:\,\Gamma\rightarrow\mathcal{U}({\rm L}^2(K\times\N))$ by $(\rho(F)\xi)(x,n)=F(x(g_n))\xi(x,n)$ and $(v_\gamma\xi)(x,n)=\xi(\alpha_{\gamma^{-1}}x,n)$. By the universal property of $C_m(\Gq)$, we get a unital $*$-homomorphism $\pi\,:\, C_m(\Gq)\rightarrow\mathcal{B}({\rm L}^2(K\times\N))$ such that $\pi(u_\gamma\alpha(F))=v_\gamma\rho(F)$ for all $\gamma\in\Gamma$ and $F\in C(G)$. Define, for $k\in\N$,  the vector $\xi_k(x,n)=\delta_{k,n}$. Since $\nu$ is a probability it follows that $\xi_k$ is a unit vector in ${\rm L}^2(K\times\N)$. Moreover, for all $\gamma\in\Gamma$ and $F\in C(G)$,
$$\Vert\pi(u_\gamma\alpha(F))\xi_k-\varepsilon_{\Gq}(u_\gamma\alpha(F))\xi_k\Vert^2=\int_K\vert F(\alpha_{\gamma^{-1}}x(g_k))-F(e)\vert^2d\nu(x)\rightarrow 0,$$
where the convergence follows from the dominated convergence Theorem since, by continuity, we have $F(\alpha_{\gamma^{-1}})x(g_k))\rightarrow F(e)$ for all $\gamma\in\Gamma$, $x\in K$ and $F\in C(G)$ and the domination is obvious since $\nu$ is a probability. By property $(T)$, there exists a non-zero $\xi\in{\rm L}^2(K\times\N))$ such that $F(e)\xi=\varepsilon_\Gq(\alpha(F))\xi=\pi(\alpha(F))\xi=\rho(F)\xi$ for all $F\in C(G)$. Define $Y:=\{x\in K\,:\,\sum_{n\in\N}\vert\xi(x,n)\vert^2>0\}$ and, for $F\in C(G)$, $X_F:=\{x\in K\,:\,\sum_{n\in\N}\vert F(x(g_n))\xi(x,n)-F(e)\xi(x,n)\vert^2\neq 0\}$. The condition on $\xi$ means that $\nu(Y)>0$ and, for all $F\in C(G)$, $\nu(X_F)=0$. Let $F_k\in C(G)$ be a dense sequence and $X=\cup_{k\in\N} X_{F_k}$ then $\nu(X)=0$ so $\nu(Y\setminus X)>0$. Hence, $Y\setminus X\neq \emptyset$. Let $x\in Y\setminus X$, we have $\sum_n\vert\xi(x,n\vert^2>0$ and, for all $k,n\in\N$, $F_k(x(g_n))\xi(x,n)=F_k(e)\xi(x,n)$. By density and continuity, $F(x(g_n))\xi(x,n)=F(e)\xi(x,n)$ for all $n\in\N$ and $F\in C(G)$. Since $\sum_n\vert\xi(x,n\vert^2>0$, there exists $n_0\in\N$ such that $\xi(x,n_0)\neq 0$ which implies that $F(x(g_{n_0}))=F(e)$ for all $F\in C(G)$. Hence, $x(g_{n_0})=e$ which implies that $g_{n_0}=e$. Hence $G$ must be discrete and, by compactness, $G$ is finite.

$(3)$. Let  $\pi\,:\,C_m(\Gq)=\Gamma{}_{\alpha,f}\ltimes C(G)\rightarrow\mathcal{B}(H)$ be a unital $*$-homomorphism and $K$ be the closed subspace $H$ given by $C(G)$-invariant vectors i.e.
$K=\{\xi\in H\,:\,\pi\circ\alpha(F)\xi=F(e)\xi\text{ for all }F\in C(G)\}.$
Then $P=\pi(\alpha(\delta_e))$ is the orthogonal projection onto $K$ which is an invariant subspace of the unitary representation $\gamma\mapsto\pi(u_\gamma)$ since $\pi(u_\gamma)P\pi(u_{\gamma})^*=\pi(\alpha(\delta_{\alpha_g(e)}))=\pi(\alpha(\delta_e))=P$ for all $\gamma\in\Gamma$. Let $\gamma\mapsto v_\gamma$ be the unitary representation of $\Gamma$ on $K$ obtained by restriction. 

Suppose that $\varepsilon_\Gq\prec\pi$  and let $\xi_n\in H$ be a sequence of unit vectors such that $\Vert\pi(x)\xi_n-\varepsilon_\Gq(x)\xi_n\Vert\rightarrow 0$ for all $x\in C_m(\Gq)$. Since $G$ is finite (hence $\widehat{G}$ has property $(T)$), so $K\neq\{0\}$. Moreover, since $\vert\,\Vert P\xi_n\Vert-1\vert\leq\Vert P\xi_n-\xi_n\Vert$, we have $\Vert P\xi_n\Vert\rightarrow 1$ and hence we may and will assume that $P\xi_n\neq 0$ for all $n$. Let $\eta_n=\frac{P\xi_n}{\Vert P\xi_n\Vert}\in K$. We have $\Vert v_\gamma\eta_n-\eta_n\Vert=\frac{1}{\Vert P\xi_n\Vert}\Vert P(v_\gamma\xi_n-\xi_n)\Vert\leq\frac{\Vert\pi(u_\gamma)\xi_n-\xi_n\Vert}{\Vert P\xi_n\Vert}\rightarrow 0$. Hence, $\gamma\mapsto v_\gamma$ has almost invariant vectors. Since $\Gamma$ has property $(T)$, let $\xi\in K$ be a non-zero invariant vector. Then, for all $x\in C_m(\Gq)$ of the form $x=u_\gamma\alpha(F)$, we have $\pi(x)\xi=F(e)\pi(u_\gamma)\xi=F(e)\xi=\varepsilon_\Gq(x)\xi$. By linearity, continuity, and density of $\mathcal{A}$ in $C_m(\Gq)$, we have $\pi(x)\xi=\varepsilon_\Gq(x)\xi$ for all $x\in C_m(\Gq)$.
\end{proof}
We mention that the third assertion of the previous theorem appears in \cite{CN15} when $\beta$ is supposed to be the trivial action.

\begin{remark} The compactness assumption on $\alpha$ in assertion $2$ of the preceding Corollary can not be removed. Indeed, for $n\geq 3$, the semi-direct product $H={\rm SL}_n(\Z)\ltimes\Z^n$ (for the linear action of ${\rm SL}_n(\Z)$ on $\Z^n$) has property $(T)$ and $H$ may be viewed as the dual of the bicrossed product associated to the matched pair $({\rm SL}_n(\Z),\mathbb{T}^n)$ with the non-compact action $\alpha\,:\,{\rm SL}_n(\Z)\curvearrowright\mathbb{T}^n$ given by viewing $\mathbb{T}^n=\widehat{\Z^n}$ and dualizing the linear action ${\rm SL}_n(\Z)\curvearrowright\mathbb{Z}^n$ and the action $\beta$ being trivial. In this example, the compact group $G=\mathbb{T}^n$ is infinite.
\end{remark}

%%%%%%%%%%%%%%%%%%%%%%%%%%%%%%%%%%%%%%%%%%%%%%%%
\section{Relative Haagerup property and bicrossed product}\label{section-relativeH}
%%%%%%%%%%%%%%%%%%%%%%%%%%%%%%%%%%%%%%%%%%%%%%%%

In this section, we study the relative co-Haagerup property of the pair $(G,\Gq)$ constructed in Section 3. The main result in this section also generalizes the characterization of relative Haagerup property of the pair $(H, \Gamma\ltimes H)$, where $H$ and $\Gamma$ are discrete groups and $H$ is abelian \cite{CT11}. We refer to Section \ref{section-AP} for the definitions of the Fourier transform and the Haagerup property.

\begin{definition}
Let $G$ and $\Gq$ be two compact quantum groups with an injective unital $*$-homomor-\\ phism $\alpha\,:\,C_m(G)\rightarrow C_m(\Gq)$ such that $\Delta_{\Gq}\circ\alpha=(\alpha\ot\alpha)\circ\Delta_G$. We say that the pair $(G,\Gq)$ has the \textit{relative co-Haagerup property}, if there exists a sequence of states $\omega_n\in C_m(\Gq)^*$ such that $\omega_n\rightarrow\varepsilon_\Gq$ in the weak* topology and $\widehat{\omega_n\circ\alpha}\in c_0(\widehat{G})$ for all $n\in\N$.
\end{definition}

Observe that, for any compact quantum group $G$, the dual $\widehat{G}$ has the Haagerup property if and only if the pair $(G,G)$ has the co-Haagerup property. Moreover, it is clear that if $\Lambda,\Gamma$ are discrete groups with $\Lambda<\Gamma$, then the pair $(C^*(\Lambda),C^*(\Gamma))$ has the relative co-Haagerup property if and only if the pair $(\Lambda,\Gamma)$ has the relative Haagerup property in the classical sense.

Let $(\Gamma,G)$ be a matched pair of a discrete group $\Gamma$ and a compact group $G$. Let $\Gq$ be the bicrossed product. In the following theorem, we characterize the relative co-Haagerup property of the pair $(G,\Gq)$ in terms of the action $\alpha$ of $\Gamma$ on $C(G)$. This is a non commutative version of \cite[Theorem 4]{CT11} and the proof is similar in spirit.
However, one of the argument of the classical case does not work in our context since $\alpha_\gamma$ is not a group homomorphism and substitutive ideas are required. Actually, for a general automorphism $\pi\in{\rm Aut}(C(G))$, there is no guarantee that $\widehat{\nu}\in C^*_r(G)\Rightarrow\widehat{\pi(\nu)}\in C^*_r(G)$. However, in the event of automorphisms coming from the action $\alpha$ given by a matched pair the aforesaid statement turns out to be true. We provide details of this idea in the next lemma. We will freely use the notations and results of Section \ref{section-bicrossed}.

\begin{lemma}\label{LemC0}
Let $\nu$ be a complex Borel measure on $G$. If $\widehat{\nu}\in C^*_r(G)$, then $\widehat{\alpha_\gamma(\nu)}\in C^*_r(G)$ for all $\gamma\in\Gamma$.
\end{lemma}

\begin{proof}
For $\gamma\in\Gamma$ define $G_\gamma={\rm Stab}_G(\gamma):=\{g\in G\,:\,\beta_g(\gamma)=\gamma\}$. Note that $G_\gamma$ is a compact open subgroup of $G$ with index $[G:G_\gamma]=\vert\gamma\cdot G\vert$. For $gG_\gamma\in G/G_\gamma$ we denote by $E_{g G_\gamma}$ the completely bounded map $E_{g G_\gamma}:=(\id\ot\omega_{1_{G_\gamma},1_{g G_\gamma}})\Delta_{\widehat{G}}\,:\, C^*_r(G)\rightarrow M(C^*_r(G))$, where $\Delta_{\widehat{G}}$ is the comultiplication on $C^*_r(G)$, for $K\subset G$ a borel set, $1_K\in {\rm L}^2(G)$ denotes the characteristic function of $K$ and, for $\xi,\eta\in {\rm L}^2(G)$, $\omega_{\xi,\eta}\in\mathcal{B}({\rm L}^2(G))^*$ denotes the functional $T\mapsto\langle T\xi,\eta\rangle$. For $F\in C(G)$, we denote by $\lambda(F):=\int_GF(x)\lambda_xd\mu(x)\in C^*_r(G)$ the convolution operator by $G$. Note that, for all $g,x\in G$ and $\gamma\in \Gamma$ one has $\langle \lambda_x 1_{G_\gamma},1_{g G_\gamma}\rangle=\mu(xG_\gamma\cap g G_\gamma)=\mu(G_\gamma)1_{g G_\gamma}(x)$ hence, for all $F\in C(G)$ and for all Borel complex measure $\nu$ on $G$, one has
$$
E_{g G_\gamma}(\lambda(F))%=\int_GF(x)\langle \lambda_x 1_{G_\gamma},1_{g G_\gamma}\rangle\lambda_xd\mu(x)
=\mu(G_\gamma)\int_{g G_\gamma} F(x)\lambda_x d\mu(x)=\mu(G_\gamma)\lambda(1_{gG_\gamma}F)\text{ and }E_{g G_\gamma}(\widehat{\nu})=\mu(G_\gamma)\int_{g G_\gamma}\lambda_xd\nu(x).
$$
The formula $E_{g G_\gamma}(\lambda(F))=\mu(G_\gamma)\lambda(1_{gG_\gamma}F)$ implies in particular that $E_{g G_\gamma}$ is actually a cb map from $C^*_r(G)$ to $C^*_r(G)$.

%Let $\nu$ be any complex Borel measure on $G$. For $g\in G$, we denote by $g\cdot\nu$ the borel complex measure defined by $g\cdot\nu(A)=\nu(g^{-1}A)$. Note that $\int_GF(x)d(g\cdot\nu)(x)=\int_G F(gx)d\nu(x)$ for all $F\in C(G)$. It follows that $\widehat{g\cdot\nu}=\int_G\lambda_{g^{-1}x}d\nu=\lambda_g^*\widehat{\nu}$. In particular $\widehat{\nu}\in C^*_r(G)\Leftrightarrow \widehat{g\cdot\nu}\in C^*_r(G)$ for all $g\in G$.

Let $\xi,\eta\in {\rm L}^2(G)$. For $x\in g G_\gamma$ one has, since $\mu$ is $\alpha$-invariant,
$$
\langle\lambda_{\alpha_\gamma(x)}\xi,\eta\rangle=\int_G\xi(\alpha_\gamma(x)^{-1}y)\overline{\eta(y)}d\mu(y)
=\int_G\xi(\alpha_\gamma(x)^{-1}\alpha_\gamma(y))\overline{\eta(\alpha_\gamma(y))}d\mu(y).
$$
By the bicrossed-product relations we have, for all $g\in G$ and $x\in gG_\gamma$, $\beta_x(\gamma)=\beta_g(\gamma)$ and $\alpha_{\beta_g(\gamma)}(x^{-1}y)=\alpha_{\beta_x(\gamma)}(x^{-1})\alpha_\gamma(y)=\alpha_\gamma(x)^{-1}\alpha_\gamma(y)$ for all $y\in G$ hence, for all $g\in G$ and $x\in gG_\gamma$,
\begin{eqnarray*}
\langle\lambda_{\alpha_\gamma(x)}\xi,\eta\rangle&=&\int_G\xi\circ\alpha_{\beta_g(\gamma)}(x^{-1}y)\overline{\eta\circ\alpha_\gamma(y)}d\mu(y)= \langle \lambda_x\xi\circ\alpha_{\beta_g(\gamma)},\eta\circ\alpha_\gamma\rangle=\langle w_\gamma\lambda_x w_{\beta_g(\gamma)}^*\xi,\eta\rangle,
\end{eqnarray*}
where $\gamma\mapsto w_\gamma$ is the unitary representation of $\Gamma$ on ${\rm L}^2(G)$ defined by $w_\gamma\xi=\xi\circ\alpha_{\gamma}^{-1}$. It follows that, for all $\gamma\in\Gamma$, $g\in G$ and $x\in g G_\gamma$, $\lambda_{\alpha_\gamma(x)}=w_\gamma\lambda_x w_{\beta_g(\gamma)}^*$ and, for any complex Borel measure $\nu$ on $G$,
\begin{eqnarray*}
\widehat{\alpha_\gamma(\nu)}&=&\int_G\lambda_{\alpha_\gamma(x)}d\nu(x)=\sum_{g G_\gamma\in G/G_\gamma}\int_{g G_\gamma}\lambda_{\alpha_\gamma(x)}d\nu(x)=\sum_{g G_\gamma\in G/G_\gamma}\int_{g G_\gamma}w_\gamma\lambda_x w_{\beta_g(\gamma)}^*d\nu(x)\\
&=&\frac{1}{\mu(G_\gamma)}w_\gamma\sum_{g G_\gamma\in G/G_\gamma}E_{g G_\gamma}(\widehat{\nu})w_{\beta_g(\gamma)}^*.
\end{eqnarray*}
Since $\mu$ is $\alpha$-invariant, the same computation shows that, for any $F\in C(G)$, one has
\begin{eqnarray*}
\lambda(F\circ\alpha_{\gamma}^{-1})&=&\int_GF(\alpha_\gamma^{-1}(x))\lambda_xd\mu(x)=\int_GF(x)\lambda_{\alpha_\gamma(x)}d\mu(x)
=\sum_{g G_\gamma\in G/G_\gamma}\int_{g G_\gamma}F(x)w_\gamma\lambda_x w_{\beta_g(\gamma)}^*d\mu(x)\\
&=&\frac{1}{\mu(G_\gamma)}w_\gamma\sum_{g G_\gamma\in G/G_\gamma}E_{g G_\gamma}(\lambda(F))w_{\beta_g(\gamma)}^*.
\end{eqnarray*}
Now, suppose that $\nu$ is a complex Borel measure on $G$ such that $\widehat{\nu}\in C^*_r(G)$ and let $F_n\in C(G)$ be a sequence such that $\lambda(F_n)\rightarrow\widehat{\nu}$ in norm. By continuity $E_{g G_\gamma}(\lambda(F_n))\rightarrow E_{g G_\gamma}(\widehat{\nu})$ for all $g\in G$. Hence, it follows from the computations above that $\lambda(F_n\circ\alpha_{\gamma}^{-1})\rightarrow\widehat{\alpha_\gamma(\nu)}$ in norm so that $\widehat{\alpha_\gamma(\nu)}\in C^*_r(G)$.
\end{proof}

\begin{theorem}\label{ThmRelH}
The following are equivalent:
\begin{enumerate}
\item The pair $(G,\Gq)$ has the relative co-Haagerup property.
\item There exists a sequence $(\mu_n)_{n\in\N}$ of Borel probability measures on $G$ such that
\begin{enumerate}
\item $\widehat{\mu}_n\in C^*_r(G)$ for all $n\in\N$;
\item $\mu_n\rightarrow\delta_e$ weak*;
\item $\Vert\alpha_\gamma(\mu_n)-\mu_n\Vert\rightarrow 0$ for all $\gamma\in\Gamma$.
\end{enumerate}
\end{enumerate}
\end{theorem}

\begin{proof}
$(1)\Longrightarrow (2)$. Let $\omega_n\in C_m(\Gq)^{*}$ be a sequence of states such that $\omega_n\rightarrow\varepsilon_\Gq$ in the weak* topology and $\widehat{\omega_n\circ\alpha}\in C^*_r(G)$. For each $n$ view $\omega_{n}\circ\alpha\in C(G)^*$ as a Borel probability measure $\mu_n$ on $G$. By hypothesis, $\widehat{\mu}_n\in C^*_r(G)$ for all $n\in\N$ and $\mu_n\rightarrow\delta_{e}$ in the weak* topology. Writing $(H_n,\pi_n,\xi_n)$ the GNS construction of $\omega_n$ and doing the same computation as in the proof of $(1)\Longrightarrow (2)$ of Theorem \ref{ThmRelT}, we find $\left\vert\int_GFd\alpha_\gamma(\mu_n)-\int_GFd\mu_n\right\vert\leq\Vert F\Vert\,\Vert\pi_n(u_\gamma)\xi_n-\xi_n\Vert=\Vert F\Vert\sqrt{2(1-{\rm Re}(\omega_n(u_\gamma))}$. Hence, $\Vert\alpha_\gamma(\mu_n)-\mu_n\Vert\leq \sqrt{2(1-{\rm Re}(\omega_n(u_\gamma))}\rightarrow \sqrt{2(1-{\rm Re}(\varepsilon_\Gq(u_\gamma))}=0$.

$(2)\implies(1)$. We first prove the following claim.

\textbf{Claim.} \textit{If $(2)$ holds, then there exists a sequence $(\nu_n)_{n\in\N}$ of Borel probability measures on $G$ satifying $(a)$, $(b)$ and $(c)$ and such that
$\alpha_\gamma(\nu_n)\sim\nu_n$ for all $\gamma\in\Gamma$, $n\in\N$.}

\textit{Proof of the claim.} By the proof of the claim in Theorem \ref{ThmRelT}, it suffices to check that whenever $\nu$ is a complex Borel measure on $G$ and $f\in \ell^1(G)$, we have $\widehat{\nu}\in C^*_r(G)\Rightarrow \widehat{f*\nu}\in C^*_r(G)$. 

Now suppose that $\widehat{\nu}\in C^*_r(G)$ and $f\in c_c(\Gamma)$, then $f*\nu=\sum f(\gamma)\alpha_\gamma(\mu)$ is a finite sum and by Lemma \ref{LemC0} we find that $\widehat{f*\mu}=\sum f(\gamma)\widehat{\alpha_\gamma(\mu)}\in C^*_r(G)$.

Suppose that $\widehat{\nu}\in C^*_r(G)$ and $f\in \ell^1(\Gamma)$. Let $f_n\in c_c(\Gamma)$ be such that $\Vert f-f_n\Vert_1\rightarrow 0$. Since for all $g\in \ell^1(\Gamma)$ and all $\nu\in C(G)^*$ the estimate $\Vert f*\nu\Vert\leq\Vert f\Vert_1\,\Vert\nu\Vert$ hold, we find
$$\Vert\widehat{f*\nu}-\widehat{f_n*\nu}\Vert_{\mathcal{B}({\rm L}^2(G))}=\Vert\widehat{(f-f_n)*\nu}\Vert_{\mathcal{B}({\rm L}^2(G))}\leq\Vert(f-f_n)*\nu\Vert_{C(G)^*}\leq\Vert\nu\Vert_{C(G)^*}\Vert f-f_n\Vert_1\rightarrow 0.$$
Consequently, as $\widehat{f_n*\nu}\in C^*_r(G)$ for all $n$, it follows that $\widehat{f*\nu}\in C^*_r(G)$.\hfill{$\Box$}

\vspace{0.2cm}

We can now finish the proof of the Theorem. Let $(\mu_n)_{n\in\N}$ be a sequence of Borel probability measures on $G$ as in the Claim. As in the proof of Theorem \ref{ThmRelT}, we construct a representation $\pi\,:\, C_m(\Gq)\rightarrow\mathcal{B}(H)$ with a sequence of unit vector $\xi_n\in H$ such that $\Vert\pi(x)\xi_n-\varepsilon_\Gq(x)\xi_n\Vert\rightarrow 0$ for all $x\in C_m(\Gq)$ and $\int Fd\mu_n=\omega_{\xi_n}\circ\pi\circ\alpha(F)$, for all $F\in C(G)$. It follows that the sequence of states $\omega_n=\omega_{\xi_n}\circ\pi\in C_m(\Gq)^*$ satisfies $\omega_n\rightarrow\varepsilon_\Gq$ weak* and $\widehat{\omega_n\circ\alpha}=\widehat{\mu_n}\in C_r^*(G)$ for all $n\in\N$.
\end{proof}

%%%%%%%%%%%%%%%%%%%%%%%%%%%%%%
\section{Crossed product quantum group}\label{section-crossed} 
%%%%%%%%%%%%%%%%%%%%%%%%%%%%

This section deals with a matched pair of a discrete group and a compact quantum group that arises in a crossed product, where the discrete group acts on the compact quantum group via quantum automorphisms. This section is longer and has four subsections. First, we analyze the
quantum group structure and the representation theory of such crossed products which was initially studied by Wang in \cite{Wa95b}, but unlike Wang we do not rely on free products which allows us to shorten the proofs. We also obtain some obvious consequences related to amenability and $K$-amenability and the computation of the intrinsic group and the spectrum of the full C*-algebra of a crossed product quantum group. The subsections deal with weak amenability, rapid decay, (relative) property $(T)$ and (relative) Haagerup property.

Let $G$ be a compact quantum group, $\Gamma$ a discrete group acting on $G$ i.e., $\alpha\,:\,\Gamma\curvearrowright G$ be an action by quantum automorphisms. We will denote by the same symbol $\alpha$ the action of $\Gamma$ on $C_m(G)$ or $C(G)$. Let $A_m=\Gamma{}_{\alpha,m}\ltimes C_m(G)$ be the full crossed product and $A=\Gamma{}_{\alpha}\ltimes C(G)$ be the reduced crossed product. By abuse of notation, we still denote by $\alpha$ the canonical injective map from $C_m(G)$ to $A_m$ and from $C(G)$ to $A$. We also denote by $u_\gamma$, for $\gamma\in\Gamma$, the canonical unitaries viewed in either $A_m$ or $A$. This will be clear from the context and cause no confusion. 

By the universal property of the full crossed product, we have a unique surjective unital $*$-homomorphism $\lambda\,:\, A_m\rightarrow A$ such that $\lambda(u_\gamma)=u_\gamma$ and $\lambda(\alpha(a))=\alpha(\lambda_G(a))$ for all $\gamma\in\Gamma$ and for all $a\in C_m(G)$. Finally, we denote by $\omega\in A^*$, the dual state of $h_G$ i.e., $\omega$ is the unique (faithful) state such that
$$\omega(u_\gamma\alpha(a))=\delta_{e,\gamma}h_G(a)\quad\text{for all }a\in C(G),\gamma\in\Gamma.$$

Again by the universal property of the full crossed product, there exists a unique unital $*$-homomorphism $\Delta_m\,:\, A_m\rightarrow A_m\ot A_m$ such that $\Delta_m(u_\gamma)=u_\gamma\ot u_\gamma$ and $\Delta_m\circ\alpha=(\alpha\ot\alpha)\circ\Delta_G$.

The following theorem is due to Wang \cite{Wa95b}. We include a short proof. 

\begin{theorem}\label{ThmCrossedProduct}
$\Gq=(A_m,\Delta_m)$ is a compact quantum group and the following holds.
\begin{enumerate}
\item The Haar state of $\Gq$ is $h=\omega\circ\lambda$, hence, $\Gq$ is Kac if and only if $G$ is Kac.
\item For all $\gamma\in\Gamma$ and all $x\in{\rm Irr}(G)$, $u^x_\gamma=(1\ot u_\gamma) (\id\ot\alpha)(u^x)\in\mathcal{B}(H_x)\ot A_m$ is an irreducible representation of $\Gq$ and the set $\{u^x_\gamma\,:\,\gamma\in\Gamma, x\in{\rm Irr}(G)\}$ is a complete set of irreducible representations of $\Gq$.
\item One has $C_m(\Gq)=A_m$, $C(\Gq)= A$, ${\rm Pol}(\Gq)={\rm Span}\{u_\gamma\alpha(a)\,:\,\gamma\in\Gamma,a\in{\rm Pol}(G)\}$, $\lambda$ is the canonical surjection from $C_m(\Gq)$ to $C(\Gq)$ and ${\rm L}^\infty(\Gq)$ is the von Neumann algebraic crossed product.
\end{enumerate}
\end{theorem}

\begin{proof}
$(1)$. Write $\mathcal{A}={\rm Span}\{u_\gamma\alpha(a)\,:\,\gamma\in\Gamma,a\in{\rm Pol}(G)\}$. Since, by definition of $A_m$, $\mathcal{A}$ is dense in $A_m$ it suffices to show the invariance of $h$ on $\mathcal{A}$ and one has
\begin{align*} 
(\id\ot h)(\Delta_m(u_\gamma\alpha(u^x_{ij})))&=\sum_ku_\gamma\alpha(u^x_{ik})h(u_\gamma \alpha(u^x_{kj}))=\delta_{\gamma, e}\delta_{x, 1}\\&=h(u_\gamma\alpha(u^x_{ij}))=(h\ot\id)(\Delta_m(u_\gamma\alpha(u^x_{ij}))), \text{ }\gamma\in \Gamma, x\in \text{Irr}(G).
\end{align*}

$(2)$. By the definition of $\Delta_m$, it is obvious that $u^x_\gamma$ is a unitary representation of $\Gq$ for all $\gamma\in\Gamma$ and $x\in{\rm Irr}(G)$. The representations $u_\gamma^x$, for $\gamma\in\Gamma$ and $x\in{\rm Irr}(G)$, are irreducible and pairwise non-equivalent since
\begin{align*}
h(\chi(u^x_r)^*\chi(u_s^y))&=h(\alpha(\chi(\overline{x}))u_{r^{-1}s}\alpha(\chi(y)))=h(u_{r^{-1}s}\alpha(\alpha_{r^{-1}s}(\chi(\overline{x}))\chi(y)))=\delta_{r,s}h_G(\chi(\overline{x})\chi(y))\\
&=\delta_{r,s}\delta_{x,y}.
\end{align*}
Finally, $\{u_\gamma^x\,:\,\gamma\in\Gamma,\,x\in{\rm Irr}(G)\}$ is a complete set of irreducibles since the linear span of the coefficients of the $u_\gamma^x$ is $\mathcal{A}$, which is dense in $C_m(G)$.

$(3)$. We established in $(2)$ that $\mathcal{A}={\rm Pol}(\Gq)$. Since, by definition, $A_m$ is the enveloping C*-algebra of $\mathcal{A}$, we have $C_m(\Gq)=A_m$. Since $\lambda\,:\,A_m\rightarrow A$ is surjective and $\omega$ is faithful on $A$, we have $C(\Gq)=A$. Moreover, since $\lambda$ is identity on $\mathcal{A}={\rm Pol}(G)$, it follows that $\lambda$ is the canonical surjection. Finally, ${\rm L}^\infty(\Gq)$ is, by definition, the bicommutant of $C(\Gq)=A$ which is also the von Neumann algebraic crossed product.
\end{proof}

\begin{remark}\label{RmkAnti}
Observe that the counit satisfies $\varepsilon_{\Gq}(u_{\gamma}\alpha(a))=\varepsilon_{G}(a)$ for any $\gamma\in \Gamma$ and $a\in$ Pol$(G)$. This follows from the uniqueness of the counit with respect to the equation $(\varepsilon\otimes\id)\circ\Delta=\id=(\id\otimes \varepsilon)\circ\Delta$ and also the fact that $\varepsilon_G\circ \alpha_\gamma(a)=\varepsilon_G(a)$, for any $\gamma\in \Gamma$ and $a\in$ Pol$(G)$. Similarly, $S_\Gq(u_\gamma\alpha(a))=u_{\gamma^{-1}}\alpha(S_{G}(\alpha_{\gamma^{-1}}(a)))$. Hence, for any $\gamma\in \Gamma$, we have $\alpha_\gamma\circ S_G=S_G\circ \alpha_\gamma$.   
\end{remark}

\begin{remark}\label{RmkFusion}

From Section \ref{SectionCQG}, we have a group homomorphism $\Gamma\rightarrow S({\rm Irr}(G))$, $\gamma\mapsto\alpha_\gamma$, where $\alpha_\gamma(x)$, for $x\in{\rm Irr}(G)$, is the class of the irreducible representation $(\id\ot\alpha_{\gamma})(u^x)$. Let $\gamma\cdot x\in{\rm Irr}(\Gq)$ be the class of $u_\gamma^x$. Observe that, we have $\gamma\ot x\ot\gamma^{-1}=\alpha_\gamma(x)$ and $\gamma\cdot x=\gamma\ot x$, by viewing $\Gamma\subset{\rm Irr}(\Gq)$ and ${\rm Irr}(G)\subset{\rm Irr}(\Gq)$. Hence, the fusion rules of $\Gq$ are described as follows:
\begin{align*}
r\cdot x\ot s\cdot y=rs\cdot \alpha_{s^{-1}}(x)\ot y=\bigoplus_{\underset{t\subset\alpha_{s^{-1}}(x)\ot y}{t\in{\rm Irr}(G)}}r s\cdot t, \quad\text{ for all }r,s\in\Gamma,\,\,x,y\in{\rm Irr}(G).
\end{align*}
Moreover, we have $\overline{\gamma\cdot x}=\gamma^{-1}\cdot \alpha_\gamma(\overline{x})$ for all $\gamma\in\Gamma$ and $x\in{\rm Irr}(G)$.
\end{remark}

\begin{corollary}\label{Amenablestuff}
The following hold.
\begin{enumerate}
\item $\Gq$ is co-amenable if and only if $G$ is co-amenable and $\Gamma$ is amenable.
\item If $G$ is co-amenable and $\Gamma$ is K-amenable, then $\widehat{\Gq}$ is K-amenable.
\end{enumerate}
\end{corollary}

\begin{proof}
$(1)$. Let $G$ be co-amenable and $\Gamma$ be amenable. Then as $C_m(G)=C(G)$ and since the full and the reduced crossed products are the same for actions of amenable groups, it follows from the previous theorem that $\Gq$ is co-amenable. Now, if $\Gq_m$ is co-amenable, its Haar state is faithful on $A_m$. In particular, $h\circ\lambda\circ\alpha=h_G\circ\lambda_G$ must be be faithful on $C_m(G)$ which implies that $G$ is co-amenable. Since $h(u_\gamma)=\delta_{\gamma,e}$, $\gamma\in \Gamma$, we conclude, from Remark \ref{RmkFull} (since the counit $\varepsilon_G$ is an $\alpha$ invariant character on $C_m(G)$), that the canonical trace on $C^*(\Gamma)$ has to be faithful. Hence, $\Gamma$ is amenable.

$(2)$. Follows from \cite[Theorem 2.1 (c)]{Cu83} since $C_m(G)=C(G)$.
\end{proof}

Note that, from the action $\alpha\,:\,\Gamma\curvearrowright C_m(G)$ by quantum automorphisms, we have a natural action, still denoted $\alpha$, of $\Gamma$ on $\chi(G)$ by group automorphisms and homeomorphisms. The set of fixed points $\chi(G)^\alpha=\{\chi\in\chi(G),:\,\chi\circ\alpha_\gamma=\chi\text{ for all }\gamma\in\Gamma\}$ is a closed subgroup. Also note that we have a natural action by group automorphisms, still denoted $\alpha$, of $\Gamma$ on ${\rm Int}(G)$.

\begin{proposition}\label{Prop-CrossedInt}
There are canonical group isomorphisms:
$${\rm Int}(\Gq)\simeq\Gamma\,_\alpha\ltimes{\rm Int}(G)\quad\text{and}\quad\chi(\Gq)\simeq\chi(G)^\alpha\times{\rm Sp}(\Gamma).$$
The second one is moreover an homeomorphism.
\end{proposition}

\begin{proof}
The proof is the same as the proof of Proposition \ref{PropInt}. The dimension of the irreducible representation $(\id\ot\alpha)(u^x)(1\ot u_\gamma)$ is equal to the dimension of $x$ and such representations, for $x\in{\rm Irr}(G)$ and $\gamma\in \Gamma$, form a complete set of irreducibles of $\Gq$. Hence we get a bijection
$$\pi\,:\,\Gamma\,_\alpha\ltimes{\rm Int}(G)\rightarrow {\rm Int}(\Gq)\,:\,(\gamma,u)\mapsto \alpha(u)u_\gamma\in C_m(\Gq).$$
Moreover, the relations in the crossed product and the group law in the semi-direct product imply that it is a group homomorphism.

\vspace{0.2cm}

Let $(\chi,\mu)\in\chi(G)^\alpha\times{\rm Sp}(\Gamma)$. Since $\chi\circ\alpha_\gamma=\chi$ for all $\gamma\in\Gamma$, the pair $(\chi,\mu)$ gives a covariant representation in $\C$, hence a unique character $\rho(\chi,\mu)\in\chi(\Gq)$ such that $\rho(\chi,\mu)(u_\gamma\alpha(a))=\mu(\gamma)\chi(a)$ for all $\gamma\in\Gamma$, $a\in C_m(G)$. It defines a map $\rho\,:\,\chi(G)^\alpha\times{\rm Sp}(\Gamma)\rightarrow\chi(\Gq)$ which is obviously injective. A direct computation shows that $\rho$ is a group homomorphism. Let us show that $\rho$ is surjective. Let $\omega\in\chi(\Gq)$, then $\chi:=\omega\circ\alpha\in \chi(G)$ and, for all $a\in C_m(G)$, $\chi\circ\alpha_\gamma(a)=\omega(u_\gamma\alpha(a)u_\gamma^*)=\omega(u_\gamma)\omega(\alpha(a))\omega(u_\gamma^*)=\chi(a)$. Hence, $\chi\in \chi(G)^\alpha$ and we have $\omega=\rho(\chi,\mu)$, where $\mu=(\gamma\mapsto\omega(u_\gamma))$. Moreover, as in the proof of Proposition \ref{PropInt}, it is easy to see that the map $\rho^{-1}$ is continuous, hence $\rho$ also, by compactness.

\end{proof}
%%%%%%%%%%%%%%%%%%%%%%%%%%%%%%
\subsection{Weak amenability}
%%%%%%%%%%%%%%%%%%%%%%%%%%%%%%

This subsection deals with weak amenability of $\widehat\Gq$ constructed in Section \ref{section-crossed}. We first prove an intermediate technical result to construct finite rank u.c.p. maps from $C(G)$ to itself using compactness of the action and elements of $\ell^{\infty}(\widehat{G})$ of finite support. Using this construction, we estimate the Cowling-Haagerup constant of $C(\Gq)$ and show that 
$C(\Gq)$ is weakly amenable when both $\Gamma$ and $\widehat{G}$ are weakly amenable and when the action is compact. This enables us to compute Cowling-Haagerup constants in some explicit examples given in Section \ref{section-examples}. We freely use the notations and definitions of Section \ref{section-AP}.

\begin{lemma}\label{LemAverage}
Suppose that the action $\alpha\,:\,\Gamma\curvearrowright G$ is compact. Denote by $H<{\rm Aut}(G)$ the compact group obtained by taking the closure of the image of $\Gamma$ in {\rm Aut}$(G)$. If $a\in \ell^\infty(\widehat{G})$ has finite support, then the linear map $\Psi\,:\,C(G)\rightarrow C(G)$, defined by $\Psi(z)=\int_H(h^{-1}\circ m_a\circ h)(z)dh$ has finite dimensional rank and $\Vert\Psi\Vert_{cb}\leq\Vert m_a\Vert_{cb}$, where $dh$ denotes integration with respect to the normalized Haar measure on $H$.
\end{lemma}

\begin{proof}
First observe that $\Psi$ is well defined since, for all $z\in C(G)$, the map $H\ni h\mapsto (h^{-1}\circ m_a\circ h)(z)\in C(G)$ is continuous. Moreover, the linearity of $\Psi$ is obvious. Since $a$ has finite support, the map $m_a$ is of the form $m_a(\cdot)=\omega_1(\cdot)y_1+\cdots+\omega_n(\cdot)y_n$, where $\omega_i\in C(G)^*$ and $y_i\in{\rm Pol}(G)$. Hence, to show that $\Psi$ has finite rank, it suffices to show that the map $\Psi_1(z)=\int_H(h^{-1}\circ\varphi\circ h)(z)dh$, $z\in C(G)$, has finite dimensional rank when $\varphi(\cdot)=\omega(\cdot)y$, with $\omega\in C(G)^*$ and $y\in{\rm Pol}(G)$. 

In this case, we have $\Psi_1(z)=\int_H\omega(h(z))h^{-1}(y)dh$, $z\in C(G)$. Write $y$ as a finite sum $y=\sum_{i=1}^N\sum_{k,l}\lambda_{i,k,l}u^{x_i}_{kl}$, where $F=\{x_1,\cdots,x_N\}\subset{\rm Irr}(G)$. Since $H$ is compact, the action of $H$ on ${\rm Irr}(G)$ has finite orbits. Writing $h\cdot x$ for the action of $h\in H$ on $x\in{\rm Irr}(G)$, the set $H\cdot F=\{h\cdot x\,:\,h\in H,x\in F\}\subset{\rm Irr}(G)$ is finite and, for all $h\in H$, $h^{-1}(y)\in\mathcal{F}$, where $\mathcal{F}$ is the finite dimensional subspace of $C(G)$ generated by the coefficients of the irreducible representations $x\in H\cdot F$. Hence, the map $h\mapsto \omega(h(z))h^{-1}(y)$ takes values in $\mathcal{F}$, for all $z\in C(G)$. It follows that $\Psi_1(z)=\int_H\omega(h(z))h^{-1}(y)dh\in\mathcal{F}$ for all $z\in C(G)$. Hence, $\Psi$ has finite dimensional rank.

Now we proceed to show that $\Vert\Psi\Vert_{cb}\leq\Vert m_a\Vert_{cb}$. For $n\in\N$, denote by $\Psi_n$ the map
\begin{align*}
\Psi_n=\id\ot\Psi\,:\,M_n(\C)\ot C(G)\rightarrow M_n(\C)\ot C(G).
\end{align*}
Observe that $\Psi_n(X)=\int_H(\id\ot (h^{-1}\circ m_a\circ h))(X)dh$ for all $X\in M_n(\C)\ot C(G)$. Hence, for $n\in \mathbb{N}$, one has 
\begin{align*}
\Vert\Psi_n(X)\Vert\leq \int_H\Vert(\id\ot (h^{-1}\circ m_a\circ h))(X)\Vert dh\leq\Vert X\Vert\int_H\Vert (h^{-1}\circ m_a\circ h)\Vert_{cb}dh\leq\Vert X\Vert\,\Vert m_a\Vert_{cb}.
\end{align*}
It follows that $\Vert\Psi\Vert_{cb}\leq\Vert m_a\Vert_{cb}$.
\end{proof}

\begin{theorem}\label{ThmWA}
We have ${\rm max}(\Lambda_{cb}(\Gamma),\Lambda_{cb}(C(G)))\leq \Lambda_{cb}(C(\Gq))$. Moreover, if the action $\Gamma\curvearrowright G$ is compact, then $\Lambda_{cb}(C(\Gq))\leq\Lambda_{cb}(\Gamma)\Lambda_{cb}(\widehat{G})$.
\end{theorem}

\begin{proof}
The first inequality is obvious by the existence of conditional expectations from $C(\Gq)$ to $C^*_r(\Gamma)$ and from $C(\Gq)$ to $C(G)$. Let us prove the second inequality. We may and will assume that $\Gamma$ and $\widehat{G}$ are weakly amenable. Fix $\epsilon>0$.

Let $a_i\in\ell^{\infty}(\widehat{G})$ be a sequence of finitely supported elements such that $\underset{i}\sup\Vert m_{a_i}\Vert_{cb}\leq\Lambda_{cb}(\widehat{G})+\epsilon$ and $m_{a_i}$ converges pointwise in norm to identity. Consider the maps $\Psi_i$ associated to $a_i$ as in Lemma \ref{LemAverage}. Observe that the sequence $\Psi_i$ converges pointwise in norm to identity. Indeed, for $x\in C(G)$,
\begin{eqnarray*}
\Vert\Psi_i(x)-x\Vert&=&\Vert\int_H((h^{-1}\circ m_{a_i}\circ h)(x)-x)dh\Vert
=\Vert\int_H(h^{-1}(m_{a_i}(h(x))-h(x))dh\Vert\\
&\leq&\int_H\Vert m_{a_i}(h(x))-h(x)\Vert dh.
\end{eqnarray*}
Now the right hand side of the above expression is converging to $0$ for all $x\in C(G)$ by the dominated convergence theorem, since $\Vert m_{a_i}(h(x))-h(x)\Vert\rightarrow_i 0$ for all $x\in C(G)$ and all $h\in H$, and
\begin{align*}
\Vert m_{a_i}(h(x))-h(x)\Vert\leq(\Vert m_{a_i}\Vert_{cb}+1)\Vert x\Vert\leq(\Lambda_{cb}(\widehat{G})+\epsilon+1)\Vert x\Vert\quad\text{for all }i\text{ and all }x\in C(G).
\end{align*}

By definition, the maps $\Psi_i$ are $\Gamma$-equivariant i.e., $\Psi_i\circ\alpha_\gamma=\alpha_\gamma\circ\Psi_i$. Hence, for all $i$, there is a unique linear extension $\widetilde{\Psi}_i\,:\, C(\Gq)\rightarrow C(\Gq)$ such that $\widetilde{\Psi}_i(u_\gamma\alpha(x))=u_\gamma\alpha(\Psi_i(x))$ for all $x\in C(G)$ and all $\gamma\in\Gamma$. Moreover, $\Vert \widetilde{\Psi}_i\Vert_{cb}\leq\Vert\Psi_i \Vert_{cb}\leq\Vert m_{a_i}\Vert_{cb}\leq\Lambda_{cb}(\widehat{G})+\epsilon$.

Consider a sequence of finitely supported maps $\psi_j\,:\,\Gamma\rightarrow\C$ going pointwise to $1$ and such that $\sup\Vert m_{\psi_j}\Vert_{cb}\leq(\Lambda_{cb}(\Gamma)+\epsilon)$, and denote by $\widetilde{\psi_j}\,:\,C(\Gq)\rightarrow C(\Gq)$ the unique linear extension such that $\widetilde{\psi_j}(u_\gamma\alpha(x))=\psi_j(\gamma)u_\gamma\alpha(x)$. Then, we have $\Vert \widetilde{\psi}_j\Vert_{cb}\leq\Vert m_{\psi_j} \Vert_{cb}\leq\Lambda_{cb}(\Gamma)+\epsilon$.

Define the maps $\varphi_{i,j}=\widetilde{\psi}_j\circ\widetilde{\Psi}_i\,:\, C(\Gq)\rightarrow C(\Gq)$. Then for all $i,j$ we have $\Vert\varphi_{i,j}\Vert_{cb}\leq(\Lambda_{cb}(\Gamma)+\epsilon)(\Lambda_{cb}(\widehat{G})+\epsilon)$.  Since $\varphi_{i,j}(u_\gamma\alpha(x))=\psi_j(\gamma)u_\gamma\alpha(\Psi_i(x))$, it is clear that $\varphi_{i,j}$ has finite dimensional rank, and $(\varphi_{i,j})_{i,j}$ is going pointwise in norm to identity. Since $\epsilon$ was arbitrary, the proof is complete.
\end{proof}

%%%%%%%%%%%%%%%%%%%%%%%%%%%%%%%%%%%%
\subsection{Rapid Decay}
%%%%%%%%%%%%%%%%%%%%%%%%%%%%%%%%%%%

In this subsection we study property $(RD)$ for crossed products. We use the notion of property $(RD)$ developed in \cite{BVZ14} and recall the definition below. Since for a discrete quantum subgroup $\widehat{G}<\widehat{\Gq}$, i.e. such that there exists a faithful unital $*$-homomorphism $C_m(G)\rightarrow C_m(\Gq)$ which intertwines the comultiplications, property $(RD)$ for $\widehat{\Gq}$ implies property $(RD)$ for $\widehat{G}$ and, since for a crossed product $\widehat{\Gq}$ coming from an action $\Gamma\curvearrowright G$ of a discrete group $\Gamma$ on a compact quantum group $G$, both $\Gamma$ and $\widehat{G}$ are discrete quantum subgroups of $\widehat{\Gq}$, it follows that property $(RD)$ for $\widehat{\Gq}$ implies property $(RD)$ for $\Gamma$ and $\widehat{G}$. Hence, we will only concentrate on proving the converse.

For a compact quantum group $G$ and $a\in C_c(\widehat{G})$ we define its Fourier transform as:
\begin{align*}
\mathcal{F}_G(a)=(h_{\widehat{G}}\ot 1)(V(a\ot 1))=\sum_{x\in {\rm Irr}(G)} {\rm dim}_q(x)({\rm Tr}_x\ot\id)((Q_x\otimes 1)u^x(ap_x\ot 1))\in{\rm Pol}(G),
\end{align*}
and its ``Sobolev 0-norm'' by $\|a\|^2_{G,0}=\sum_{x\in {\rm Irr}(G)}\frac{{\rm dim}_q(x)^2}{{\rm dim}(x)}{\rm Tr}_x(Q_x^\ast(a^\ast a)p_x Q_x)$.

Let $\alpha\,:\,\Gamma\curvearrowright G$ be an action by quantum automorphisms and denote by $\Gq$ the crossed product. Recall that ${\rm Irr}(\Gq)=\{\gamma\cdot x\,:\,\gamma\in\Gamma\text{ and }x\in{\rm Irr}(G)\}$, where $\gamma\cdot x$ is the equivalence class of
\begin{align*}
u_\gamma^x=(1\ot u_\gamma)(\id\ot\alpha)(u^x)\in\mathcal{B}(H_x)\ot C(\Gq).
\end{align*}
Let $V_{\gamma\cdot x}\,:\,H_{\gamma\cdot x}\rightarrow H_x$ be the unique unitary such that $u^{\gamma\cdot x}=(V_{\gamma\cdot x}^*\ot 1) u_\gamma^x (V_{\gamma\cdot x}\ot 1)$.

\begin{lemma}\label{LemQ_x}
 For any $\gamma\in \Gamma$ and $x\in {\rm Irr}(G)$, one has $Q_{\gamma\cdot x}=V_{\gamma\cdot x}^\ast Q_x V_{\gamma\cdot x}$ and ${\rm dim}_q(\gamma\cdot x)={\rm dim}_q(x)$.
\end{lemma}

\begin{proof}
Since $V_{\gamma\cdot x}$ is unitary, it suffices to show the first assertion. Recall that $Q_{\gamma\cdot x}$ is uniquely determined by the properties that it is invertible, ${\rm Tr}_{\gamma\cdot x}(Q_{\gamma\cdot x})={\rm Tr}_{\gamma\cdot x}(Q_{\gamma\cdot x}^{-1})>0$ and that $Q_{\gamma\cdot x}\in {\rm Mor}(u^{\gamma\cdot x}, u^{\gamma\cdot x}_{cc})$, where $u^{\gamma\cdot x}_{cc}=(\id\otimes S_{\Gq}^2)(u^{\gamma\cdot x})$. It is obvious that $Q:=V_{\gamma\cdot x}^\ast Q_x V_{\gamma\cdot x}$ is invertible and that ${\rm Tr}_{\gamma\cdot x}(Q)={\rm Tr}_{\gamma\cdot x}(Q^{-1})>0$. Hence, we will be done once we show that $Q\in {\rm Mor}(u^{\gamma\cdot x}, u^{\gamma\cdot x}_{cc})$. To this end, we first note that we have, by Remark \ref{RmkAnti}, for any $\gamma\in \Gamma$ and $a\in {\rm Pol}(G)$, $S_{\Gq}^2(u_{\gamma}\alpha(a))=u_\gamma\alpha(S_G^2(a))$. Thus, $(\id\otimes S_{\Gq}^2)(u^x_\gamma)=(1\otimes u_{\gamma})(\id\otimes \alpha)((\id\otimes S_G^2)(u^x))$. It follows that $Q_x\in {\rm Mor}(u^x_\gamma, (u^x_\gamma)_{cc})$ hence $Q\in {\rm Mor}(u^{\gamma\cdot x}, u^{\gamma\cdot x}_{cc})$.  
\end{proof}

\begin{lemma}\label{LemFourier}
Let $a\in C_c(\widehat{\Gq})$ and write $a=\sum_{\gamma\in S,x\in T}ap_{\gamma\cdot x}$, where $S\subset\Gamma$ and $T\subset{\rm Irr}(G)$ are finite subsets. For $\gamma\in S$, define $a_\gamma\in C_c(\widehat{G})$ by $a_\gamma=\sum_{x\in T}V_{\gamma\cdot x}ap_{\gamma\cdot x}V_{\gamma\cdot x}^*p_x$.
The following holds.
\begin{enumerate}
\item $\mathcal{F}_\Gq(a)=\sum_{\gamma\in S}u_\gamma\alpha(\mathcal{F}_G(a_\gamma))$.
\item $\Vert a\Vert_{\Gq,0}^2=\sum_{\gamma\in S}\Vert a_\gamma\Vert_{G,0}^2$.
\end{enumerate}
\end{lemma}

\begin{proof}
Observe that, since $V_{\gamma\cdot x}$ is unitary, ${\rm Tr}_{\gamma\cdot x}(V_{\gamma\cdot x}^*AV_{\gamma\cdot x}B)={\rm Tr}_{x}(AV_{\gamma\cdot x}BV_{\gamma\cdot x}^*)$ for all $\gamma\in\Gamma$, all $x\in{\rm Irr}(G)$ and all $A\in\mathcal{B}(H_{ x})$, $B\in\mathcal{B}(H_{\gamma\cdot x})$. Hence,
\begin{eqnarray*}
\mathcal{F}_\Gq(a)&=&\sum_{\gamma\in S,x\in T}{\rm dim}_q(\gamma\cdot x)({\rm Tr}_{\gamma\cdot x}\ot\id)((Q_{\gamma\cdot x}\otimes 1)u^{\gamma\cdot x}(ap_{\gamma\cdot x}\ot 1))\\
&=&\sum_{\gamma\in S,x\in T}{\rm dim}_q(x)({\rm Tr}_{\gamma\cdot x}\ot\id)((V_{\gamma\cdot x}^*\ot 1)(Q_x\otimes 1)(V_{\gamma\cdot x}\ot 1)(V_{\gamma\cdot x}^*\ot 1) u_\gamma^x (V_{\gamma\cdot  x}\ot 1)(ap_{\gamma\cdot x}\ot 1))\\
&=&\sum_{\gamma\in S,x\in T}{\rm dim}_q(x)({\rm Tr}_{x}\ot\id)((Q_x\ot 1)u_\gamma^x (V_{\gamma\cdot  x}ap_{\gamma\cdot x}V_{\gamma\cdot x}^*\ot 1))\\
&=&\sum_{\gamma\in S}u_\gamma\alpha\left(\sum_{x\in T}{\rm dim}_q(x)({\rm Tr}_{x}\ot\id)((Q_x\ot 1)u^x (V_{\gamma\cdot  x}ap_{\gamma\cdot x}V_{\gamma\cdot x}^*\ot 1))\right)=\sum_{\gamma\in S}u_\gamma\alpha(\mathcal{F}_G(a_\gamma)).
\end{eqnarray*}
This shows assertion $1$. Assertion $2$ follows from a similar computation using again Lemma \ref{LemQ_x}.
\end{proof}

A function $l:\text{Irr}(G)\rightarrow [0,\infty)$ is called a \textit{length function on} ${\rm Irr}(G)$ if $l(1)=0$, $l(\overline{x})=l(x)$ and that $l(x)\leq l(y)+l(z)$ whenever $x\subset y\otimes z$.

\begin{lemma}\label{Length_L_0}
Let $\alpha\,:\,\Gamma\curvearrowright G$ be an action of $\Gamma$ on $G$ by quantum automorphisms and let $l$ be a length function on ${\rm Irr}(G)$ which is $\alpha$-invariant, i.e., $l(x)=l(\alpha_\gamma(x))$ for all $\gamma\in \Gamma$ and $x\in$ \emph{Irr}$(G)$. Let   $l_\Gamma$ be a length function on $\Gamma$. Let $\Gq$ be the crossed product. The function $l_0\,:\,{\rm Irr}(\Gq)\rightarrow[0,\infty)$, defined by $l_0(\gamma\cdot x)=l_\Gamma(\gamma)+l(x)$ is a length function on ${\rm Irr}(\Gq)$.
\end{lemma}

\begin{proof}
We have $l_0(1)=l_\Gamma(e)+l(1)=0$ and, by Remark \ref{RmkFusion},
\begin{align*}
l_0(\overline{\gamma\cdot x})=l_0(\gamma^{-1}\cdot  \alpha_\gamma(\overline{x}))=l_\Gamma(\gamma^{-1})+l(\alpha_\gamma(\overline{x}))=l_\Gamma(\gamma)+l(\overline{x})=l_0(\gamma\cdot x).
\end{align*}
Again, from Remark \ref{RmkFusion},  $\gamma\cdot x\subset r\cdot y\ot s\cdot z$ if and only if $\gamma=rs$ and $x\subset\alpha_{\gamma^{-1}}(y)\ot z$. Hence, 
\begin{eqnarray*}
l_0(\gamma\cdot x)&=&l_\Gamma(\gamma)+l(x)\leq l_\Gamma(r)+ l_\Gamma(s)+l(\alpha_{\gamma^{-1}}(y))+l(z)\\
&=&l_\Gamma(r)+l(y)+ l_\Gamma(s)+l(z)=l_0(r\cdot y)+l_0(s\cdot z).
\end{eqnarray*}
\end{proof}

Given a length function $l:\text{Irr}(G)\rightarrow [0,\infty)$, consider the element $L=\sum_{x\in \text{Irr}(G)}l(x)p_x$ which is affilated to $c_0(\widehat{G})$. Let $q_n$ denote the spectral projections of $L$ associated to the interval $[n,n+1)$. We say that \textit{$(\widehat{G},l)$ has property $($RD$)$}, if there exists a polynomial $P\in \mathbb{R}[X]$ such that for every $k\in \mathbb{N}$ and $a\in q_kc_c(\widehat{G})$, we have $\|\mathcal{F}(a)\|_{C(G)}\leq P(k)\|a\|_{G,0}$. Finally, \textit{$\widehat{G}$ is said to have  Property $($RD$)$} if there exists a length function $l$ on ${\rm Irr}(G)$ such that $(\widehat{G},l)$ has property $(RD)$.

We prove property $(RD)$ for the dual of a crossed product in the following Theorem. In case the action of the group is trivial, i.e., when the crossed product reduces to a tensor product, this result is proved in \cite[Lemma 4.5]{CF14}. For semi-direct products of classical groups, this result is due to Jolissaint \cite{Jo90}.

\begin{theorem}\label{ThmRD}
Let  $\alpha\,:\,\Gamma\curvearrowright G$  be an action by quantum automorphisms. Let $l$ be a $\alpha$-invariant length function on ${\rm Irr}(G)$. If $(\widehat{G},l)$ has property $(RD)$ and $\Gamma$ has property $(RD)$, then $(\widehat{\Gq},l_0)$ has property $(RD)$, where $\Gq$ is the crossed product and $l_{0}$ is as in Lemma \ref{Length_L_0}.
\end{theorem} 

\begin{proof}
Let $l_\Gamma$ be any length function on $\Gamma$ for which $(\Gamma,l_\Gamma)$ has property (RD) and let $l_0$ be the length function on ${\rm Irr}(\Gq)$ defined by $l_0(\gamma\cdot x)=l_\Gamma(\gamma)+ l(x)$, for $\gamma\in \Gamma$ and $x\in \text{Irr}(G)$. Let $L_0=\sum_{\gamma\in\Gamma,x\in{\rm Irr}(G)}l_0(\gamma\cdot x)p_{\gamma\cdot x}=\sum_{\gamma\in\Gamma,x\in{\rm Irr}(G)}(l_\Gamma(\gamma)+l(x))p_{\gamma\cdot x}$ and $L=\sum_{x\in{\rm Irr}(G)}l(x)p_x$. Finally, let $p_n$ and $q_n$ be the spectral projections of respectively $L_0$ and $L$ associated to the interval $[n,n+1)$. Let $a\in c_c(\widehat{\Gq})$ and write $a=\sum_{\gamma\in S,x\in T}ap_{\gamma\cdot x}$, where $S\subset\Gamma$ and $T\subset{\rm Irr}(G)$ are finite subsets. Now suppose that $a\in p_kc_c(\widehat{\Gq})$. Since $p_k=\sum_{\gamma\in\Gamma,x\in{\rm Irr}(G),k\leq l_\Gamma(\gamma)+l(x)<k+1}p_{\gamma\cdot x}$, we must have
$$S\subset\{\gamma\in\Gamma\,:\, l_\Gamma(\gamma)<k+1\}\quad\text{and}\quad T\subset\{x\in{\rm Irr}(G)\,:\,l(x)< k+1\}.$$
It follows that, for all $\gamma\in S$, the element $a_\gamma$ defined in Lemma \ref{LemFourier} is in $q_Kc_c(\widehat{G})$, where  $q_K=\sum_{j=0}^k q_j$.

Let $P_1$ and $P_2$ be polynomials witnessing $(RD)$ respectively for $(\widehat{G},l)$ and $(\Gamma,l_\Gamma)$. Let, for $i=1,2$, $C_i\in\R_+$ and $N_i\in\N$ be such that $P_i(k)\leq C_i(k+1)^{N_i}$ for all $k\in\N$. Then, for all $b\in q_Kc_c(\widehat{G})$,
\begin{eqnarray*}
\Vert\mathcal{F}_G(b)\Vert&\leq& \sum_{j\leq k}\Vert\mathcal{F}_G(bq_j)\Vert\leq\sum_{j\leq k}P_1(j)\Vert bq_j\Vert_{G,0}\leq\sum_{j\leq k}C_1(j+1)^{N_1}\Vert bq_j\Vert_{G,0}\\
&\leq& C_1(k+1)^{N_1}\sum_{j\leq k}\Vert bq_j\Vert_{G,0}= C_1(k+1)^{N_1+1}\Vert b\Vert_{G,0}.
\end{eqnarray*}
Similarly, $\Vert\psi*\phi\Vert_{\ell^2(\Gamma)}\leq C_2(k+1)^{N_2+1}\|\psi\|_{\ell^2(\Gamma)}\|\phi\|_{\ell^2(\Gamma)}$ for all $\phi$ in $\ell^2(\Gamma)$ and all functions $\psi$ on $\Gamma$ (finitely) supported on words of $l_\Gamma$-length less than equal to $k$.

Let $y$ be a finite sum $y=\sum_s u_s\alpha(b_s)\in{\rm Pol}(\Gq)$. We have $\Vert y\Vert_{2,h_\Gq}^2=\sum_s\Vert b_s\Vert_{2,h_G}^2$ and, by Lemma \ref{LemFourier} and the preceding discussion,

\begin{eqnarray*}
\|\mathcal{F}_\Gq(a)y\|^2_{2,h_\Gq}&=&\|\sum_{\gamma\in S,s}u_{\gamma s}\alpha(\alpha_{s^{-1}}(\mathcal{F}_G(a_\gamma))b_s)\|^2_{2,h_\Gq}=\|\sum_{\gamma\in S,t}u_{t}\alpha(\alpha_{t^{-1}\gamma}(\mathcal{F}_G(a_\gamma))b_{\gamma^{-1}t})\|^2_{2,h_\Gq}\\
&=&\sum_t\Vert\sum_{\gamma\in S}\alpha_{t^{-1}\gamma}(\mathcal{F}_G(a_\gamma))b_{\gamma^{-1}t}\Vert_{2,h_G}^2\leq\sum_t\left(\sum_{\gamma\in S}\Vert\alpha_{t^{-1}\gamma}(\mathcal{F}_G(a_\gamma))b_{\gamma^{-1}t}\Vert_{2,h_G}\right)^2\\
&\leq & C_1^2(k+1)^{2(N_1+1)}\sum_t\left(\sum_{\gamma\in S}\Vert a_\gamma\Vert_{G,0}\Vert b_{\gamma^{-1}t}\Vert_{2,h_G}\right)^2=C_1^2(k+1)^{2(N_1+1)}\Vert\psi*\phi\Vert^2_{l^2(\Gamma)},
\end{eqnarray*}

where $\psi,\phi\in \ell^2(\Gamma)$ are defined by $\psi(\gamma)=\|a_\gamma\|_{G,0}$ and $\phi(s)=\|b_s\|_{2,h_G}$ where $\gamma,s\in \Gamma$. We note that $\|\psi\|^2_{\ell^2(\Gamma)}=\sum_{\gamma\in S} \|a_\gamma\|^2_{G,0}=\|a\|^2_{\Gq,0}$ and $\|\phi\|^2_{\ell^2(\Gamma)}=\sum_s\|b_s\|_{2,h_G}^2=\|y\|^2_{2,h_\Gq}$. But since $\psi$ is supported on $S$ i.e., on elements of $\Gamma$ of length less than equal to $k$, we have 
\begin{align*}
\|\mathcal{F}_\Gq(a)y\|^2_{2,h_\Gq}\leq (C_1C_2)^2(k+1)^{2(N_1+N_2+2)}\|\psi\|^2_{l^2(\Gamma)}\|\phi\|^2_{\ell^2(\Gamma)}= P(k)^2\|a\|^2_{\Gq,0}\|y\|^2_{2,h_\Gq},
\end{align*}
where $P(x)=C_1C_2(x+1)^{N_1+N_2+2}$. As $y$ is arbitrary, the proof is complete.
\end{proof}

\begin{remark}\label{RmkRD}
There may not exist an $\alpha$-invariant length function on ${\rm Irr}(G)$. However, if $\Gamma\curvearrowright G$ is compact, then the action $\alpha\,:\,\Gamma\curvearrowright{\rm Irr}(G)$ has finite orbits. Hence, for any length function $l$ on ${\rm Irr}(G)$, the length function $l_\alpha$ defined by $l_\alpha(x)= \sup_{\gamma\in \Gamma}\, l(\alpha_\gamma(x))$, for $x\in{\rm Irr}(G)$, is $\alpha$-invariant. Hence, $\widehat{\Gq}$ has $(RD)$ whenever $\Gamma$ and $\widehat{G}$ have $(RD)$.
\end{remark}

%%%%%%%%%%%%%%%%%%%%%%%%%%%%%%
\subsection{Property (T)}
%%%%%%%%%%%%%%%%%%%%%%%%%%%%%%

We characterize relative co-property $(T)$ of the pair $(G,\Gq)$ in a similar way we did characterize relative co-property $(T)$ for bicrossed product. We study the the property $(T)$ for $\widehat{\Gq}$.

When $\pi\,:\, A\rightarrow\mathcal{B}(H)$ is a unital $*$-homomorphism from a unital C*-algebra $A$, we denote by $\widetilde{\pi}\,:\,A^{**}\rightarrow\mathcal{B}(H)$ its unique normal extension. Also, we view any state $\omega\in A^*$ as a normal state on $A^{**}$. Observe that if $(H,\pi,\xi)$ is the GNS construction for the state $\omega$ on $A$, then $(H,\widetilde{\pi},\xi)$ is the GNS construction for the normal state $\omega$ on $A^{**}$.

Let $M=C_m(G)^{**}$ and $p_0\in M$ be the unique central projection such that $p_0xp_0=\widetilde{\varepsilon}_G(x)p_0$ for all $x\in M$.

In the following theorem, we characterize the relative co-property $(T)$ of the pair $(G,\Gq)$ in terms of the action $\alpha$ of $\Gamma$ on $G$. The proof is similar to the proof of Theorem \ref{ThmRelT} but technically more involved.

\begin{theorem}\label{ThmRelT2}
The following are equivalent:
\begin{enumerate}
\item The pair $(G,\Gq)$ does not have the relative co-property $(T)$.
\item There exists a sequence $(\omega_n)_{n\in\N}$ of states on $C_m(G)$ such that
\begin{enumerate}
\item $\omega_n(p_0)=0$  for all $n\in\N$;
\item $\omega_n\rightarrow\varepsilon_G$ { weak*};
\item $\Vert\alpha_\gamma(\omega_n)-\omega_n\Vert\rightarrow 0$ for all $\gamma\in\Gamma$.
\end{enumerate}
\end{enumerate}
\end{theorem}

\begin{proof}
For a representation $\pi\,:\,C_m(\Gq)\rightarrow\mathcal{B}(H)$, we have $\varepsilon_G\subset\pi\circ\alpha$ if and only if $K_\pi\neq\{0\}$, where
\begin{align*}
K_\pi=\{\xi\in H\,:\,\pi\circ\alpha(a)\xi=\varepsilon_G(a)\xi\text{ for all }a\in C_m(G)\}.
\end{align*}
Let $\rho=\pi\circ\alpha\,:\,C_m(G)\rightarrow\mathcal{B}(H)$ and observe that the orthogonal projection onto $K_\pi$ is the projection $\widetilde{\rho}(p_0)$. Indeed, for all $\xi\in H$, $a\in C_m(G)$, we have $\pi\circ\alpha(a)\widetilde{\rho}(p_0)\xi=\widetilde{\rho}(ap_0)\xi=\varepsilon_G(a)\widetilde{\rho}(p_0)\xi$, which implies that ${\rm Im}(\widetilde{\rho}(p_0))\subset K_\pi$. Moreover, if $\xi\in K_\pi$, we have
$\widetilde{\rho}(a)\xi=\widetilde{\varepsilon}_G(a)\xi$ for all $a\in C_m(G)$. Since $C_m(G)$ is $\sigma$-weakly dense in $M$ and the representations $\widetilde{\rho}$ and $\widetilde{\varepsilon}_G$ are normal, it follows that the equation $\widetilde{\rho}(a)\xi=\widetilde{\varepsilon}_G(a)\xi$ is valid for all $a\in M$. Hence, for $a=p_0$ we get $\widetilde{\rho}(p_0)\xi=\widetilde{\varepsilon}_G(p_0)\xi=\xi$, which in turn implies that $K_\pi\subset {\rm Im}(\widetilde{\rho}(p_0))$.

$(1)\implies (2)$. Suppose that the pair $(G,\Gq)$ does not have the relative co-property $(T)$. Let $\pi\,:\,C_m(\Gq)\rightarrow\mathcal{B}(H)$ be a representation such that $\varepsilon_\Gq\prec\pi$ and $K_\pi=\{0\}$. Denote by $\omega_{\xi,\eta}\in C_m(G)^*$ the functional given by $\omega_{\xi,\eta}(a)=\langle\pi\circ\alpha(a)\xi,\eta\rangle$. Hence, $\omega_{\xi,\eta}(p_0)=\langle \widetilde{\rho}(p_0)\xi,\eta\rangle=0$ for all $\xi,\eta\in H$.

Since $\varepsilon_\Gq\prec\pi$, let $(\xi_n)_{n\in\N}$ be a sequence of unit vectors in $H$ such that $\Vert\pi(x)\xi_n-\varepsilon_\Gq(x)\xi_n\Vert\rightarrow 0$ for all $x\in C_m(\Gq)$. Define $\omega_n=\omega_{\xi_n,\xi_n}$. Then, we have $\omega_n(p_0)=0$ for all $n\in\N$. For all $a\in C_m(G)$ we have, $
\vert\omega_n(a)-\varepsilon_G(a)\vert=\vert\langle\pi(\alpha(a))\xi_n-\varepsilon_G(a)\xi_n,\xi_n\rangle\vert\leq\Vert\pi(\alpha(a))\xi_n-\varepsilon_\Gq(\alpha(a))\xi_n\Vert\rightarrow 0$. Moreover, exactly as in the proof of Theorem \ref{ThmRelT}, we find $\Vert\alpha_\gamma(\omega_n)-\omega_n\Vert\leq 2\Vert\pi(u_\gamma)\xi_n-\xi_n\Vert=\Vert\pi(u_\gamma)\xi_n-\varepsilon_\Gq(u_\gamma)\xi_n\Vert\rightarrow 0$.

$(2)\implies(1)$. For a state $\omega\in C_m(G)^*=M_*$ we denote by $s(\omega)\in M$ its support. Recall that $s(\omega)\in M$ is the unique projection in $M$ such that $N_\omega=M(1-s(\omega))$, where $N_\omega$ is the $\sigma$-weakly closed left ideal defined by $N_\omega=\{x\in M\,:\,\omega(x^*x)=0\}$ and note that $\omega$ is faithful on $s(\omega)Ms(\omega)$. In the sequel, we still denote by $\alpha_\gamma$ the unique $*$-isomorphism of $M$ which extends $\alpha_\gamma\in\text{Aut}(C_m(G))$. We first prove the following claim. 

\textbf{Claim.} \textit{If $(2)$ holds, then there exists a sequence $(\omega_n)_{n\in\N}$ of states on $C_m(G)$ satifying $(a)$, $(b)$ and $(c)$ and such that
$\alpha_\gamma(s(\omega_n))=s(\omega_n)$ for all $\gamma\in\Gamma$, $n\in\N$.}

\textit{Proof of the claim.} Denote by $\ell^1(\Gamma)_{1,+}$ the set of positive $\ell^1$ functions $f$ on $\Gamma$ with $\Vert f\Vert_1=1$. For a state $\omega\in C_m(G)^*=M_*$ and $f\in \ell^1(\Gamma)_{1,+}$, define the state $f*\omega\in C_m(G)^*$ by the convex combination
\begin{align*}
f*\omega=\sum_{\gamma\in\Gamma}f(\gamma)\alpha_\gamma(\omega).
\end{align*}
Observe that, for all $\gamma\in\Gamma$ we have $\delta_\gamma*\omega=\alpha_\gamma(\omega)$ and $\alpha_\gamma(f*\omega)=f_\gamma*\omega$, where $f_\gamma\in \ell^1(\Gamma)_{1,+}$ is defined by $f_\gamma(r)=f(\gamma^{-1}r)$, $r\in\Gamma$. Moreover, if $f\in \ell^1(\Gamma)_{1,+}$ is such that $f(\gamma)>0$ for all $\gamma\in\Gamma$, then since $(f*\omega)(x^*x)=\sum_\gamma f(\gamma)\omega(\alpha_{\gamma^{-1}}(x^*x))$, we have that $(f*\omega)(x^*x)=0$ if and only if $\omega(\alpha_{\gamma^{-1}}(x^*x))=0$ for all $\gamma\in\Gamma$. It follows that
\begin{align*}
N_{f*\omega}=\cap_{\gamma\in\Gamma}\alpha_\gamma(N_\omega)=M\left(\wedge_{\gamma\in\Gamma}(1-\alpha_\gamma(s(\omega)))\right).
\end{align*}
Hence, $s(f*\omega)=1-\wedge_{\gamma\in\Gamma}(1-\alpha_\gamma(s(\omega)))=\vee_{\gamma\in\Gamma}\alpha_\gamma(s(\omega))$. Hence, we have $\alpha_\gamma(s(f*\omega))=s(f*\omega)$ for all $\gamma\in\Gamma$. Finally, since $\varepsilon_G\circ\alpha_\gamma=\varepsilon_G$, we deduce that, for all $\gamma\in\Gamma$, $\alpha_\gamma(p_0)$ is a central projection of $M$ satisfying $a\alpha_\gamma(p_0)=\alpha_\gamma(\alpha_{\gamma^{-1}}(a)p_0)=\varepsilon_G(\alpha_{\gamma^{-1}}(a))\alpha_\gamma(p_0)=\varepsilon_G(a)\alpha_\gamma(p_0)$, $\gamma\in \Gamma$. 
By uniqueness of such a projection, we find $\alpha_\gamma(p_0)=p_0$ for all $\gamma\in\Gamma$. Hence, for all $f\in \ell^1(\Gamma)_{1,+}$,
\begin{align*}
(f*\omega)(p_0)=\sum_{\gamma}f(\gamma)\omega(\alpha_{\gamma^{-1}}(p_0))=\sum_\gamma f(\gamma)\omega(p_0)=\omega(p_0).
\end{align*}
Let $(\omega_n)_{n\in\N}$ be a sequence of states on $C_m(G)$ satisfying $(a)$, $(b)$ and $(c)$. We have, for all $f\in \ell^1(\Gamma)_{1,+}$ with finite support
\begin{equation}\label{EqProba2}
\Vert f*\omega_n-\omega_n\Vert\leq\sum_\gamma f(\gamma)\Vert \delta_\gamma*\omega_n-\omega_n\Vert=\sum_\gamma f(\gamma)\Vert \alpha_\gamma(\omega_n)-\omega_n\Vert\rightarrow 0.
\end{equation}
Since such functions $f$ are dense in $\ell^1(\Gamma)_{1,+}$ (in the $\ell^1$-norm), it follows that $\Vert f*\omega_n-\omega_n\Vert\rightarrow 0$ for all $f\in \ell^1(\Gamma)_{1,+}$.

Let $\xi\in \ell^1(\Gamma)_{1,+}$ be any function such that $\xi>0$ and define $\nu_n=\xi*\omega_n$. By the preceding discussion, we know that $\alpha_\gamma(s(\nu_n))=s(\nu_n)$ for all $\gamma\in\Gamma$ and $\nu_n(p_0)=\omega_n(p_0)=0$ for all $n\in\N$. Moreover, by Equation $(\ref{EqProba2})$, we have $\Vert \alpha_\gamma(\nu_n)-\nu_n\Vert=\Vert \xi_\gamma*\omega_n-\xi*\omega_n\Vert\leq\Vert\xi_\gamma*\omega_n-\omega_n\Vert+\Vert\omega_n-\xi*\omega_n\Vert\rightarrow 0$
for all $\gamma\in\Gamma$. Since $\omega_n\rightarrow\varepsilon_G$ in the weak* topology and $\varepsilon_G\circ\alpha_\gamma=\varepsilon_G$, we have, $\vert\omega_n(\alpha_\gamma(a))-\varepsilon_G(a)\vert\rightarrow 0$ for all $a\in C_m(G)$ and all $\gamma\in\Gamma$. Hence, the Lebesgue dominated convergence theorem implies that, for all $a\in C_m(G)$,
$$\vert \nu_n(a)-\varepsilon_G(a)\vert=\left\vert\sum_\gamma f(\gamma)(\omega_n(\alpha_{\gamma^{-1}}(a))-\varepsilon_G(a))\right\vert\leq\sum_\gamma f(\gamma)\vert\omega_n(\alpha_{\gamma^{-1}}(a))-\varepsilon_G(a))\vert\rightarrow 0.$$
It follows that $\nu_n\rightarrow\varepsilon_G$ in the weak* topology and this completes the proof of the claim.\hfill{$\Box$}

\vspace{0.2cm}

We can now finish the proof of the Theorem. Let $(\omega_n)_{n\in\N}$ be a sequence of states on $C_m(G)$ as in the Claim. Let $M_n=s(\omega_n)Ms(\omega_n)$ and, since $\omega_n$ is faithful on $M_n$, view $M_n\subset\mathcal{B}(H_n)$ where $(H_n,\xi_n)$ is the GNS construction of the f.n.s. $\omega_n$ on $M_n$. Define $\rho_n\,:\, C_m(G)\subset M\rightarrow M_n\subset\mathcal{B}(H_n)$ by $a\mapsto s(\omega_n)as(\omega_n)$. By definition, the unique normal extension of $\rho_n$ is the map $\widetilde{\rho}_n\,:\, M\rightarrow M_n$, defined by $x\mapsto s(\omega_n)xs(\omega_n)$. Since $\alpha_\gamma(s(\omega_n))=s(\omega_n)$, the action $\alpha$ restricts to an action, still denoted by $\alpha$ of $\Gamma$ on $M_n$. Since $M_n\subset\mathcal{B}(H_n)$ is in standard form, we may consider the standard implementation (see [TakII Definition 1.6]) of the action of $\Gamma$ on $M_n$ to get a unitary representation $u_n\,:\,\Gamma\rightarrow\mathcal{U}(H_n)$ such that $\alpha_\gamma(x)=u_n(\gamma)xu_n(\gamma^{-1})$ for all $x\in M_n$ and $\gamma\in \Gamma$.

By the universal property of $A_m$, for $n\in\N$ there exists a unique unital $*$-homomorphism
\begin{align*}
\pi_n\,:\,A_m\rightarrow\mathcal{B}(H_n)\quad\text{such that}\quad\pi_n(u_\gamma)=u_n(\gamma)\quad\text{and}\quad\pi_n\circ\alpha=\rho_n.
\end{align*}

Since $\omega_n(p_0)=0$, we have $s(\omega_n)p_0s(\omega_n)=0$. Hence, $\widetilde{\rho}_n(p_0)=0$ and $K_{\pi_n}=\{0\}$  $\forall n\in\N$. It follows that, if we define $H=\oplus_{n} H_n$ and $\pi=\oplus_{n}\pi_n\,:\, C_m(\Gq)\rightarrow\mathcal{B}(H)$, then  $K_\pi=\{0\}$ as well. Hence, it suffices to show that $\varepsilon_\Gq\prec\pi$. Since $\xi_n$ is in the self-dual cone of $\omega_n$ and $u_n(\gamma)$ is the standard implementation of $\alpha_\gamma$, it follows from \cite[Theorem 1.14]{Ta00} that $u_n(\gamma)\xi_n$ is also in the self-dual cone of $\omega_n$ for all $n\in\N$. Hence, we may apply \cite[Theorem 1.2]{Ta00} to get $\Vert u_n(\gamma)\xi_n-\xi_n\Vert^2\leq\Vert\omega_{u_n(\gamma)\xi_n}-\omega_{\xi_n}\Vert$ for all $n\in\N$, $\gamma\in\Gamma$. Observe that $\omega_{u_n(\gamma)\xi_n}(x)=\alpha_\gamma(\omega_n)(x)$ and $\omega_{\xi_n}(x)=\omega_n(x)$ for all $x\in M$. Hence,
\begin{align*}
\Vert u_n(\gamma)\xi_n-\varepsilon_\Gq(u_\gamma)\xi_n\Vert=\Vert u_n(\gamma)\xi_n-\xi_n\Vert\leq\Vert\alpha_\gamma(\omega_n)-\omega_n\Vert^{\frac{1}{2}}\rightarrow 0.
\end{align*}

Since $\omega_n\rightarrow\varepsilon_G$ in the weak* topology, it follows that for all $x=u_\gamma\alpha(a)\in C_m(\Gq)$, we have
\begin{eqnarray*}
\Vert\pi(x)\xi_n-\varepsilon_\Gq(x)\xi_n\Vert&= &\Vert\pi(u_\gamma)\pi(\alpha(a))\xi_n-\varepsilon_G(a)\xi_n\Vert\\
&\leq&\Vert\pi(u_\gamma)(\pi(\alpha(a))\xi_n-\varepsilon_G(a)\xi_n)\Vert+\vert \varepsilon_G(a)\vert\,\Vert\pi(u_\gamma)\xi_n-\xi_n\Vert\\
&\leq&\Vert\pi(\alpha(a))\xi_n-\varepsilon_G(a)\xi_n\Vert+\vert \varepsilon_G(a)\vert\,\Vert u_n(\gamma)\xi_n-\xi_n\Vert\rightarrow 0.\\
\end{eqnarray*}
By linearity and the triangle inequality, we have $\Vert\pi(x)\xi_n-\varepsilon_\Gq(x)\xi_n\Vert\rightarrow 0$ for all $x\in\mathcal{A}$. We conclude the proof using the density of $\mathcal{A}$ in $C_m(\Gq)$.
\end{proof}

We now turn to Property (T).

\begin{theorem}\label{Thm-Crossed-T}
The following holds:
\begin{enumerate}
\item If $\widehat{\Gq}$ has property (T), then $\Gamma$ has property $T$ and  $\chi(G)^\alpha$ is finite.
\item If $\widehat{\Gq}$ has property $(T)$ and $\alpha$ is compact then $\widehat{G}$ and $\Gamma$ have property $(T)$.
\item If $\widehat{G}$ has property (T) and $\Gamma$ has property (T), then $\widehat{\Gq}$ has property (T).
\end{enumerate}
\end{theorem}
\begin{proof}
$(1)$. This is the same proof as of assertion $1$ of Theorem \ref{ThmPropT}. First, we use the counit on $C_m(G)$ and the universal property of $C_m(\Gq)$ to construct a surjective $*$-homomorphism $C_m(\Gq)\rightarrow C^*(\Gamma)$ which intertwines the comultiplications. We then use \cite[Proposition 6]{Fi10} to conclude that $\Gamma$ has property $(T)$. To end the proof of $(1)$, we show that $\chi(G)^\alpha$ is discrete. Let $\chi_n\in \chi(G)^\alpha$ be any sequence such that $\chi_n\rightarrow \varepsilon_G$ weak* in $C_m(G)^*$. We define a unital $*$-homomorphism $\chi\,:\, C_m(G)\rightarrow\mathcal{B}(l^2(\N))$ by $(\chi(a)\xi)(n)=\chi_n(a)\xi(n)$ for all $a\in C_m(G)$ and $\xi\in l^2(\N)$. Since $\chi_n\in{\rm Sp}(C_m(G))^\alpha$ we have $\chi\circ\alpha_\gamma=\chi$ for all $\gamma\in\Gamma$. Hence, considering the trivial representation of $\Gamma$ on $l^2(\N)$ we obtain a covariant representation so there exists a unique unital $*$-homomorphism $\pi\,:\, C_m(\Gq)\rightarrow\mathcal{B}(l^2(\N))$ such that $\pi(u_\gamma\alpha(a))=\chi(a)$ for all $a\in C_m(G)$ and all $\gamma\in\Gamma$. Since $\chi_n\rightarrow \varepsilon_G$ weak* the sequence of unit vectors defined by $\xi_n=\delta_n\in l^2(\N)$ is a sequence of almost invariant vectors. By property $(T)$ we have $\varepsilon_\Gq\subset\pi$ which easily implies that, for some $n\in\N$, $\chi_n=\varepsilon_G$.

$(2)$. We repeat again the proof of assertion $2$ of Theorem \ref{ThmPropT}. By $(1)$, it suffices to show that $\widehat{G}$ has Property $(T)$. Let $\rho\,:\,C_m(G)\rightarrow\mathcal{B}(H)$ with $\varepsilon_G\prec\pi$ and define the compact group $K=\overline{\alpha(\Gamma)}\subset{\rm Aut}(G)$ with its Haar probability $\nu$. Note that any $x\in{\rm Aut}(G)$, in particular any $x\in K$, satisfies $\varepsilon_G\circ x=\varepsilon_G$. Define the covariant representation $(\rho_\alpha,v)$, $\rho_\alpha\,:\, C_m(G)\rightarrow\mathcal{B}({\rm L}^2(K,H))$ and $v\,:\,\Gamma\rightarrow\mathcal{U}({\rm L}^2(K,H))$ by $(\rho_\alpha(a)\xi)(x)=\rho(x^{-1}(a))\xi(x)$ and $(v_\gamma\xi)(x)=\xi(\alpha_{\gamma^{-1}}x)$. By the universal property of $C_m(\Gq)$ we get a unital $*$-homomorphism $\pi\,:\, C_m(\Gq)\rightarrow\mathcal{B}({\rm L}^2(K,H))$ such that $\pi(u_\gamma\alpha(a))=v_\gamma\rho_\alpha(a)$. Let $\xi_n\in H$ be a sequence of unit vectors such that $\Vert\rho(a)\xi_n-\varepsilon_G(a)\xi_n\Vert\rightarrow 0$ for all $a\in C_m(G)$ and define the vectors $\eta_n(x)=\xi_n$ for all $x\in K$, $n\in\N$.  Since $\nu$ is a probability, $\eta_n$ is a unit vector in ${\rm L}^2(K,H)$ for all $n\in\N$. Moreover, for all $a\in C_m(G)$ and $\gamma\in \Gamma$,
$$\Vert\pi(u_\gamma\alpha(a))\eta_n-\varepsilon_{\Gq}(u_\gamma\alpha(a))\xi_n\Vert^2=\int_K\Vert\rho(x^{-1}(\alpha_\gamma(a)))\xi_n-\varepsilon_G(a)\xi_n\Vert^2d\nu(x)\rightarrow 0,$$
where the convergence follows from the dominated convergence Theorem, since
$$\Vert\rho(x^{-1}(\alpha_\gamma(a)))\xi_n-\varepsilon_G(a)\xi_n\Vert=\Vert\rho(x^{-1}(\alpha_\gamma(a)))\xi_n-\varepsilon_G(x^{-1}(\alpha_\gamma(a)))\xi_n\Vert\rightarrow 0,$$
for all $a\in C_m(G)$, $x\in K$ and $\gamma\in\Gamma$ and the domination hypothesis is obvious since $\nu$ is a probability. Hence, $\varepsilon_{\Gq}\prec\pi$ and it follows from Property $(T)$ that there exists a non-zero $\pi$-invariant vector $\xi\in{\rm L}^2(G,H)$. In particular, for all $a\in C_m(G)$, $\pi(\alpha(a)\xi=\varepsilon_G(a)\xi$. Hence, $\nu(Y)>0$ where $Y=\{x\in K\,:\,\Vert\xi(x)\Vert>0\}$ and, for all $a\in C_m(G)$, $\nu(X_a)=0$ where $X_a=\{x\in K\,:\,\rho(x^{-1}(a))\xi(x)\neq\varepsilon_G(a)\xi(x)\}$. As in the proof of assertion $2$ of Theorem \ref{ThmPropT}, we deduce from the separability of $C_m(G)$ that there exists $x\in K$ for which $\xi(x)\neq 0$ and $\rho(x^{-1}(a))\xi(x)=\varepsilon_G(a)\xi(x)$ for all $a\in C_m(G)$. It follows that the vector $\eta:=\xi(x)\in H$ is a non-zero $\rho$-invariant vector.

$(3)$. We use the notations introduced in the proof of Theorem \ref{ThmRelT2}. Let $\pi\,:\,C_m(\Gq)\rightarrow\mathcal{B}(H)$ be a representation and consider the representation $\rho=\pi\circ\alpha\,:\, C_m(G)\rightarrow\mathcal{B}(H)$ and the unitary representation $v_\gamma=\pi(u_\gamma)$ of $\Gamma$ on $H$. Let $K_\pi=\{\xi\in H\,:\,\rho(a)\xi=\varepsilon_G(a)\xi\text{ for all }a\in C_m(G)\}$ and recall that the orthogonal projection onto $K_\pi$ is $P=\widetilde{\rho}(p_0)$ and that $\alpha_\gamma(p_0)=p_0$ for all $\gamma\in\Gamma$. Hence,
$v_\gamma Pv_{\gamma^{-1}}=\widetilde{\rho}(\alpha_\gamma(p_0))=P$ for all $\gamma\in\Gamma$, and it follows that $K_\pi$ is an invariant subspace of $\gamma\mapsto v_\gamma$. Suppose that $\varepsilon_\Gq\prec\pi$. By property $(T)$ of $\widehat{G}$, the space $K_\pi$ is non-zero and we can argue exactly as in the proof of Theorem \ref{ThmPropT} to conclude the result.
\end{proof}

\begin{remark}
It follows from the proof of the first assertion of the previous theorem that $C^\ast(\Gamma)$ is a compact quantum subgroup of the compact quantum group $\Gq$. Now, an irreducible representation of $\Gq$ of the form $u^x_\gamma$ (with dimension say $m$), when restricted to the subgroup $C^\ast(\Gamma)$, decomposes as a direct sum of $m$ copies of $\gamma$. It now follows from \cite[Theorem 6.3]{Pa13} that $C^\ast(\Gamma)$ is a central subgroup (see \cite[Definition 6.1]{Pa13}). Furthermore, $\Gamma$ induces an action on the chain group $c(G)$ \cite[Definition 7.4]{Pa13} of $G$ and it follows from Remark \ref{RmkFusion} that the chain group (and hence the center, see \cite[Section 7]{Pa13}) of $\Gq$ is the semidirect product group $c(G)\rtimes \Gamma$. 
\end{remark}

\begin{remark}
\textbf{(Kazhdan Pair for $\Gq$)} Let $(E_1,\delta_1)$ be a Kazhdan pair for $G$ and $(E_2,\delta_2)$ be a Kazhdan Pair for $\Gamma$. Then it is not hard to show that $E=(E_1\cup E_2)\subset \text{Irr}(\Gq)$ and $\delta=\text{min}(\delta_1,\delta_2)$ is a Kazhdan pair for $\Gq$. Indeed, let $\pi:C_m(\Gq)\rightarrow \mathcal{B}(H)$ be a $\ast$-representation having a $(E,\delta)$-invariant (unit) vector $\xi$. Then restricting to the subalgebra $C_m(G)$ (and denoting the corresponding representation by $\pi_G$), we get an $(E_1,\delta_1)$ invariant vector and hence, there is an invariant vector $\eta\in H$. We may assume $\|\xi-\eta\|<1$ (this follows from a quantum group version of Proposition 1.1.9 of \cite{BDV08}, which can be proved in an exactly similar fashion). Now, by restricting $\pi$ to $\Gamma$, denoting the corresponding representation by $u$, we have that the closed linear $u$-invariant subspace generated by $u_\gamma\eta, \gamma\in \Gamma$ (which we denote by $H_\eta$), is a subspace of the space of $\pi_G$-invariant vectors (as $u_\gamma\pi_G(a)u_\gamma^{-1}=\pi_G(\alpha_\gamma(a))$). Let $P_{H_{\eta}}$ denote the orthogonal projection onto $H_{\eta}$. Now, the vector $P_{H_{\eta}}\xi$, which is non-zero, as $\|\xi-\eta\|<1$, is an $(E_2,\delta_2)$-invariant vector for the representation $u$, restricted to $H_\eta$. So, there exists an $u$-invariant vector $\eta_0\in H_\eta$. This vector is, of course then, $\pi$-invariant and hence, we are done.   
\end{remark}

%%%%%%%%%%%%%%%%%%%%%%%%%%%%%%%
\subsection{Haagerup property}
%%%%%%%%%%%%%%%%%%%%%%%%%%%%%%%%

In this section, we study the relative co-Haagerup property of the pair $(G,\Gq)$ given by a crossed product and provide a characterization analogous to the bicrossed product case. We also extend a result of Jolissaint on Haagerup property for finite von Neumann algebra crossed product to a non-finite setting.  Thus, we can decide
whether $\rm{L}^\infty(\Gq)$ has the Haagerup property.  Finally, we provide sufficient conditions for $\widehat{\Gq}$ to posses the 
Haagerup property.

For the relative Haagerup property of crossed product, we obtain the following result similar to Theorem \ref{ThmRelH}. The proof is even simpler in the crossed product case, since $\alpha$ is an action by quantum automorphisms.

\begin{theorem}\label{Q-RelHaag}
The following are equivalent:
\begin{enumerate}
\item The pair $(G,\Gq)$ has the relative co-Haagerup property.
\item There exists a sequence $(\omega_n)_{n\in\N}$ of states on $C_m(G)$ such that
\begin{enumerate}
\item $\widehat{\omega}_n\in c_0(\widehat{G})$ for all $n\in\N$;
\item $\omega_n\rightarrow\varepsilon_G$ weak*;
\item $\Vert\alpha_\gamma(\omega_n)-\omega_n\Vert\rightarrow 0$ for all $\gamma\in\Gamma$.
\end{enumerate}
\end{enumerate}
\end{theorem}
\begin{proof}
$(1)\Rightarrow (2)$. The argument is exactly the same as the proof of $(1)\implies (2)$ of Theorem \ref{ThmRelH}.

$(2)\Rightarrow(1)$. We first prove the following claim.

\textbf{Claim.} \textit{If $(2)$ holds, then there exists a sequence $(\nu_n)_{n\in\N}$ of states on $C_m(G)$ satifying $(a)$, $(b)$ and $(c)$ and such that
$\alpha_\gamma(s(\nu_n))=s(\nu_n)$ for all $\gamma\in\Gamma$, $n\in\N$.}

\textit{Proof of the claim.} By the proof of the claim in Theorem \ref{ThmRelT2}, it suffices to check that, whenever $\nu$ is a state on $C_m(G)$ and $f\in \ell^1(\Gamma)$, we have $\widehat{\nu}\in c_0(\widehat{G})\Rightarrow \widehat{f*\nu}\in c_0(\widehat{G})$.

We first show that $\widehat{\nu}\in c_0(\widehat{G})\Rightarrow \widehat{\alpha_\gamma(\nu)}\in c_0(\widehat{G})$. Note that we still denote by $\alpha$ the action of $\Gamma$ on ${\rm Irr}(G)$ (see Remark \ref{RmkFusion}). Now let $\nu$ be a state on $C_m(G)$ such that $\widehat{\nu}\in c_0(\widehat{G})$ and let $\epsilon>0$. By assumptions, the set $F=\{x\in{\rm Irr}(G)\,:\,\Vert(\id\ot\nu)(u^x)\Vert_{\mathcal{B}(H_x)}\geq\epsilon\}$ is finite. Hence, the set
\begin{align*}
\{x\in{\rm Irr}(G)\,:\,\Vert(\id\ot\nu)(u^{\alpha_{\gamma^{-1}}(x)})\Vert_{\mathcal{B}(H_x)}\geq\epsilon\}=\{x\in{\rm Irr}(G)\,:\,\alpha_{\gamma^{-1}}(x)\in F\}=\alpha_\gamma(F)
\end{align*}
is also finite. Since $\widehat{\alpha_\gamma(\nu)}=\left((\id\ot\nu)(u^{\alpha_{\gamma^{-1}}(x)})\right)_{x\in{\rm Irr}(G)}$, it follows that $\widehat{\alpha_\gamma(\nu)}\in c_0(\widehat{G})$.

From this we can now conclude that for all $f\in \ell^1(\Gamma)$, we have $\widehat{\nu}\in c_0(\widehat{G})\Rightarrow \widehat{f*\nu}\in c_0(\widehat{G})$ as in the proof of the Claim in Theorem \ref{ThmRelH}.\hfill{$\Box$}

\vspace{0.2cm}

We can now finish the proof of the Theorem. Let $(\nu_n)_{n\in\N}$ be a sequence of states on $C_m(G)^*$ as in the Claim. As in the proof of Theorem \ref{ThmRelT2}, we construct a representation $\pi\,:\, C_m(\Gq)\rightarrow\mathcal{B}(H)$ with a sequence of unit vectors $\xi_n\in H$ such that $\Vert\pi(x)\xi_n-\varepsilon_\Gq(x)\xi_n\Vert\rightarrow 0$ for all $x\in C_m(\Gq)$ and $\nu_n=\omega_{\xi_n}\circ\pi\circ\alpha$. It follows that the sequence of states $\omega_n=\omega_{\xi_n}\circ\pi\in C_m(\Gq)^*$, satisfies $\omega_n\rightarrow\varepsilon_\Gq$ in the weak* topology and $\widehat{\omega_n\circ\alpha}=\widehat{\nu}_n\in c_0(\widehat{G})$ for all $n\in\N$.
\end{proof}

We now turn to the Haagerup property. We will need the following result which is of independent interest. This is the non-tracial version of \cite[Corollary 3.4]{Jo07} and the proof is similar. We include a proof for the convenience of the reader. We refer to \cite{CS13,OT13} for the Haagerup property for arbitrary von Neumann algebras.

\begin{proposition}\label{PropHaag}
Let $(M,\nu)$ be a von Neumann algebra with a f.n.s. $\nu$ and let $\alpha\,:\,\Gamma\curvearrowright M$ be an action which leaves $\nu$ invariant. If $\alpha$ is compact, $\Gamma$ and $M$ have the Haagerup property, then $\Gamma\ltimes M$ has the Haagerup property.
\end{proposition}

\begin{proof}
Let $H<{\rm Aut}(M)$ be the closure of the image of $\Gamma$ in ${\rm Aut}(M)$. By assumption $H$ is compact. Let ${\rm L}^2(M)$ denote the GNS space of $\nu$. 

We first make an easy observation. Whenever $\psi\,:\, M\rightarrow M$ is a ucp, normal and $\nu$-preserving map, then for all $x\in M$, the map $H\ni h\mapsto h^{-1}\circ\psi\circ h(x)\in M$ is $\sigma$-weakly continuous. Hence, we can define $\Psi(x)=\int_Hh^{-1}\circ\psi\circ h(x)dh$, where $dh$ is the normalized Haar measure on $H$. By construction, the map $\Psi\,:\,M\rightarrow M$ is ucp, $\nu$-preserving, $\Gamma$-equivariant and normal. Moreover, for all $\xi\in {\rm L}^2(M)$, the map $H\ni h\mapsto T_{h^{-1}}\circ T_\psi\circ T_h\xi\in{\rm L}^2(M)$, where $T_h$ and $T_\psi$ are respectively the ${\rm L}^2$-extensions of $h$ and $\psi$, is norm continuous. Consequently, $\int_HT_{h^{-1}}\circ T_\psi\circ T_hdh\in\mathcal{B}({\rm L}^2(M))$ and by definition of $\Psi$ we have that the ${\rm L}^2$-extension of $\Psi$ is given by $T_{\Psi}=\int_HT_{h^{-1}}\circ T_\psi\circ T_hdh\in\mathcal{B}({\rm L}^2(M))$. Let $\mathcal{B}$ denote the unit ball of ${\rm L}^2(M)$. Consider the set $A=\{h\mapsto T_{h^{-1}}\circ T_\psi\circ T_h\xi\,:\,\xi\in\mathcal{B}\}\subset C(H,\mathcal{B})$. It is easy to check that $A$ is equicontinuous and, since $T_\psi$ is compact, the set $A(h)=\{f(h)\,:\,f\in A\}$ is precompact for all $h\in H$. By Ascoli's Theorem, $A$ is precompact in $C(H,\mathcal{B})$. Since the map $H\times C(H,\mathcal{B})\rightarrow\mathcal{B}$, defined by $(h,f)\mapsto f(h)$ is continuous, the image of $H\times\overline{A}$ is compact and contains $\mathcal{B}_\psi=\{T_{h^{-1}}\circ T_\psi\circ T_h(\mathcal{B}),h\in H\}$. Since the image of $\mathcal{B}$ under $T_\Psi$ is contained in the closed convex hull of $\overline{\mathcal{B}_\psi}$, it follows that $T_\Psi$ is compact.

We use the standard notations $N=\Gamma\rtimes M=\{u_\gamma x\,:\,\gamma\in\Gamma, x\in M\}''\subset\mathcal{B}(\ell^2(\Gamma)\ot {\rm L}^2(M))$. We write $\widetilde{\nu}$ for the dual state of $\nu$ on $N$. Let $\psi_i$ be a sequence of normal, ucp, $\nu$-preserving and ${\rm L}^2$-compact maps on $M$ which converge pointwise in $\Vert \cdot\Vert_{2,\nu}$ to identity. Consider the sequence of $\nu$-preserving, ucp, normal, ${\rm L}^2$-compact and $\Gamma$-equivariant maps $\Psi_i$ given by $\Psi_i(x)=\int_Hh^{-1}\circ\psi\circ h(x)dh$ for all $x\in M$. Note that $(\Psi_i)_i$ is still converging pointwise in $\Vert \cdot \Vert_{2,\nu}$ to identity since, by the dominated convergence Theorem we have,
$$\Vert\Psi_i(x)-x\Vert_{2,\nu}=\left\Vert\int_Hh^{-1}(\psi_i(h(x))-h(x))dh\right\Vert_{2,\nu}\leq\int_H\Vert \psi_i(h(x))-h(x)\Vert_{2,\nu}dh\rightarrow 0.$$
By the $\Gamma$-equivariance, we can consider the normal ucp $\widetilde{\nu}$-preserving maps on $N$ given by $\widetilde{\Psi_i}(u_\gamma x)=u_\gamma\Psi_i(x)$. Observe that the sequence $(\widetilde{\Psi}_i)$ is still converging pointwise in $\Vert.\Vert_{2,\widetilde{\nu}}$ to identity and the ${\rm L}^2$-extension of $\widetilde{\Psi}_i$ is given by $T_{\widetilde{\Psi}_i}=1\ot T_{\Psi_i}\in\mathcal{B}(\ell^2(\Gamma)\ot{\rm L}^2(M))$.

Let $\phi_i$ be a sequence of positive definite and $c_0$ functions on $\Gamma$ converging to $1$ pointwise and consider the normal ucp $\widetilde{\nu}$-preserving maps on $N$ given by $\widetilde{\phi}_i(u_\gamma x)=\phi_i(\gamma)u_\gamma x$. Observe that the sequence $(\widetilde{\phi}_i)$ is converging pointwise in $\Vert\cdot \Vert_{2,\widetilde{\nu}}$ to identity and the ${\rm L}^2$-extension of $\widetilde{\phi}_i$ is given by $T_{\widetilde{\phi}_i}=T_{\phi_i}\ot 1\in\mathcal{B}(\ell^2(\Gamma)\ot{\rm L}^2(M))$, where $T_{\phi_i}(\delta_\gamma)=\phi_i(\gamma)\delta_\gamma$ is a compact operator on $\ell^2(\Gamma)$.

Hence, if we define the sequence of normal, ucp, $\widetilde{\nu}$-preserving maps on $N$ by $\varphi_{i,j}=\widetilde{\phi_j}\circ\widetilde{\Psi}_i$, we have $\varphi_{i,j}(u_\gamma x)=\phi_j(\gamma)u_\gamma\Psi_i(x)$; the sequence $(\varphi_{i,j})$ is converging pointwise in $\Vert\cdot\Vert_{2,\widetilde{\nu}}$ to identity and the ${\rm L}^2$-extension of $\varphi_{i,j}$ is given by $T_{\varphi_{i,j}}=T_{\phi_j}\ot T_{\Psi_i}\in\mathcal{B}(\ell^2(\Gamma)\ot{\rm L}^2(M))$ is compact.
\end{proof}

\begin{corollary}
The following holds.
\begin{enumerate}
\item If ${\rm L}^\infty(\Gq)$ has the Haagerup property, then ${\rm L}^\infty(G)$ and $\Gamma$ both have the Haagerup property.
\item If ${\rm L}^\infty(G)$ has the Haagerup property, $\alpha\,:\,\Gamma\curvearrowright{\rm L}^\infty(G)$ is compact and $\Gamma$ has the Haagerup property, then ${\rm L}^\infty(\Gq)$ has the Haagerup property.
.\end{enumerate}
\end{corollary}

\begin{proof}
$(1)$. Follows from the fact that there exists normal, faithful, Haar-state preserving conditional expectations from ${\rm L}(\Gq)$ to ${\rm L}(\Gamma)$ and to ${\rm L}^\infty(G)$. The former is given by $u_\gamma a\mapsto h_G(a)u_\gamma$ and the latter is given by $u_\gamma a\mapsto \delta_{\gamma,e}a$, $a\in {\rm L}^\infty(G)$ and $\gamma\in\Gamma$.

$(2)$. It is an immediate consequence of Proposition \ref{PropHaag}.
\end{proof}

\begin{theorem}\label{Thm-Crossed-H}
Suppose $\widehat{G}$ has the Haagerup property and $\Gamma$ has the Haagerup property, and further suppose that the action of $\Gamma$ on $G$ is compact. Then $\widehat{\Gq}$ has the Haagerup property. 
\end{theorem}

\begin{proof}
Since $\widehat{G}$ has the Haagerup property, this assures the existence of states $(\mu_n)_{n\in\mathbb{N}}$ on $C_m(G)$ such that $(1)$ $\widehat{\mu_n}\in c_0(\widehat{G})$ for all $n\in\N$ and $(2)$ $\mu_n\rightarrow \varepsilon_G$ weak$^\ast$. Our first task is to construct a sequence of $\alpha$-invariant states on $C_m(G)$ satisfying (1) and (2) above. This is similar to our arguments before (while dealing with property (T) and Haagerup property). Since the action of $\Gamma$ is compact, the closure of $\Gamma$ in Aut$(G)$ is compact, and we denote this subgroup by $H$. Letting $dh$ denote the normalized Haar measure on $H$, we define states $\nu_n\in C_m(G)^*$ by $\nu_n(a)=\int_H \mu_n(h^{-1}(a))dh$, for all $a\in C_m(G)$. It is easily seen that $\nu_n$ is invariant under the action of $\Gamma$ for each $n$. Now, since the action is compact, all orbits of the induced action on Irr$(G)$ are finite. We need this to show that $\mu_n$ satisfy (1) above. So, let $\epsilon>0$. As $\mu_n$ satisfied (1), the set 
$L=\{x\in \text{Irr}(G):\|(\id\ot\mu_n)(u^x)\|\geq \frac{\epsilon}{2}\}$ is finite and the set $K=H\cdot L\subset \text{Irr}(G)$ is also finite, as all the orbits are finite. For $h\in H\subset{\rm Aut}(G)$ and $x\in{\rm Irr}(G)$ write $V_{h,x}\in\mathcal{B}(H_x)$ to be the unique unitary such that $(\id\ot h^{-1})(u^x)=(V_{h,x}^*\ot 1)(\id\ot u^{h^{-1}(x)})(V_{h,x}\ot 1)$. If $x\notin K$ then, for all $h\in H$, $h^{-1}(x)\notin L$. Hence,  $\Vert(\id\ot\mu_n)(u^{h^{-1}(x)})\Vert<\frac{\epsilon}{2}$ for all $h\in H$ and it follows that
\begin{eqnarray*}
\Vert(\id\ot\nu_n)(u^x)\Vert&=&\left\Vert\int_H(\id\ot\mu_n)((\id\ot h^{-1})(u^x))dh\right\Vert\leq
\int_H\Vert V_{h,x}^*(\id\ot\mu_n)(u^{h^{-1}(x)})V_{h,x}\Vert dh\\
&\leq&\int_H\Vert(\id\ot\mu_n)(u^{h^{-1}(x)})\Vert dh\leq\frac{\epsilon}{2}<\epsilon\quad\text{for all }x\notin K.
\end{eqnarray*}

Hence, $\{x\in{\rm Irr}(G)\,:\,\Vert(\id\ot\nu_n)(u^x)\Vert\geq\epsilon\}\subset K$ is a finite set and (1) holds for $\nu_n$. To show that (2) holds, we first note that given any $a\in C_m(G)$, one has $\mu_n(h^{-1}(a))\rightarrow\varepsilon_G(h^{-1}(a))=\varepsilon_G(a)$ for all $h\in H$ (since $H$ acts on $G$ by quantum automorphisms). By the dominated convergence Theorem, we see that (2) holds for $\nu_n$. Now, since $\Gamma$ has the Haagerup property, we can construct states $\tau_n$ on $C^\ast(\Gamma)$ satisfying (1) and (2) above. And since the states $\mu_n$ on $C_m(G)$ are $\alpha$-invariant, we can construct the crossed product states $\phi_n=\tau_n\ltimes \mu_n$ on $C_m(\Gq)$ $($see \cite[Proposition and Definition 3.4]{Wa95b} and also \cite[Exercise 4.1.4]{BO08} for the case of c.c.p. maps$)$. The straightforward computations that need to be done to see that the sequence of states $(\phi_n)_{n\in \mathbb{N}}$ satisfy (1) and (2) above, are left to the reader. This then shows that $\widehat{\Gq}$ has the Haagerup property.
\end{proof}

%%%%%%%%%%%%%%%%%%%%%%%%%%%%%%%%%%%%%%%%%%%%%%%
\section{Examples}\label{section-examples}
%%%%%%%%%%%%%%%%%%%%%%%%%%%%%%%%%%%%%%%%%%%%%%%

For coherent reading, we have dedicated this section only to examples arising 
from both matched pairs and crossed products. It is to be noted that it is not hard to come up with examples of compact matched pairs 
of groups for which only one of the actions $\alpha$ or $\beta$ is non-trivial which means that the other is an action by group homomorphisms. 
However, it is harder to come up with examples for which both 
$\alpha$ and $\beta$ are non-trivial. We called such matched pairs \textit{non-trivial}. Starting out with a compact matched pair for 
which either $\alpha$ or $\beta$ is trivial, we describe a 
process to deform the original matched pair by what we call a \textit{crossed homomorphism} in such a way that we manufacture a 
new compact matched pair for which both actions are non-trivial. For pedagogical
reasons, we have made two subsections dealing with matched pairs: the first one $($Section $7.1.1)$, in 
which we describe how to perturb $\beta$ when it is trivial, followed by Section $7.1.2$ in which
we construe how to perturb $\alpha$ when it is trivial. It has to be noted that it is indeed possible to formalize our process of deformation in a unified way but, since such a formulation would increase the technicalities and would not produce any new explicit examples, we have chosen to separate the presentation in the two basic deformations described above. Our deformations are chosen carefully
so as to ensure that the geometric group theoretic properties $($that we have studied in detail 
throughout the paper$)$ passes from the initial bicrossed product to the one obtained after
the deformation very naturally. Such deformations also allow us to keep track of the invariants $\chi(\cdot)$
and $\rm{Int}(\cdot)$ of the associated compact quantum groups. These explicit constructions
allow us to exhibit: $(i)$ a pair of non-isomorphic non-trivial compact bicrossed products each of which has relative property $(T)$ but the dual does not have property $(T)$, $(ii)$ an infinite family of pairwise non-isomorphic non-trivial compact bicrossed products whose dual are non-amenable with the Haagerup property, $(iii)$ an infinite family
of pairwise non-isomorphic non-trivial compact bicrossed products whose duals have property $(T)$.

We also provide non-trivial examples of crossed products of a discrete group on a non-trivial compact quantum group in Section $7.2$. The action is coming from the conjugation action of a countable subgroup of $\chi(G)$ on the compact quantum group $G$. In this situation we completely understand weak amenability, $(RD)$, Haagerup property and property $(T)$ in terms of $G$ and $\Gamma$ and we also discuss explicit examples involving the free orthogonal and free unitary quantum groups.

%%%%%%%%%%%%%%%%%%%%%%%%%%%%%%%%%%%%%%%%%%%%%%%
\subsection{Examples of bicrossed products}\label{SectionEx}
%%%%%%%%%%%%%%%%%%%%%%%%%%%%%%%%%%%%%%%%%%%%%%%

In this section, we focus on deformation of actions in matched pairs when one of them is trivial. The analysis involved helps
to construct non-trivial examples. 

%%%%%%%%%%%%%%%%%%%%%%%%%%%%%%%%%%%%%%%%%%%%%%%
\subsubsection{From matched pairs with trivial $\beta$}
%%%%%%%%%%%%%%%%%%%%%%%%%%%%%%%%%%%%%%%%%%%%%%%

Let $\alpha$ be any action of a discrete group $\Gamma$ on a compact group $G$ by group homomorphisms. Taking $\beta$ to be the trivial action of $G$ on $\Gamma$, the relations  in Equation $(\ref{EqMatched})$ are satisfied and we get a compact matched pair. It is possible to upgrade this example in order to obtain a new compact matched pair $(\Gamma,\widetilde{G})$ for which the associated actions $\widetilde{\alpha}$ and $\widetilde{\beta}$ are both non-trivial.

Indeed, given an action $\alpha$ of the discrete group $\Gamma$ on the compact group $G$ and a continuous map $\chi\,:\,G\rightarrow \Gamma$,  we define a continuous map
$$G\times G\rightarrow G\,\,\,\text{by }(g,h)\mapsto g*h,\quad\text{where}\quad g*h=g\alpha_{\chi(g)}(h)\quad\text{for all }g,h\in G.$$

Observe that $e*g=g*e=g$ for all $g\in G$ if and only if $\chi(e)\in{\rm Ker}(\alpha)$. Moreover, it is easy to check that the map $(g,h)\mapsto g*h$ is associative if and only if $\chi(gh)^{-1}\chi(g)\chi(\alpha_{\chi(g)^{-1}}(h))\in{\rm Ker}(\alpha)$ for all $g,h\in G$. Finally, under the preceding hypothesis, the map $(g,h)\mapsto g*h$ turns $G$ into a compact group since the inverse of $g\in G$ exists and is given by $\alpha_{\chi(g)^{-1}}(g^{-1})$ and this inversion is a continuous map from $G$ to itself.

Hence it is natural to define a \textit{crossed homomorphism} as a continuous map $\chi\,:\,G\rightarrow\Gamma$ such that $\chi(e)=e$ and $\chi(gh)=\chi(g)\chi(\alpha_{\chi(g)^{-1}}(h))$ for all $g,h\in G$. Observe that the continuity of $\chi$, the compactness of $G$ and the discreteness of $\Gamma$ all together imply that the image of $\chi$ is finite. By the preceding discussion, any crossed homomorphism $\chi$ gives rise to a new compact group structure on $G$. We denote this compact group by $G_\chi$. Observe that, since the Haar measure on $G$ is invariant under $\alpha$, so the Haar measure on $G_\chi$ is equal to the Haar measure on $G$. Hence we have $G_\chi=G$ as probability spaces.

The group $G_\chi$ can also be defined as the graph of $\chi$ in the semi-direct product $H=\Gamma_\alpha\ltimes G$. Indeed, it is easy to check that the graph ${\rm Gr}(\chi)=\{(\chi(g),g)\,:\,g\in G\}$ of a continuous map $\chi\,:\,G\rightarrow\Gamma$, which is a closed subset of $H$, is a subgroup of $H$ if and only if $\chi$ is a crossed homomorphism. Moreover, the map $G_\chi\rightarrow {\rm Gr}(\chi)$, $g\mapsto(\chi(g),g)$, $g\in G$,  is an isomorphism of compact groups.

Since $G_\chi=G$ as topological spaces, $\alpha$ still defines an action of $\Gamma$ on the compact space $G_\chi$ by homeomorphisms. However, $\alpha$ may not be an action by group homomorphisms anymore. Actually, for $\gamma\in\Gamma$, $\alpha_\gamma$ is a group homomorphism of $G_\chi$ if and only if $\chi(g)^{-1}\gamma^{-1}\chi(\alpha_\gamma(g))\gamma\in{\rm Ker}(\alpha)$ for all $g\in G$ which happens for example if $\chi$ satisfies $\chi\circ\alpha_\gamma=\gamma\chi(\cdot)\gamma^{-1}$.

We define a continuous right action of $G_\chi$ on the discrete space $\Gamma$ by $\beta_g(\gamma)=\chi(\alpha_\gamma(g))^{-1}\gamma\chi(g)$ for all $\gamma\in\Gamma,g\in G$. It is an easy exercise to check that $\alpha$ and $\beta$ satisfy the relations in Equation $(\ref{EqMatched})$, hence, by Proposition \ref{PropRecMatched} we get a new compact matched pair $(\Gamma,G_\chi)$ with possibly non-trivial actions $\alpha$ and $\beta$. To see that the pair $(\Gamma,G_\chi)$ is matched without using Proposition \ref{PropRecMatched}, it suffices to view $\Gamma$ and $G_\chi$ as closed subgroups of $H=\Gamma_\alpha\ltimes G$ via the identification explained before and check that $\Gamma G_\chi=H$ and $\Gamma\cap G_\chi=\{e\}$. It is easy to check that the actions $\alpha$ and $\beta$ obtained by this explicit matching are the ones we did define.

Let $\Gq_\chi$ denote the bicrossed product associated with the matched pair $(\Gamma,G_\chi)$.

\begin{proposition}\label{PropExAlphaH}
If the action $\alpha\,:\,\Gamma\curvearrowright{\rm Irr}(G)$ has all orbits finite and the group $\Gamma$ has  the Haagrup property, then $\widehat{\Gq}_\chi$ has the Haagerup property for all crossed homomorphisms $\chi\,:\,G\rightarrow\Gamma$.
\end{proposition}

\begin{proof}
Recall that if $\alpha\,:\,\Gamma\curvearrowright G$ is an action by compact group automorphisms, then the action $\alpha\,:\,\Gamma\curvearrowright{\rm L}^\infty(G)$ is compact if and only if the image of $\Gamma$ in ${\rm Aut}(G)$ is precompact which in turn is equivalent to
 the associated action of $\Gamma$ on ${\rm Irr}(G)$ to have all orbits finite. Now let $\chi\,:\,G\rightarrow\Gamma$ be a crossed homomorphism. Since $G_\chi=G$ as compact spaces and as probability spaces, the action $\alpha\,:\,\Gamma\curvearrowright{\rm L}^\infty(G)$ is compact if and only if the action $\Gamma\curvearrowright {\rm L}^\infty(G_\chi)$ is compact and the former is equivalent to 
the action $\Gamma\curvearrowright{\rm Irr}(G)$ to have all orbits finite. Hence, the proof follows from assertion $4$ of Corollary \ref{CorMatched}.
\end{proof}

Observe that a continuous group homomorphism $\chi\,:\,G\rightarrow\Gamma$ is a crossed homomorphism if and only if $\chi\circ\alpha_\gamma=\chi$ for all $\gamma\in{\rm Im}(\chi)$.

Now we give a systematic way to construct explicit non-trivial examples of the situation considered in the first part of this section. So, consider a non-trivial action $\alpha$ of a countable discrete group $\Gamma$ on a compact group $G$ by group homomorphisms and let $\Lambda<\Gamma$ be a finite subgroup. Define the action $\alpha^\Lambda$ of $\Gamma$ on $G^\Lambda=\Lambda\times G$ by $\alpha^\Lambda_\gamma(r,g)=(r,\alpha_\gamma(g))$ and the $\alpha^\Lambda$-invariant group homomorphism $\chi\,:\,G^\Lambda\rightarrow\Gamma$ by $\chi(r,g)=r$, $r\in \Lambda, g\in G, \gamma\in\Gamma$. Thus, we get a compact matched pair $(\Gamma,G^\Lambda_\chi)$ where $G^\Lambda_\chi=\Lambda\times G$ as a compact space and the group law is given by $(r,g)\cdot(s,h)=(r,g)\alpha_{\chi(r,g)}(s,h)=(rs,g\alpha_r(h))$, $r,s\in \Lambda$ and $g,h\in G$. Hence, $G^\Lambda_\chi=\Lambda\,_\alpha\ltimes G$ as a compact group and the action $\beta$ of $G^\Lambda_\chi$ on $\Gamma$ is given by $\beta_{(r,g)}(\gamma)=r^{-1}\gamma r$, $r\in\Lambda, g\in G,\gamma\in \Gamma$. Hence, $\beta$ is non-trivial if and only if $\Lambda$ is not in the center of $\Gamma$.

One has $\left(G^\Lambda_\chi\right)^\alpha=\Lambda\times G^{\alpha}$ and, since the action $\beta$ of $\left(G^\Lambda_\chi\right)^\alpha$ on $\Gamma$ is by inner automorphisms, the associated action on ${\rm Sp}(\Gamma)$ is trivial. Hence, if we denote by $\Gq_\Lambda$ the associated bicrossed product, then Proposition \ref{PropInt} implies that $\chi(\Gq_\Lambda)\simeq \Lambda\times G^{\alpha}\times{\rm Sp}(\Gamma)$. We claim that there is a canonical group isomorphism $\pi\,:\,{\rm Sp}(G^\Lambda_\chi)\rightarrow{\rm Sp}(\Lambda)\times{\rm Sp}_{\Lambda}(G)$, where ${\rm Sp}_{\Lambda}(G)=\{\omega\in{\rm Sp}(G)\,:\,\omega\circ\alpha_r=\omega\text{ for all }r\in \Lambda\}$ is a subgroup of ${\rm Sp}(G)$. Indeed, denoting by $\iota_{G}\,:\,G\rightarrow G^\Lambda_\chi$, $g\mapsto(1,g)$ and $\iota_{\Lambda}\,:\,\Lambda\rightarrow G^\Lambda_\chi$, $r\mapsto(r,1)$ the two canonical injective (and continuous) group homomorphisms, we may define $\pi(\omega)=(\omega\circ\iota_\Lambda,\omega\circ\iota_{G})$. Using the relations in the semi-direct product and the fact that $\omega$ is invariant on conjugacy classes, we see that $\omega\circ\iota_{G}\in{\rm Sp}_{\Lambda}(G)$. Since $G^\Lambda_\chi$ is generated by $\iota_\Lambda(\Lambda)$ and $\iota_G(G)$, so $\pi$ is injective. The surjectivity of $\pi$ follows from the universal property of semi-direct products.

Observe that $\Gamma^\beta=C_\Gamma(\Lambda)$ is the centralizer of $\Lambda$ in $\Gamma$. Since, $\alpha_\gamma({\rm Sp}_{\Lambda}(G))={\rm Sp}_{\Lambda}(G)$ for every $\gamma\in C_\Gamma(\Lambda)$, so $\alpha$ induces a right action of $C_\Gamma(\Lambda)$ on ${\rm Sp}_{\Lambda}(G)$ and we have, by Proposition \ref{PropInt}, ${\rm Int}(\Gq_\Lambda)\simeq {\rm Sp}(\Lambda)\times\left({\rm Sp}_{\Lambda}(G)\rtimes_\alpha C_\Gamma(\Lambda)\right)$.

We will write $\Gq=\Gq_{\{1\}}$. We have thus proved the first assertion of the following theorem.

\begin{theorem}\label{ThmExAlpha}
Let $\Lambda<\Gamma$ be any finite subgroup. Then the following holds.
\begin{enumerate}
\item $\chi(\Gq_\Lambda)\simeq \Lambda\times G^{\alpha}\times{\rm Sp}(\Gamma)$ and ${\rm Int}(\Gq_\Lambda)\simeq {\rm Sp}(\Lambda)\times\left({\rm Sp}_{\Lambda}(G)\rtimes_\alpha C_\Gamma(\Lambda)\right)$.
\item The following conditions are equivalent.
\begin{itemize}
\item $(G,\Gq)$ has the relative property $(T)$.
\item $(G^\Lambda_\chi,\Gq_\Lambda)$ has the relative property $(T)$.
\end{itemize}
\item The following conditions are equivalent.
\begin{itemize}
\item $(G,\Gq)$ has the relative Haagerup property.
\item $(G^\Lambda_\chi,\Gq_\Lambda)$ has the relative Haagerup property.
\end{itemize}
\item If the action $\Gamma\curvearrowright{\rm Irr}(G)$ has all orbits finite and $\Gamma$ has the Haagerup property, then $\widehat{\Gq}_\Lambda$ has the Haagerup property.
\item If the action $\Gamma\curvearrowright{\rm Irr}(G)$ has all orbits finite and $\Gamma$ is weakly amenable, then $\widehat{\Gq}_\Lambda$ is weakly amenable and $\Lambda_{cb}(\widehat{\Gq}_\Lambda)\leq \Lambda_{cb}(\Gamma)$.
\end{enumerate}
\end{theorem}

\begin{proof}
$(2)$. $(\Downarrow)$ Suppose that the pair $(G^\Lambda_\chi,\Gq_\Lambda)$ does not have the relative property $(T)$. Let $(\mu_n)$ be a sequence of Borel probability measures on $\Lambda\times G$ satisfying the conditions of assertion $2$ of Theorem \ref{ThmRelT}. Since $\{e\}\times G$ is open and closed in $\Lambda\times G$, we have $1_{\{e\}\times G}\in C(\Lambda\times G)$, and since $\mu_n\rightarrow\delta_{(e,e)}$ in the weak* topology we deduce that $\mu_n(\{e\}\times G)\rightarrow 1$. Hence, we may and will assume that  $\mu_n(\{e\}\times G)\neq 0$ for all $n\in\N$. Define a sequence $(\nu_n)$ of Borel probability measures on $G$ by $\nu_n(A)=\frac{\mu_n(\{e\}\times A)}{\mu_n(\{e\}\times G)}$, where $A\subset G$ is Borel. Then $\nu_n(\{e\})=\mu_n(\{(e,e)\})=0$ for all $n\in\N$ and it is easy to check that, for all $F\in C(G)$, $1_{\{e\}}\ot F\in C(\Lambda\times G)$ and $\int_{G}F d\nu_n=\frac{1}{\mu_n(\{e\}\times G)}\int_{\Lambda\times G}1_{\{e\}}\ot Fd\mu_n$. It follows from this formula and the fact that $\mu_n\rightarrow\delta_{(e,e)}$ in the weak* topology that we also have $\nu_n\rightarrow\delta_{e}$ in the weak* topology. Finally, the previous formula also implies that, for all $F\in C(G)$,
\begin{eqnarray*}
\vert\alpha_\gamma(\nu_n)(F)-\nu_n(F)\vert &=&
\frac{1}{\mu_n(\{e\}\times G)}\vert\alpha^\Lambda_\gamma(\mu_n)(1_{\{e\}}\ot F)-\mu_n(1_{\{e\}}\ot F)\vert\\
&\leq&\frac{\Vert1_{\{e\}}\ot F\Vert}{\mu_n(\{e\}\times G)}\Vert\alpha^\Lambda_\gamma(\mu_n)-\mu_n\Vert\leq\frac{\Vert F\Vert}{\mu_n(\{e\}\times G)}\Vert\alpha^\Lambda_\gamma(\mu_n)-\mu_n\Vert.
\end{eqnarray*}
Hence, $\Vert\alpha_\gamma(\nu_n)-\nu_n\Vert\leq\frac{\Vert\alpha^\Lambda_\gamma(\mu_n)-\mu_n\Vert}{\mu_n(\{e\}\times G)}\rightarrow 0$ and thus $(G,\Gq)$ does not have the relative property $(T)$.

$(\Uparrow)$ Now suppose that the pair $(G,\Gq)$ does not have the relative property $(T)$. Let $(\mu_n)$ be a sequence of Borel probability measures on $G$ satisfying the conditions of assertion $2$ of Theorem \ref{ThmRelT}. For each $n$ define the probability measure $\nu_n$ on $G^\Lambda_\chi=\Lambda\,_\alpha\ltimes G$ by $\nu_n=\delta_e\ot\mu_n$. We have $\nu_n(\{e,e\})=\mu_n(\{e\})=0$ and $\int_{G^\Lambda_\chi}Fd\nu_n=\int_GF(e,g)d\mu_n(g)$ for all $F\in C(G^\Lambda_\chi)$. Hence $\nu_n\rightarrow\delta_e$ in the weak* topology. Moreover, since for all $F\in C(G_\chi^\Lambda)$, we have
\begin{eqnarray*}
\vert\alpha_\gamma^\Lambda(\nu_n)(F)-\nu_n(F)\vert &=&
\left\vert\int_GF(e,\alpha_\gamma(g))d\mu_n(g)-\int_GF(e,g)d\mu_n\right\vert
=\vert\alpha_\gamma(\mu_n)(F_e)-\mu_n(F_e)\vert\\
&\leq& \Vert F_e\Vert\,\Vert\alpha_\gamma(\mu_n)-\mu_n\Vert
\leq\Vert F\Vert\,\Vert\alpha_\gamma(\mu_n)-\mu_n\Vert,
\end{eqnarray*}
where $F_e=F(e,\cdot)\in C(G)$, we have $\Vert\alpha^\Lambda_\gamma(\nu_n)-\nu_n\Vert\leq\Vert\alpha_\gamma(\mu_n)-\mu_n\Vert\rightarrow 0$.

$(3)$. By Theorem \ref{ThmRelH} and the proof of $(2)$, it suffices to prove the following claim.

\textbf{Claim.} \textit{Let $\alpha\,:\,\Lambda\curvearrowright G$ be an action of a finite group $\Lambda$ on a compact group $G$ by group automorphisms and define the compact group $H=\Lambda\,_\alpha\ltimes G$. The following holds.
\begin{enumerate}[(a)]
\item Let $\mu$ be a Borel probability measure on $G$ and define the Borel probability measure $\nu$ on $H$ by $\nu=\delta_e\ot\mu$. If $\widehat{\mu}\in C^*_r(G)$ then $\widehat{\nu}\in C^*_r(H)$.
\item Let $\mu$ be a Borel probability on $H$ such that $\mu(\{e\}\times G)\neq 0$ and define the Borel probability $\nu$ on $G$ by $\nu(A)=\frac{\mu(\{e\}\times A)}{\mu(\{e\}\times G)}$ for all $A\in\mathcal{B}(G)$. If $\widehat{\mu}\in C^*_r(H)$ then $\widehat{\nu}\in C^*_r(G)$.
\end{enumerate}}

\textit{Proof of the claim.} Let $\lambda^G$ and $\lambda^H$ denote the left regular representations of $G$ and $H$ respectively. For $F\in C(G)$ (resp. $F\in C(H)$), write
$\lambda^G(F)$ (resp. $\lambda^H(F)$) the convolution operator by $F$ on ${\rm L}^2(G,\mu_G)$ (resp. ${\rm L}^2(H,\mu_H)$), where $\mu_G$ (resp. $\mu_H$), is the Haar probability on $G$ (resp. $H$). Observe that $\mu_H=m\ot\mu_G$, where $m$ is the normalized counting measure on $\Lambda$.

$(a)$. Recall that, for all $F\in C(H)$, $\int_{H}Fd\nu=\int_GF(e,g)d\mu(g)$. Moreover, using the definition of the group law in $H$, we find that $\lambda^H_{(e,g)}=1\ot\lambda^G_g\in\mathcal{B}(l^2(\Lambda)\ot{\rm L}^2(G))$, for all $g\in G$. It follows that
$$\widehat{\nu}=\int_G\lambda^H_{(e,g)}d\mu(g)=\int_G(1\ot\lambda^G_{g})d\mu(g)=1\ot\widehat{\mu}\in M(C^*_r(H))\subset\mathcal{B}(l^2(\Lambda)\ot{\rm L}^2(G)).$$
Note that for all $F\in C(G)$, $1_{\{e\}}\ot F\in C(H)$, since $\Lambda$ is finite. We claim that $\lambda^H(1_{\{e\}}\ot F)=\frac{1}{\vert\Lambda\vert}(1\ot\lambda^G(F))$. Indeed, 
\begin{eqnarray*}
\lambda^H(1_{\{e\}}\ot F)&=&\int_H\delta_{r,e}F(g)\lambda^H_{(e,g)}d\mu_H(r,g)
=\int_H\delta_{r,e}F(g)(1\ot\lambda^G_g)d\mu_H(r,g)\\
&=&\int_G\left(\frac{1}{\vert\Lambda\vert}\sum_{r\in\Lambda}\delta_{r,e}F(g)(1\ot\lambda^G_g)\right)d\mu_G(g)=\frac{1}{\vert\Lambda\vert}(1\ot\lambda^G(F)).
\end{eqnarray*}
Suppose that $\widehat{\mu}\in C^*_r(G)$ and let $F_n\in C(G)$ be a sequence such that $\Vert\widehat{\mu}-\lambda^G(F_n)\Vert\rightarrow 0$. Hence, $1\ot\lambda^G(F_n)\rightarrow\widehat{\nu}$. Since $1\ot\lambda^G(F_n)=\vert\Lambda\vert\lambda^H(1_{\{e\}}\ot F_n)\in C^*_r(H)$ $\forall n\in\N$, we have $\widehat{\nu}\in C^*_r(H)$.

$(b).$ Recall that, for all $F\in C(G)$, $1_{\{e\}}\ot F\in C(\Lambda\times G)=C(H)$ and $\int_{G}Fd\nu=\frac{1}{\mu(\{e\}\times G)}\int_{\Lambda\times G}1_{\{e\}}\ot Fd\mu$. Using the definition of the group law in $H$, an easy computation shows that for all $r\in\Lambda$, $\xi\in {\rm L}^2(G)$, $\lambda^H_{(r,g)}(\delta_e\ot\xi)=\delta_r\ot\lambda^G_g(\xi\circ\alpha_{r^{-1}})$. It follows that,
\begin{eqnarray*}
\langle\widehat{\nu}\xi,\eta\rangle&=&\int_G\langle\lambda^G_g\xi,\eta\rangle d\nu(g)=\frac{1}{\mu(\{e\}\times G)}\int_{\Lambda\times G}\delta_{e,r}\langle\lambda^G_g\xi,\eta\rangle d\mu(r,g)\\
&=&\frac{1}{\mu(\{e\}\times G)}\int_{\Lambda\times G}\langle\lambda^H_{(r,g)}\delta_e\ot\xi,\delta_e\ot\eta\rangle d\mu(r,g)\quad\text{for all }\xi,\eta\in{\rm L}^2(G).
\end{eqnarray*}
Hence, $\widehat{\nu}=\frac{1}{\mu(\{e\}\times G)}V^*\widehat{\mu}V$, where $V\,:\,{\rm L}^2(G)\rightarrow l^2(\Lambda)\ot{\rm L}^2(G)={\rm L}^2(H)$ is the isometry defined by $V\xi=\delta_e\ot\xi$, $\xi\in{\rm L}^2(G)$. To end the proof it suffices to show that $V^* C^*_r(H)V\subset C^*_r(G)$. 

Let $F\in C(H)$ and define $F_e\in C(G)$ by $F_e(g)=F(e,g)$, $g\in G$. We will actually show that $V^*\lambda^H(F)V=\frac{1}{\vert \Lambda\vert}\lambda^G(F_e)$ and this will finish the argument. For $\xi,\eta\in{\rm L}^2(G)$, we have
\begin{eqnarray*}
\langle V^*\lambda^H(F)V\xi,\eta\rangle
&=&\langle\lambda^H(F)\delta_e\ot\xi,\delta_e\ot\eta\rangle
=\int_HF(r,g)\langle\lambda^H_{(r,g)}\delta_e\ot\xi,\delta_e\ot\eta\rangle d\mu_H(r,g)\\
&=&\int_H\delta_{r,e}F(e,g)\langle\lambda^G_g\xi,\eta\rangle d\mu_H(r,g)
=\int_G\frac{1}{\vert\Lambda\vert}\sum_{r\in \Lambda}\delta_{r,e}F(e,g)\langle\lambda^G_g\xi,\eta\rangle d\mu_G(g)\\
&=&\frac{1}{\vert\Lambda\vert}\int_GF(e,g)\langle\lambda^G_g\xi,\eta\rangle d\mu_G(g)
=\frac{1}{\vert\Lambda\vert}\langle\lambda^G(F_e)\xi,\eta\rangle.
\end{eqnarray*}
\hfill{$\Box$}

 $(4)$. It is easy to check that, if $\alpha\,:\,\Gamma\curvearrowright G$ is compact then $\alpha^\Lambda=\id\ot\alpha\,:\,\Gamma\curvearrowright \Lambda\times G$ is compact, for all finite group $\Lambda$. Hence, the proof follows from Proposition \ref{PropExAlphaH}.
 
$(5)$. Observe that, for a general compact matched pair $(\Gamma, G)$ with associated actions $\alpha$ and $\beta$, the continuity of $\beta$ forces each stabilizer subgroup $G_\gamma$, for $\gamma\in\Gamma$, to be open, hence finite index by compactness of $G$. Consider the closed normal subgroup $G_0=\cap_{\gamma\in \Gamma}G_\gamma={\rm Ker}(\beta)<G$. Equation \ref{EqMatched} implies that $G_0$ is globally invariant under $\alpha$ and the $\alpha$-action of $\Gamma$ on $G_0$ is by group automorphisms. Hence, we may consider the crossed product quantum group $\Gq_0$, with $C_m(\Gq_0)=\Gamma_{\alpha, f}\ltimes C(G_0)$, which is a quantum subgroup (in fact normal subgroup in the sense of Wang \cite{Wa09}) of the bicrossed product quantum group $\Gq$ with $C_m(\Gq)=\Gamma_{\alpha, f}\ltimes C(G)$. This is because $G_0$ is globally invariant under the action $\alpha$ of $\Gamma$ and hence, by the universal property, we have a surjective unital $\ast$-homomorphism $\rho:\Gamma_{\alpha, f}\ltimes C(G)\rightarrow \Gamma_{\alpha, f}\ltimes C(G_0)$ which is easily seen to intertwines the comultiplications. Since $\rho$ acts as identity on $C_m(\Gamma)$, it follows using Theorem \ref{ThmBicrossed}(2) that $C_m(\Gq/\Gq_0)=\alpha(C_m(G/G_0))$ (see Definition \ref{Defn:FinInd}). Hence, if we assume that $G_0$ is a finite index subgroup of $G$, then $\Gq_0$ is a finite index subgroup of $\Gq$. If we further assume that $\Gamma$ is weakly amenable and the action $\alpha$ of $\Gamma$ on $G$ is compact then the action $\alpha$ of $\Gamma$ restricted to $G_0$ is also compact and Theorem \ref{ThmWA} (with the fact that $\Gq_0$ is Kac) implies that $\widehat{\Gq}_0$ is weakly amenable with $\Lambda_{cb}(\widehat{\Gq}_0)\leq \Lambda_{cb}(\Gamma)$. Using part $(2)$ of Theorem \ref{wab}, we conclude that $\widehat{\Gq}$ is weakly amenable and $\Lambda_{cb}(\widehat{\Gq})\leq \Lambda_{cb}(\Gamma)$. In our case, with $G=G_\chi^\Lambda$, the finiteness of $\Lambda$ forces $G_0$ to be always of finite index in $G$. Since, by assumption, the action of $\Gamma$ on Irr$(G)$ has all orbits finite, we conclude, as in the proof of Proposition \ref{PropExAlphaH}, that the action $\alpha$ is compact.
\end{proof}

\begin{example}(\textbf{Relative Property (T)})\label{ExRelT}
Take $n\in\N$, $n\geq 2$, $\Gamma={\rm SL}_n(\Z)$, $G=\mathbb{T}^n$ and $\alpha$ the canonical action of ${\rm SL}_n(\Z)$ on $\mathbb{T}^n={\rm Sp}(\Z^n)$ coming from the linear action of ${\rm SL}_n(\Z)$ on $\Z^n$. Taking a finite subgroup $\Lambda<{\rm SL}_n(\Z)$, we manufacture a compact bicrossed product $\Gq_\Lambda$ with non-trivial actions $\alpha$ and $\beta$ $($described in the beginning of this section$)$ whenever $\Lambda$ is a non-central subgroup. Note that $\left(\mathbb{T}^n\right)^{{\rm SL}_n(\Z)}=\{e\}$ hence $\chi(\Gq_\Lambda)\simeq \Lambda\times{\rm Sp}({\rm SL}_n(\Z))$.

Suppose $n\geq 3$. In this case, $D({\rm SL}_n(\Z))={\rm SL}_n(\Z)$, where $D(F)$ denotes the derived subgroup of a group $F$. Since every element of ${\rm Sp}({\rm SL}_n(\Z))$ is trivial on commutators, we have ${\rm Sp}({\rm SL}_n(\Z))=\{1\}$, for all $n\geq 3$. It follows that $\chi(\Gq_\Lambda)\simeq \Lambda$. Hence, for all $n,m\geq 3$ and all finite subgroups $\Lambda<{\rm SL}_n(\Z)$, $\Lambda'<{\rm SL}_m(\Z)$, we have $\Gq_\Lambda\simeq\Gq_{\Lambda'}$ implies $\Lambda\simeq \Lambda'$.

However, for $n=2$, the group ${\rm Sp}({\rm SL}_2(\Z))$ is non-trivial. Actually, we have
\begin{eqnarray}\label{EqDualSL2}
{\rm Sp}({\rm SL}_2(\Z))\simeq\{(k,l)\in \Z/4\Z\times \Z/6\Z\,:\,k\equiv l\text{ mod }2\},
\end{eqnarray}
which is a finite group of order $12$. Indeed, by the well known isomorphism ${\rm SL}_2(\Z)\simeq\Z/4\Z\underset{\Z/2\Z}{*}\Z/6\Z$, it suffices to compute the group of $1$-dimensional unitary representations of an amalgamated free product $\Gamma_1\underset{\Sigma}{*}\Gamma_2$. It is easy to check that the map $\psi\,:\,{\rm Sp}(\Gamma_1\underset{\Sigma}{*}\Gamma_2)\rightarrow T$ defined by $\psi(\omega)=(\omega\vert_{\Gamma_1},\omega\vert_{\Gamma_2})$, where $T$ is the subgroup of ${\rm Sp}(\Gamma_1)\times{\rm Sp}(\Gamma_2)$ defined by $T=\{(\omega,\mu)\in{\rm Sp}(\Gamma_1)\times{\rm Sp}(\Gamma_2)\,:\,\omega\vert_\Sigma=\mu\vert_\Sigma\}$, is an isomorphism of compact groups. Hence, using the canonical identification ${\rm Sp}(\Z/m\Z)\simeq\Z/m\Z$, we obtain the isomorphism in Equation $(\ref{EqDualSL2})$.

Since the pair $(\Z^2,{\rm SL}_2(\Z)\ltimes\Z^2)$ has the relative property $(T)$, we deduce from Theorem \ref{ThmExAlpha} that, for any finite subgroup $\Lambda<{\rm SL}_2(\Z)$, the pair $(G^\Lambda_\chi,\Gq_\Lambda)$ has the relative property $(T)$. Identifying ${\rm SL}_2(\Z)$ with $\Z/4\Z\underset{\Z/2\Z}{*}\Z/6\Z$, one finds that every finite subgroup is conjugated to $\{1\}$ or $\Z/2\Z$ or $\Z/4\Z$ or $\Z/6\Z$. The only non-central subgroups are conjugated to $\Lambda_1=\Z/4\Z$ or $\Lambda_2=\Z/6\Z$. Hence, we get two non-trivial compact bicrossed products $\Gq_{\Lambda_i}$, $i=1,2$, such that $(G^{\Lambda_i}_\chi,\Gq_{\Lambda_i})$ has the relative property $(T)$ and $\widehat{\Gq_{\Lambda_i}}$ does not have property $(T)$ since ${\rm SL}_2(\Z)$ has the Haagerup property. Moreover, $\Gq_{\Lambda_1}$ and $\Gq_{\Lambda_2}$ are not isomorphic since $\vert\Lambda_1\vert\neq\vert\Lambda_2\vert$.
\end{example}

\begin{remark} (\textbf{Haagerup Property and Weak Amenability})
We depict here a procedure to construct compact bicrossed products with the Haagerup property and Weak Amenability. Suppose that $\Gamma$ is a countable subgroup of a compact group $G$ and consider the action $\alpha\,:\,\Gamma\curvearrowright G$ by inner automorphisms i.e. $\alpha_\gamma(g)=\gamma g\gamma^{-1}$ for all $\gamma\in\Gamma$, $g\in G$. Let $\Lambda<\Gamma$ be any finite subgroup and consider the matched pair $(G^\Lambda_\chi,\Gamma)$ introduced earlier in this section. Let $\Gq_\Lambda$ be the bicrossed product. Observe that, since the action $\alpha$ is inner, the associated action on ${\rm Irr}(G)$ is trivial. Indeed, for any unitary representation $\pi$ of $G$, the unitary $\pi(\gamma)$ is an intertwiner between $\alpha_\gamma(\pi)$ and $\pi$ for all $\gamma\in \Gamma$. Hence, if $\Gamma$ has the Haagerup property, then for any finite subgroup $\Lambda<\Gamma$ the bicrossed product $\widehat{\Gq}_\Lambda$ has the Haagerup property. Similarly, if $\Gamma$ is weakly amenable, then for any finite subgroup $\Lambda<\Gamma$ the bicrossed product $\widehat{\Gq}_\Lambda$ is weakly amenable and $\Lambda_{cb}(\widehat{\Gq}_\Lambda)\leq \Lambda_{cb}(\Gamma)$. 
\end{remark}
%%%%%%%%%%%%%%%%%%%%%%%%%%%%%%%%%%%%%%%%%%%%%%%
\subsubsection{From matched pair with trivial $\alpha$}
%%%%%%%%%%%%%%%%%%%%%%%%%%%%%%%%%%%%%%%%%%%%%%%
In this section, we consider the dual situation, i.e., starting with a matched pair with $\alpha$ being trivial and modifying it to some non-trivial action for a probably different matched pair.

Let $\beta$ be any continuous right action of the compact group $G$ on the discrete group $\Gamma$ by group automorphisms. Taking $\alpha$ to be the trivial action of $\Gamma$ on $G$, the relations in Equation $(\ref{EqMatched})$ are satisfied and we get a matched pair.

\begin{remark}
Note that if the group $\Gamma$ is finitely generated then the right semi-direct product group $H=\Gamma\rtimes_\beta G$ is virtually a direct product. In other words, there is a finite index subgroup of $H$ which is a direct product of a subgroup of $G$ (which acts trivially on $\Gamma$) and $\Gamma$.

Indeed, since $\Gamma$ is discrete and $\beta$ is continuous, the stabilizer subgroup 
$G_{\gamma}:= \{g\in G: \gamma\cdot g=\gamma\}$ is open in $G$ for all $\gamma\in\Gamma$. Since $G$ is compact, $G_\gamma$ has finite index in $G$. Now consider the subgroup $G_{\beta}=\cap_{\gamma \in \Gamma}G_{\gamma}$, which acts trivially on $\Gamma$. In case $\Gamma$ is finitely generated, it follows that $G_\beta$ is also finite index in $G$ and thus the direct product $\Gamma\times G_{\beta}$ is a finite index subgroup of $H$.

However, if the discrete group is not finitely generated then this need not be the case. For instance, let a compact group $K$  act on a finite group $F$ non-trivially. Let $K_n=K$ for $n\in \mathbb{N}$. One can then induce, in the natural way, an action of the compact group $G=\prod_{n\in\mathbb{N}} K_n$ on the discrete group $\Gamma=\oplus_{n\in \mathbb{N}}F_n$, where $F_n=F$ for all $n$. In this case, it is easy to see that the subgroup $G_\beta$ is not of finite index.
\end{remark}

Getting back to the process of modifying $\alpha$, we call a map $\chi\,:\,\Gamma\rightarrow G$ a \textit{crossed homomorphism} if $\chi(e)=e$ and $\chi(rs)=\chi(\beta_{\chi(s)^{-1}}(r))\chi(s)$ for all $r,s\in\Gamma$. Given a crossed homomorphism, we define a new discrete group $\Gamma_\chi$ which is equal to $\Gamma$ as a set and the group multiplication is given by $r*s=\beta_{\chi(s)}(r)s$ for all $r,s\in\Gamma$. As before, $\Gamma_\chi$ is canonically isomorphic to the graph ${\rm Gr}(\chi)=\{(\gamma,\chi(\gamma))\,:\,\gamma\in\Gamma\}$ of $\chi$, which is a subgroup of the right semi-direct product $H=\Gamma\rtimes_\beta G$ (since $\chi$ is a crossed homomorphism).

Observe that $\beta$ still defines a continuous right action of $G$ on the countable set $\Gamma_\chi$ and for $g\in G$, $\beta_g$ is a group homomorphism of $\Gamma_\chi$ if and only if $g^{-1}\chi(\gamma)^{-1}g\chi(\beta_g(\gamma))\in{\rm Ker}(\beta_g)$ for all $\gamma\in\Gamma$, which happens for example if $\chi\circ\beta_g=g^{-1}\chi(\cdot)g$. Moreover, the formula $\alpha_\gamma(g)=\chi(\gamma)g\chi(\beta_g(\gamma))^{-1}$, for all $\gamma\in\Gamma,g\in G$, defines an action of $\Gamma_\chi$ on the compact space $G$ by homeomorphisms and in addition $\alpha$ and $\beta$ satisfy the relations in Equation $(\ref{EqMatched})$. Consequently, we get a new matched pair $(\Gamma_\chi,G)$ with possibly non-trivial actions $\alpha$ and $\beta$. As before, one can describe this new matched pair explicitly by viewing $\Gamma_\chi$ and $G$ as closed subgroups of the right semi-direct product $H=\Gamma\rtimes_\beta G$.

Observe that a group homomorphism $\chi\,:\,\Gamma\rightarrow G$ is a crossed homomorphism if and only if $\chi=\chi\circ\beta_g$ for all $g\in{\rm Im}(\chi)$.

\begin{remark}\label{RmkExBetaH}
Suppose that the crossed homomorphism satisfies $\chi\circ\beta_g=\chi$ for all $g\in{\rm Im}(\chi)$  and let $_{\chi}\Gq$ be the associated bicrossed product. Then the following are equivalent.
\begin{enumerate}
\item $\Gamma$ has the Haagerup property.
\item $\widehat{_{\chi}\Gq}$ has the Haagerup property.
\end{enumerate}
Indeed, by Corollary \ref{CorMatched}, it suffices to show that the action $\alpha$ of $\Gamma_\chi$ on $G$ is compact when viewed as an action of $\Gamma_\chi$ on ${\rm L}^\infty(G)$. Since $\alpha_\gamma(g)=\chi(\gamma)g\chi(\gamma)^{-1}$ for $g\in G$ and $\gamma\in\Gamma_{\chi}$, $\alpha$ is an action by inner automorphisms, thus it is always compact since it is trivial on ${\rm Irr}(G)$. Indeed, for any unitary representation $u$ of $G$, the unitary $u(\chi(\gamma))$ is an intertwiner between $\alpha_\gamma(u)$ and $u$ for $\gamma\in \Gamma_\chi$. 
\end{remark}
A systematic way to construct explicit examples using the deformation above is to consider any countable discrete group $\Gamma_0$ which has a finite non-abelian quotient $G$ and take $\Gamma=\Gamma_0\times G$ with the right action of $G$ on $\Gamma$ given by $\beta_g(\gamma,h)=(\gamma,g^{-1}hg)$, $g,h\in G$ and $\gamma\in \Gamma_0$. Since $G$ is non-abelian, $\beta$ is non-trivial. Let $q\,:\,\Gamma_0\rightarrow G$ be the quotient map and define the morphism $\chi\,:\,\Gamma\rightarrow G$ by $\chi(\gamma,h)=q(\gamma)$, $\gamma\in \Gamma_0$, $h\in G$. Then, we obviously have $\chi\circ\beta_g=\chi$ for all $g\in G$ . Therefore, $\chi$ is a crossed homomorphism and the action $\alpha$ of $\Gamma_\chi$ on $G$ is given by $\alpha_{(\gamma,h)}(g)=q(\gamma)gq(\gamma^{-1})$, $\gamma\in\Gamma_0$, $h,g\in G$, which is also non-trivial since $G$ is non-abelian. Thus $(\Gamma_\chi,G)$ is a compact matched pair. Let $_{\chi}\Gq$ denote the bicrossed product.

\begin{proposition}\label{PropExBeta}
We have $\chi(_{\chi}\Gq)\simeq Z(G)\times{\rm Sp}(\Gamma_0)\times{\rm Sp}(G)$ and ${\rm Int}(_{\chi}\Gq)={\rm Sp}(G)\times\Gamma_0\times Z(G)$.
\end{proposition}

\begin{proof}
Note that $\Gamma_\chi=\Gamma_0\times G$ as a set and the group law is given by $(r,g)(s,h)=(rs,q(s)^{-1}gq(s)h)$ for all $r,s\in\Gamma_0$ and $g,h\in G$. Since the action $\beta$ of $G$ on $\Gamma_\chi$ is given by $\beta_g(s,h)=(s,g^{-1}hg)$, $s\in \Gamma_0, g,h\in G$, we have $\Gamma_\chi^\beta=\Gamma_0\times Z(G)$ and the action of $Z(G)$ on $\Gamma_\chi$ is trivial. 
Since the action $\alpha$ of $\Gamma_\chi^\beta$ on $G$ is given by $\alpha_{(r,g)}(h)=q(r)hq(r)^{-1}$, $r\in \Gamma_0, g,h\in G$, we find $G^\alpha=Z(G)$. Again, since the action $\alpha$ is by inner automorphisms, the associated action on ${\rm Sp}(G)$ is trivial. It follows from Proposition \ref{PropInt} that $\chi(_{\chi}\Gq)\simeq Z(G)\times{\rm Sp}(\Gamma_\chi)$ and ${\rm Int}(_{\chi}\Gq)={\rm Sp}(G)\times\Gamma_0\times Z(G)$. Let $\iota_{\Gamma_0}\,:\,\Gamma_0\rightarrow\Gamma_\chi$, $r\mapsto(r,1)$ and $\iota_G\,:\,G\rightarrow\Gamma_\chi$, $g\mapsto(1,g)$. Observe that $\iota_{\Gamma_0}$ and $\iota_G$ are group homomorphisms. To finish the proof, we claim that the map $\psi\,:\,{\rm Sp}(\Gamma_\chi)\rightarrow{\rm Sp}(\Gamma_0)\times{\rm Sp}(G)$, defined by $\omega\mapsto(\omega\circ\iota_{\Gamma_0},\omega\circ\iota_G)$, $\omega\in {\rm Sp}(\Gamma_\chi)$, is a group isomorphism. Indeed, it is obviously a group homomorphism. Since $\Gamma_\chi$ is generated by $\iota_{\Gamma_0}(\Gamma_0)$ and $\iota_G(G)$, so $\psi$ is injective. Let $\omega_1\in{\rm Sp}(\Gamma_0)$ and $\omega_2\in{\rm Sp}(G)$. Define the continuous map $\omega\,:\,\Gamma_\chi\rightarrow S^1$ by $\omega(r,g)=\omega_1(r)\omega_2(g)$, $r\in \Gamma_0, g\in G$. Then, for all $r,s\in\Gamma_0$, $g,h\in G$,
\begin{eqnarray*}
\omega((r,g)\cdot(s,h))&=&\omega(rs,q(s)^{-1}gq(s)h)=\omega_1(r)\omega_1(s)\omega_2(q(s)^{-1})\omega_2(g)\omega_2(q(s))\omega_2(h)\\
&=&\omega_1(r)\omega_2(g)\omega_1(s)\omega_2(h)=\omega(r,g)\omega(s,h).
\end{eqnarray*}
Hence, $\omega\in{\rm Sp}(\Gamma_\chi)$ and $\psi(\omega)=(\omega_1,\omega_2)$, so $\psi$ is surjective.
\end{proof}

\begin{example} (\textbf{Haagerup Property})
Observe that any finite non-abelian group $G$ provides an example with $\Gamma_0=\mathbb{F}_n$, where $n$ is bigger than the number of generators of $G$, so that $G$ is a quotient of $\Gamma_0$ in the obvious way. All bicrossed products obtained in this way are not co-amenable but their duals do have the Haagerup property by Remark \ref{RmkExBetaH}.

To get explicit examples we take, for $n\geq 4$, $G=A_n$ the alternating group which is simple, has only one irreducible representation of dimension $1$ (the trivial representation) so that $Z(G)=\{1\}$ and ${\rm Sp}(G)=\{1\}$. Moreover, viewing $A_n$ generated by the $n-2$ $3$-cycles, we have a surjection $\Gamma_0=\mathbb{F}_{n-2}\rightarrow A_n=G$. Associated to this data, we get a non-trivial compact bicrossed product $\Gq_n$ non co-amenable and whose dual has the Haagerup property and such that $\chi(\Gq_n)\simeq{\rm Sp}(\mathbb{F}_{n-2})=\mathbb{T}^{n-2}$. In particular $\Gq_n$ and $\Gq_m$ are not isomorphic for $n\neq m$. It shows the existence of an infinite family of pairwise non-isomorphic non-trivial compact bicrossed product whose dual are non amenable with the Haagerup property.
\end{example}

We now consider more explicit examples on property $(T)$.

\begin{example}\label{ExMatchedT} (\textbf{Property $(T)$})
Let $n\geq 3$ be a natural number and $p\geq 3$ be a prime number. Let $\F_p$ denote the finite field of order $p$. Define $\Gamma_0={\rm SL}_n(\Z)$, $G={\rm SL}_n(\F_p)$ and let $q\,:\,{\rm SL}_n(\Z)\rightarrow {\rm SL}_n(\F_p)$ be the canonical quotient map. We get a matched pair $(\Gamma_\chi,G)$ with both actions $\alpha$ and $\beta$ non-trivial and we denote the bicrossed product by $\Gq_{n,p}$. Since for $n,p\geq 3$, we have $D({\rm SL}_n(\Z))={\rm SL}_n(\Z)$ and $D({\rm SL}_n(\F_p))={\rm SL}_n(\F_p)$, we deduce as in Example \ref{ExRelT} that ${\rm Sp}({\rm SL}_n(\F_p))=\{1\}={\rm Sp}({\rm SL}_n(\Z))$. It follows from Proposition \ref{PropExBeta} that
$${\rm Int}(\Gq_{n,p})\simeq{\rm SL}_n(\Z)\times Z({\rm SL}_n(\F_p))\simeq{\rm SL}_n(\Z)\times \Z/d\Z\text{ and }\chi(\Gq_{n,p})=Z({\rm SL}_n(\F_p))\simeq\Z/d\Z,$$
where $d={\rm gcd}(n,p-1)$. In particular, the quantum groups $\Gq_{p}=\Gq_{p,p}$ for $p$ prime and $p\geq 3$, are pairwise non-isomorphic. They are non-commutative and non-cocommutative by Remark $\ref{RmkNonTrivialMatched}$. Moreover, assertion $2$ of Theorem \ref{ThmPropT} implies that $\widehat{\Gq}_{p}$ have property $(T)$. We record this in the form of a theorem.

\begin{theorem}\label{InfiniteExample}
There exists an infinite family of pairwise non isomorphic non-trivial compact bicrossed products whose duals have property $(T)$.
\end{theorem}
These are the first explicit examples of non-trivial discrete quantum groups with property $(T)$.

One can also consider a similar but easier family of examples with $\beta$ being trivial. We still take a natural number $n\geq 3$ and a prime number $p\geq 3$.  But we consider $\Gamma={\rm SL}_n(\Z)$ and $G={\rm SL}_n(\F_p)$ with the action $\alpha$ being given by $\alpha_\gamma(g)=[\gamma]g[\gamma]^{-1}$, $\gamma\in \Gamma, g\in G$, and $\beta$ being the trivial action. Let $\mathbb{H}_{n,p}$ denote the bicrossed product associated to the matched pair $(\Gamma,G)$. One can check, as before, that ${\rm Int}(\mathbb{H}_{n,p})\simeq{\rm SL}_n(\Z)$ and ${\rm Sp}(C_m(\mathbb{H}_{n,p}))\simeq\Z/d\Z$, where $d={\rm gcd}(n,p-1)$. Hence, the quantum groups $\mathbb{H}_{p}=\mathbb{H}_{p,p}$ for $p$ prime and $p\geq 3$, are pairwise non-isomorphic. They arise from  matched pairs for which the $\beta$ action is trivial but still they are non-commutative and non-cocommutative since $\Gamma$ and $G$ are both non-abelian. Also, their duals have property $(T)$.
\end{example}

%%%%%%%%%%%%%%%%%%%%%%%%%%%%%%%%%%%%%%%%%%%%%%%
\subsection{Examples of crossed products}\label{SectionExCrossed}
%%%%%%%%%%%%%%%%%%%%%%%%%%%%%%%%%%%%%%%%%%%%%%%

In this section, we provide non-trivial examples of crossed products. Our examples are of the type considered in \cite{Wa95b}. Let $G$ be a compact quantum group and define, for all $g\in \chi(G)$, the map $\alpha_g=(g^{-1}\ot\id\ot g)\circ\Delta^{(2)}$. It defines a continuous group homomorphism $\chi(G)\ni g\mapsto \alpha_g\in {\rm Aut}(G)$. Since $\chi(G)$ is compact, it follows that the action $\Gamma\curvearrowright G$ is always compact, for any countable subgroup $\Gamma<\chi(G)$. Actually, the action of $\chi(G)$ on $\text{Irr}(G)$ is trivial since, for $g\in\chi(G)$ and $x\in{\rm Irr}(G)$ a straightforward computation gives $(\id\ot\alpha_g)(u^x)=(V_g^*\ot 1)u^x(V_g\ot 1)$,  where $V_g=(\id\ot g)(u^x)$. Let $G_\Gamma$ denote the crossed product. For a subgroup $\Sigma<H$, we denote by $C_H(\Sigma)$ the centralizer of $\Sigma$ in $H$. Applying our results on crossed products to $G_\Gamma$ we get the following Corollary.

\begin{corollary}\label{CorExCrossed}
The following holds.
\begin{enumerate}
\item ${\rm Int}(G_\Gamma)\simeq {\rm Int}(G)\times \Gamma$ and $\chi(G_\Gamma)\simeq C_{\chi(G)}(\Gamma)\times{\rm Sp}(\Gamma)$.
\item ${\rm max}(\Lambda_{{\rm cb}}(C(G)),\Lambda_{{\rm cb}}(\Gamma))\leq\Lambda_{{\rm cb}}(C(G_\Gamma))\leq\Lambda_{{\rm cb}}(\widehat{G})\Lambda_{{\rm cb}}(\Gamma)$.
\item $\widehat{G}$ and $\Gamma$ have $(RD)$ if and only if $\widehat{G}_\Gamma$ has $(RD)$.
\item $\widehat{G_\Gamma}$ has the Haagerup property if and only if $\widehat{G}$ and $\Gamma$ have the Haagerup property.
\item $\widehat{G}_\Gamma$ has property $(T)$ if and only if $\widehat{G}$ and $\Gamma$ have property $(T)$.
\end{enumerate}
\end{corollary}

\begin{proof}
All the statements directly follow from the results of section $6$ and the discussion preceding the statement of the Corollary except assertion $1$ for which there is something to check: the action of $\chi(G)$ on ${\rm Int}(G)$ associated to the action $\alpha$ is trivial indeed, for all unitary $u\in C_m(G)$ for which $\Delta(u)=u\ot u$, one has $\alpha_g(u)=g(u)ug(u^*)=u$. Moreover, the action of $\chi(G)$ on itself associated to the action $\alpha$ is, by definition, the action by conjugation. Hence assertion $1$ directly follows from Proposition \ref{Prop-CrossedInt}.
\end{proof}

\begin{example}\label{ExCrossedFree}
We consider examples with $G=U_N^+$, the free unitary quantum group or $G=O_N^+$, the free orthogonal quantum group. It is well known that $\chi(U_N^+)=U(N)$ and $\chi(O_N^+)=O(N)$ and that ${\rm Int}(U_N^+)={\rm Int}(O_N)^+=\{1\}$. It is also known that the Cowling-Haagerup constant for $O_N^+$ and $U_N^+$ are both $1$ \cite{Fr13}, and $\widehat{O_N^+}$ and $\widehat{U_N^+}$ have (RD) \cite{Ve07} and the Haagerup property \cite{Br12}. Hence, for any $N\geq 2$ and any subgroups $\Sigma<O(N)$ and $\Gamma<U(N)$ the following holds.

\begin{itemize}
\item ${\rm Int}((O_N^+)_\Sigma)\simeq\Sigma$ and ${\rm Int}((U_N^+)_\Gamma)\simeq\Gamma$.
\item $\chi((O_N^+)_\Sigma)\simeq C_{O(N)}(\Sigma)\times{\rm Sp}(\Sigma)$ and $\chi((U_N^+)_\Gamma)\simeq C_{U(N)}(\Gamma)\times{\rm Sp}(\Gamma)$.
\item $\Lambda_{{\rm cb}}(\widehat{(O_N^+)_\Sigma})=\Lambda_{{\rm cb}}(\Sigma)$ and $\Lambda_{{\rm cb}}(\widehat{(U_N^+)_\Gamma})=\Lambda_{{\rm cb}}(\Gamma)$.
\item $\widehat{(O_N^+)_\Sigma}$ (resp. $(\widehat{(U_N^+)_\Gamma}$) has $(RD)$ if and only if $\Sigma$ (resp. $\Gamma$) has $(RD)$.
\item $\widehat{(O_N^+)_\Sigma}$ (resp. $(\widehat{(U_N^+)_\Gamma}$) has the Haagerup property if and only if $\Sigma$ (resp. $\Gamma$) has the Haagerup property.
\item $\widehat{(O_N^+)_\Sigma}$ and $(\widehat{(U_N^+)_\Gamma}$ do not have Property $(T)$.
\end{itemize}
\end{example}

\begin{example}(\textbf{Relative Haagerup Property})\label{ExCrossedRelH}
Since the action of $\chi(G)$ on $C_m(G)$ is given by $(\id\ot\alpha_g)(u^x)=(V_g^*\ot 1)u^x(V_g\ot 1)$,  where $V_g=(\id\ot g)(u^x)$, we have, 
\begin{equation}\label{EqActChi}
\alpha_g(\omega)(u^x_{ij})=\sum_{r,s}g(u^x_{ir})\omega(u^x_{rs})g((u^x_{js})^*), \text{ for all } \omega\in C_m(G)^*.
\end{equation}
Define the sequence of dilated Chebyshev polynomials of second kind by the initial conditions $P_0(X)=1$, $P_1(X)=X$ and the recursion relation $XP_k(X)=P_{k+1}(X)+P_{k-1}(X)$, $k\geq 1$. It is proved in \cite{Br12} (see also \cite{FV14}) that the net of states $\omega_t\in C_m(O_N^+)^*$ defined by $\omega_t(u^k_{ij})=\frac{P_k(t)}{P_k(N)}\delta_{i,j}$, for $k\in{\rm Irr}(O_N^+)=\N$ and $t\in(0,1)$ realize the co-Haagerup property for $O_N^+$, i.e., $\widehat{\omega}_t\in c_0(\widehat{O_N^+})$ for $t$ close to $1$ and $\omega_t\rightarrow \varepsilon_{O_N^+}$ in the weak* topology when $t\rightarrow 1$. Now let $g\in\chi(O_N^+)$. By Equation $(\ref{EqActChi})$, we have $\alpha_g(\omega_t)(u^k_{ij})=\frac{P_k(t)}{P_k(N)}\sum_rg(u^k_{ir})g((u^k_{jr})^*)=\frac{P_k(t)}{P_k(N)}\delta_{i,j}=\omega_t(u^k_{ij})$. Hence, $\alpha_g(\omega_t)=\omega_t$ for all $g\in\chi(G)$ and all $t\in(0,1)$. It follows that for any $N\geq 2$ and any subgroup $\Gamma<O(N)$, the pair $(O_N^+,(O_N^+)_\Gamma)$ has the relative co-Haagerup property however, the dual of $(O_N^+)_\Gamma)$ does not have the Haagerup property whenever $\Gamma$ does not have the Haagerup property.
\end{example}

\bigskip

\footnotesize

\noindent 
{\sc Pierre FIMA \\ \nopagebreak
  Univ Paris Diderot, Sorbonne Paris Cit\'e, IMJ-PRG, UMR 7586, F-75013, Paris, France \\ 
  Sorbonne Universit\'es, UPMC Paris 06, UMR 7586, IMJ-PRG, F-75005, Paris, France \\
  CNRS, UMR 7586, IMJ-PRG, F-75005, Paris, France \\
   Department of Mathematics, IIT Madras, Chennai 600036, India \\
\em E-mail address: \tt pierre.fima@imj-prg.fr}

\bigskip

\noindent
{\sc Kunal MUKHERJEE \\
  Department of Mathematics, IIT Madras \\
  Chennai 600036, India \\
\em E-mail address: \tt kunal@iitm.ac.in}

\bigskip

\noindent
{\sc Issan PATRI \\
  Institute of Mathematical Sciences \\
  Chennai 600113, India \\
\em E-mail address: \tt issanp@imsc.res.in}

\end{document}